\documentclass[letterpaper,12pt]{iopart}
\usepackage{citesort}
\usepackage{url}

\usepackage{amsfonts} 
\usepackage{amssymb} 
\usepackage{amsmath} 
\usepackage{amsthm}
\usepackage{color}
\usepackage[pdftex]{graphicx} 
\usepackage{soul}
\usepackage{caption}
\usepackage{enumerate}

\newcommand{\Path }{}
\newcommand{\Ndata}{N}

\newcommand{\dimObs}{d}

\newcommand{\NEigen}{m}

\newcommand{\Disc}{\mathcal{D}}
\newcommand{\InfDelay}{\Psi}

\newcommand{\Range}{H}
\newcommand{\basin}{{\mathcal{B}_{\mu}}}
\newcommand{\nbd}{\mathcal{U}}

\newcommand{\pow}{r}

\newcommand{\real}{\mathbb{R}}
\newcommand{\num}{\mathbb{N}}
\newcommand{\integer}{\mathbb{Z}}
\newcommand{\cmplx}{\mathbb{C}}
\newcommand{\TorusD}[1]{\mathbb{T}^{#1}}

\newcommand{\blue}{\textcolor{black}}

\DeclareMathOperator{\Leb}{Leb}
\DeclareMathOperator{\spn}{span}
\DeclareMathOperator{\spec}{spec}
\DeclareMathOperator{\ran}{ran}
\DeclareMathOperator{\Real}{Re}
\DeclareMathOperator{\Imag}{Im}
\DeclareMathOperator{\grad}{grad}
\DeclareMathOperator{\vol}{vol}

\newcounter{enum_sav}

\newtheorem{thm}{Theorem}
\newtheorem{lem}[thm]{Lemma}
\newtheorem{prop}[thm]{Proposition}
\newtheorem{defn}[thm]{Definition}
\newtheorem{cor}[thm]{Corollary}

\newtheorem{Assumption}{Assumption}
\newtheorem*{rk*}{Remark}{\bf}{\rm}

\begin{document}
\title{Delay-coordinate maps and the spectra of Koopman operators}
\author{Suddhasattwa Das}
\address{Courant Institute of Mathematical Sciences, New York University, New York, New York, USA}
\ead{dass@cims.nyu.edu}

\author{Dimitrios Giannakis}
\address{Center for Atmosphere Ocean Science, Courant Institute of Mathematical Sciences, New York University, New York, New York, USA}
\ead{dimitris@cims.nyu.edu}

\begin{abstract} The Koopman operator induced by a dynamical system is inherently linear and provides an alternate method of studying many properties of the system, including attractor reconstruction and forecasting. Koopman eigenfunctions represent the non-mixing component of the dynamics. They factor the dynamics, which can be chaotic, into quasiperiodic rotations on tori. Here, we describe a method through which these eigenfunctions can be obtained from a kernel integral operator, which also annihilates the continuous spectrum. We show that incorporating a large number of delay coordinates in constructing the kernel of that operator results, in the limit of infinitely many delays, in the creation of a map into the discrete spectrum subspace of the Koopman operator. This enables efficient approximation of Koopman eigenfunctions from high-dimensional data in systems with pure point or mixed spectra.

\end{abstract}

\section{Introduction}

The tasks of dimension reduction and forecasting of time series are very common in physical and engineering sciences, where the time-series studied are often partial observations of a nonlinear dynamical system. A classical example of such time series is data collected from the Earth's climate system, where many of the active degrees of freedom are difficult to access via direct observations (e.g., subsurface ocean circulation). Moreover, the available observations typically mix together different physical processes operating on a wide range of spatial and temporal scales. For instance, in the climate system, the seasonal cycle and the El Ni\~no Southern Oscillation (the latter evolving on interannual timescales) both have strong associated signals in sea surface temperature \cite{WangEtAl17}. In such applications, identifying dynamically important, coherent patterns of variability from the data can enhance our scientific understanding and predictive capabilities of complex phenomena.

Ergodic theory, and in particular its operator-theoretic formulation \cite{BudisicEtAl12,EisnerEtAl15}, provides a natural framework to address these objectives. In this framework, the focus is on the action of the dynamical system on spaces of observables (functions of the state), as opposed to the dynamical flow itself. The advantage of this approach, first realized in the seminal work of Koopman \cite{Koopman31}, is that the action of a general dynamical system on spaces of observables is always linear. As a result, with appropriate regularity assumptions, the problem of identification and prediction of dynamically intrinsic coherent patterns can be formulated as an estimation problem for the spectrum of a linear evolution operator. In addition, for systems exhibiting ergodic behavior, spectral quantities such as eigenvalues and eigenfunctions can be statistically estimated from time-ordered data without prior knowledge of the state space geometry or the equations of motion. At the same time, spaces of observables are also infinite dimensional, so the issue of finite-dimensional approximation of (potentially unbounded) operators becomes relevant. 

Starting from the techniques proposed in \cite{DellnitzJunge99,MezicBanaszuk04,Mezic05}, the operator-theoretic approach to ergodic theory has stimulated the development of a broad range of techniques for data-driven modeling of dynamical systems. These methods employ either the Koopman \cite{MezicBanaszuk04,Mezic05,RowleyEtAl09,Schmid10,BudisicEtAl12,TuEtAl14,WilliamsEtAl15,GiannakisEtAl15,BruntonEtAl17,ArbabiMezic16,Giannakis17,GiannakisDas_tracers,KordaMezic17,KordaEtAl17} or the Perron-Frobenius (transfer) operators \cite{DellnitzJunge99,DellnitzEtAl00,FroylandEtAl14b,FroylandEtAl14,FroylandPadberg09}, which are duals to one another in appropriate function spaces. The goal common to these techniques is to approximate spectral quantities for the operator in question, such as eigenvalues, eigenfunctions, and spectral projections, from measured values of observables along orbits of the dynamics. To that end, a diverse range of approaches has been employed, including state space partitions \cite{DellnitzJunge99,DellnitzJunge99,DellnitzEtAl00,FroylandEtAl14b,FroylandEtAl14,FroylandPadberg09}, harmonic averaging \cite{MezicBanaszuk04,Mezic05,DasGiannakis_RKHS_2018}, iterative methods \cite{Schmid10,RowleyEtAl09}, dictionary/basis representations \cite{TuEtAl14,WilliamsEtAl15,GiannakisEtAl15,Giannakis17,GiannakisDas_tracers,KordaMezic17}, delay-coordinate embeddings \cite{BruntonEtAl17,ArbabiMezic16,GiannakisEtAl15,Giannakis17}, and spectral-moment estimation \cite{KordaEtAl17}.

Compared to observables identified by eigendecomposition techniques based on kernel integral operators that do not depend on the dynamics (e.g., covariance \cite{AubryEtAl91,HolmesEtAl96} or heat operators \cite{BelkinNiyogi03,CoifmanLafon06,BerrySauer16b}, the latter of which have been popular in manifold learning applications), eigenfunctions of evolution operators are likely to offer higher physical interpretability and predictability, as they are determined from an operator intrinsic to the dynamical system. In particular, one of the key properties of Koopman or Perron-Frobenius eigenfunctions for ergodic dynamical systems is that they evolve periodically and with a single frequency (even if the underlying dynamical system is aperiodic), and thus have high predictability. This and a number of other attractive properties motivate the identification of such eigenfunctions of data. 

Yet, for systems of sufficient complexity, Koopman and Perron-Frobenius operators have significantly more complicated spectral behavior than kernel integral operators, generally exhibiting a continuous spectral component and/or non-isolated eigenvalues, which presents challenges to the construction of data-driven approximation techniques with spectral convergence guarantees. Indeed, to our knowledge, spectral convergence results for the data-driven approximation of Koopman eigenvalues and eigenfunctions have been limited to special cases such as quasiperiodic rotations on tori \cite{Giannakis17}, or systems observed through measurement functions lying in finite-dimensional invariant subspaces \cite{ArbabiMezic16}. 

The main contribution of our work is the construction of a data-driven approximation scheme for Koopman eigenvalues and eigenfunctions that provably converges for a broad class of ergodic dynamical systems and observation maps, encompassing many of the applications encountered in the physical and engineering sciences. Our approach will be based on a combination of ideas from delay-coordinate maps of dynamical systems \cite{SauerEtAl91}, kernel integral operators for machine learning \cite{BelkinNiyogi03,CoifmanLafon06,BerrySauer16b,BerryHarlim16,VonLuxburgEtAl08}, and Galerkin approximation techniques for variational eigenvalue problems \cite{BabuskaOsborn91}. Using these tools, we will construct a compact kernel integral operator that commutes with the Koopman operator in an asymptotic limit of infinitely many delays, and employ the finite-dimensional common eigenspaces of these operators as Galerkin approximation spaces for the Koopman eigenvalue problem. We will show that orthonormal bases of these spaces can be stably and efficiently approximated from finitely many measurements taken near the attractor, and the resulting data-driven Galerkin schemes converge in the asymptotic limit of large data. 

\section{\label{secMainResults}Assumptions and statement of main results} 

A common underlying assumption \blue{in the statistical modeling of dynamical systems is ergodicity. This assumption encapsulates} the working principle that the global properties (with respect to $\mu$) of an observable $F$ can be obtained from a time series for $F$, namely, $F(x_0),\ldots,F(x_{\Ndata-1})$, where $x_0,\ldots,x_{\Ndata-1}$ is an unobserved trajectory on the state space of the dynamical system. Moreover, ergodicity implies that $L^2$ inner products between observables can be approximated by time-correlations. Also, our methods rely on integral operators, and these can be approximated as matrices under the ergodic hypothesis. We now make our assumptions more precise. 

\begin{Assumption}\label{asmptn:standing}
Let $M$ be a \blue{topological manifold}, equipped with its Borel $\sigma$-algebra. $\Phi^t:M\to M$, $ t \in \real$, is a \blue{continuous} flow on $M$ with an ergodic, Borel probability measure $\mu$ with a compact support $X$. $ F:M\to\real^\dimObs$ is a continuous measurement function through which we collect a time-ordered data set consisting of $ N $ samples $ F(x_0), F( x_1 ), \ldots, F( x_{N-1} )$, each $F(x_n)$ lying in $\dimObs$-dimensional data space. Here, $ x_n = \Phi^{n \, \Delta t}( x_0 ) $, and $ \Delta t $ is a fixed sampling interval such that the map $\Phi^{\Delta t}$ is ergodic for the invariant measure $ \mu $.
\end{Assumption}

\paragraph{The Koopman operator.} Central to all our following discussions will be the concept of the Koopman operator. Koopman operators \cite{BudisicEtAl12,EisnerEtAl15,Nadkarni} act on observables by composition with the flow map, i.e., by time shifts. The space $L^2(X,\mu)$ of square-integrable, complex-valued functions on $X$ will be our space of observables. Given an observable $f \in L^2( X, \mu)$ and time $t\in\real$, $U^t:L^2(X,\mu) \to L^2(X,\mu) $ is the operator defined as
\[(U^tf):x\mapsto f\left(\Phi^t(x)\right), \quad \text{for $\mu$-a.e.\ $x \in X$}.\]
$U^t$ is called the Koopman operator at time $ t $ associated with the flow. For measure-preserving systems, $ U^t $ is unitary, and has a well-defined spectral expansion consisting in general of both point and continuous parts lying in the unit circle \cite{Mezic05}. The problems of mode decomposition and non-parametric prediction can both be stated in terms of the Koopman operator \cite{Giannakis17}. 
We will now describe an important tool for studying Koopman operators, namely their eigenfunctions. 

\paragraph{Koopman eigenfunctions.} Every eigenfunction $z$ of $U^t$ satisfies the following equation for some $\omega\in\real$:
\begin{equation}\label{eqn:Def_koop_eigen}
U^tz=\exp(i\omega t)z.
\end{equation}
Koopman eigenfunctions are particularly useful for prediction and dimension reduction in dynamical systems. This is because, as seen in \eqref{eqn:Def_koop_eigen}, the knowledge of an eigenfunction $z$ at time $t=0$ enables accurate predictions of $z$ up to any time $t$, since $U^t$ operates on $z$ as a multiplication operator by a time-periodic, single-frequency multiplication factor. Moreover, it is possible to construct a dimension reduction map, sending the high-dimensional data $ F( x ) \in \real^d $ to the vector $ ( z_1( x ), \ldots, z_l( x ) ) \in \mathbb{ C }^l $, where $ l \ll d $, and the $ z_1, \ldots, z_l $ are Koopman eigenfunctions corresponding to rationally independent frequencies $ \omega_1, \ldots, \omega_l $ \cite{Mezic05,GiannakisEtAl15,Giannakis17}. In this representation, the $ z_j $ can be thought of as ``coordinates'' corresponding to distinct \blue{periodic} processes operating at the timescales $ 2\pi/\omega_j $. Also of interest (and in some cases easier to compute) are the projections of the observation map $ F $ onto the Koopman eigenfunctions, called Koopman modes \cite{Mezic05}. Data-driven techniques for computing Koopman eigenvalues, eigenfunctions, and modes that have been explored in the past include methods based on generalized Laplace analysis \cite{MezicBanaszuk04,Mezic05}, dynamic mode decomposition (DMD) \cite{SchmidSesterhenn08,RowleyEtAl09,Schmid10,TuEtAl14}, extended DMD (EDMD) \cite{WilliamsEtAl15,KordaMezic17}, Hankel matrix analysis \cite{TuEtAl14,ArbabiMezic16,BruntonEtAl17}, and data-driven Galerkin methods \cite{GiannakisEtAl15,Giannakis17,GiannakisDas_tracers}. The latter approach, as well as the related work in \cite{BerryEtAl15}, additionally address the problem of nonparametric prediction of observables and probability densities.

Let $\Disc$ be the closed subspace of $L^2(X,\mu)$ spanned by the eigenfunctions of $U^t$, and $\Disc^\bot$ its orthogonal complement. \blue{As is well known \cite{Halmos1956}, and will be discussed in more detail in Section~\ref{sect:theory}, the subspaces $\mathcal{D}$ and $\mathcal{D}^\perp$ represent the quasiperiodic and mixing (chaotic) components of the dynamics, respectively. Moreover, they are both invariant under $U^t$ for every time $t\in \real$, thus inducing an invariant splitting \cite{Mezic05}}
\begin{equation} \label{eqn:L2_decomp}
L^2(X,\mu)=\Disc\oplus\Disc^\bot.
\end{equation}
Systems for which $\Disc$ contains non-constant functions and $\Disc^\bot$ is non-empty are called mixed-spectrum systems. 

\paragraph{Kernel integral operators.} The method that we will describe in this paper relies heavily on kernel integral operators. A kernel is a function $k:M\times M\to\real$, measuring the similarity between pairs of points on $M$. Kernel functions can be of various designs, and are meant to capture the nonlinear geometric structures of data; see for example \cite{BelkinNiyogi03,Kernel1,CoifmanLafon06}. One advantage of using kernels is that they can be defined so as to operate directly on the data space, e.g., $ k( x, y ) = \kappa( F( x), F( y ) ) $ for some function $ \kappa : \real^d \times \real^d \to \real $ of appropriate regularity. Defined in this manner, $ k $ can be evaluated using measured quantities $ F( x )$ without explicit knowledge of the underlying state $ x $. Associated with a square-integrable kernel $ k\in L^2(X \times X, \mu \times \mu ) $ is a compact integral operator $K:L^2(X,\mu)\to L^2(X,\mu) $ such that 
\begin{equation} \label{eqn:KOp}
Kf(x) := \int_X k(x,y)f(y) \, d\mu(y).
\end{equation}
In some cases, we will make the following assumptions on kernels.
\begin{Assumption}\label{asmptn:ker}
 The kernel $ k : M \times M \to \real$ is (i) symmetric and continuous; (ii) strictly positive-valued.
\end{Assumption}

\paragraph{Overview of approach.} We will address the eigenvalue problem for $U^t$ by solving an eigenvalue problem for a kernel integral operator $P_{Q}$, which is accessible from data, and in the limit of $Q\to\infty$, commutes with $U^t$. Since commuting operators have common eigenspaces, this will allow us to compute eigenfunctions of $ U^t $ through expansions in eigenbases obtained from $ P_{Q} $. \blue{These operators have Markovian kernels $p_Q : M \times M \to \real$ (i.e., $ p_Q \geq 0 $ and $ \int_X p_Q( x, \cdot ) \, d\mu = 1 $, for $\mu $-a.e.\ $ x \in M$), whose construction begins from a family of distance-like functions $d_Q : M \times M \to \real$, defined by
\begin{equation}\label{eqn:def:dQ}
d^2_{Q}(x,y) =\frac{ 1 }{ Q } \sum_{q=0}^{Q-1} \left\lVert F(\Phi^{q\, \Delta t}(x)) - F(\Phi^{q\, \Delta t }(y)) \right\rVert^2.
\end{equation}
Here, $Q$ is a positive integer parameter, and $ \lVert \cdot \rVert $ the canonical 2-norm on $ \real^d $. Intuitively, $d_Q(x,y)$ assigns a distance-like quantity between points $x $ and $y$ equal to the root-mean square distance between $Q$ consecutive ``snapshots'' of the observable $F$, measured along dynamical trajectories starting from $ x $ and $y$. In other words, $d_Q $ corresponds to a distance between data in delay-coordinate space with $Q$ delays. Several of our results will depend on the asymptotic behavior of $d_Q$ as $ Q\to \infty$, which we will study in detail.}

\blue{Composing $d_Q$ with a continuous shape function $ h : \real \to \real $, leads to a kernel $k_Q : M \times M \to \real$, $ k_Q = h \circ d_Q $, assigning a pairwise measure of similarity between points in $M$. In this paper, we will nominally work with Gaussian shape functions, $ h(s) = e^{-s^2/\epsilon} $, parameterized by a bandwidth parameter $ \epsilon > 0 $, so that
\begin{equation}\label{eqn:def:kQ}
k_{Q}(x,y) =e^{-d^2_{Q}(x,y) /\epsilon}. 
\end{equation} 
Such kernels satisfy Assumption \ref{asmptn:ker}(i), (ii). They are popular in manifold learning applications \cite{BelkinNiyogi03,CoifmanLafon06,BerrySauer16b} due to their localizing behavior as $ \epsilon \to 0 $ and their ability to approximate heat kernels, but our results also hold for many other kernel choices; e.g., \cite{Genton01}. Having constructed $k_Q$, the kernel $p_{Q}$ associated with the integral operator $P_{Q}$ is obtained via a \emph{Markov normalization} procedure \cite{CoifmanLafon06,BerrySauer16b}, described in Section~\ref{secMarkov}. With these definitions, we are ready to state our main results.}

\begin{thm}\label{thm:A}
Under Assumption \ref{asmptn:standing}, there exists a real, self-adjoint, ergodic, compact Markov operator $P : L^2(X,\mu ) \to L^2( X, \mu ) $, which commutes with $U^t$, and is a limit of operators $ P_{1}, P_{2}, \ldots $ (also real, self-adjoint, ergodic, compact, and Markov) in the $L^2(X,\mu)$ operator-norm topology. The operators $P_{Q}$ have Markov kernels $ p_{Q} : M \times M \to \real $ satisfying the conditions in Assumption~\ref{asmptn:ker}, and determined from delay-coordinate mapped observations $F(x), F(\Phi^{\Delta t}(x)), \ldots, F( \Phi^{(Q-1)\Delta t}(x) ) $ with $ Q $ delays. Moreover, the kernel $ p : M \times M \to \real $ of $ P $ lies in $L^\infty(X\times X, \mu \times \mu )$, and $p_{Q}$ converges to $p$ in $L^p(X\times X,\mu \times \mu)$ norm with $1 \leq p < \infty $.
\end{thm}

\blue{The strong convergence of the compact operators $P_{Q}$ to $P$ leads to the following spectral convergence result (e.g., Section~7 in \cite{BabuskaOsborn91} and \cite{ALL2001}).}

\begin{cor}[spectral convergence] \label{corSpectral} Under the assumptions of Theorem~\ref{thm:A}, the following hold:
\begin{enumerate}[(i)]
\item For every nonzero eigenvalue $ \lambda $ of $ P $ with multiplicity $ \alpha $ and every neighborhood $ S \subset \real $ of $ \lambda $ such that $ \spec( P ) \cap S = \{ \lambda \} $, there exists $ Q_0 \in \num_0 $ such that for all $ Q > Q_0 $, $ \spec( P_{Q} ) \cap S $ contains $ \alpha $ elements converging as $ Q \to \infty $ to $ \lambda $. 
\item Let $ \Pi $ be any projector to the eigenspace $W_\lambda $ of $ P $ at eigenvalue $ \lambda $. Let also $ \Pi_Q $ be any projector to the union of the eigenspaces of $ P_{Q} $ corresponding to the eigenvalues in $ \spec( P_{Q} ) \cap S $. Then, as $ Q \to \infty $, $ \Pi_Q $ converges strongly to $ \Pi $. Moreover, the \blue{gap (distance)} between $ W_\lambda $ and $\ran \Pi_Q $, \blue{defined as in \cite{BabuskaOsborn91}}, 
converges to zero.
\end{enumerate}
\end{cor}

Theorem~\ref{thm:B} below is a continuation of Theorem~\ref{thm:A}, and can be used to conclude some useful properties of the operator $P$. 
\begin{thm}\label{thm:B}
\blue{Let $\Phi^t$ be a measurable flow on a compact set $X$ supporting an invariant ergodic probability measure $\mu$}, and $T$ be a kernel integral operator with a real-valued, symmetric kernel $ \tau\in L^2(X\times X,\mu\times\mu)$ such that $T$ commutes with $U^t$ (e.g., $T=P$ from Theorem~\ref{thm:A}). Then: 
\begin{enumerate}[(i)]
\item $\tau$ lies in the tensor product subspace $\Disc\otimes\Disc$, and is invariant under the flow $\Phi^t\times\Phi^t$. 
\item $ \Disc $ and $ \Disc^\perp $ are invariant under $T $. Moreover, $ \overline{\ran T} $ is a subspace of $ \Disc $, $ \Disc^\perp$ is a subspace of $\ker T $, and both $ \ran T $ and $ \ker T $ are invariant under $ U^t $.
\setcounter{enum_sav}{\value{enumi}}
\end{enumerate}
Moreover, if $ \ran T $ contains non-constant functions:
\begin{enumerate}[(i)]
\setcounter{enumi}{\value{enum_sav}}
\item There exists a measurable map $\pi:X\to \TorusD{D}$ for some $D\in\num$, whose components consist of joint eigenfunctions of $T$ and $U^t$, such that $\pi$ factors $\Phi^t$ into a rotation on the torus by a vector $\vec{\omega}\in\real^{D}$, i.e., $\pi(\Phi^t(x))=\pi(x)+\vec{\omega}t\bmod{1}$ for $ \mu$-a.e.\ $ x \in X$.
\item There exists a choice of dimension $ D $ from (iii) and a symmetric kernel $\hat{\tau}\in L^2(\TorusD{D}\times\TorusD{D},\Leb)$ on the $D$-torus, such that $\tau(x,y)=\hat \tau(\pi(x),\pi(y)) $ for $ \mu\times\mu$-a.e.\ $(x,y) \in X \times X $.
\end{enumerate}
\end{thm}

Note that Theorems \ref{thm:A} and \ref{thm:B} hold for operators acting on $L^2 $ spaces only. To be able to say more about the behavior of these operators on spaces of continuous functions, an additional assumption on the Koopman eigenfunctions and the observation map will be needed. \blue{In what follows, $F_D$ will denote the orthogonal projection of $F$ onto the quasiperiodic subspace $\Disc$ from~\eqref{eqn:L2_decomp}}. 
\begin{Assumption}\label{asmptn:C0_ae}
\blue{All Koopman eigenfunctions, as well as the quasiperiodic component of the observation map $F_{\Disc}$, are continuous.}
\end{Assumption}

\blue{Although we will explicitly assume that $F_{\Disc}$ is continuous, we are not aware of a counter-example where the observation map $F$ is continuous (in accordance with Assumption~\ref{asmptn:standing}), the Koopman eigenfunctions $z$ are continuous, but $F_{\Disc}$ is not continuous. On the other hand, smooth dynamical systems on smooth manifolds with discontinuous Koopman eigenfunctions (and in fact, pure point spectra) are known to exist, in both discrete- \cite{AnosovKatok1970} and continuous-time settings \cite{ConstantinEtAl2008}. This indicates that the continuity requirement on Koopman eigenfunctions in Assumption~\ref{asmptn:C0_ae} is complementary to the assumed continuity of the dynamical flow in Assumption~\ref{asmptn:standing}. The following theorem establishes a number of properties of $P$ under these additional continuity assumptions.} 
\begin{thm}\label{thm:D}
Let Assumptions \ref{asmptn:standing} and \ref{asmptn:C0_ae} hold. Then, the kernel $p$ of the operator $P$ from Theorem \ref{thm:A} is uniformly continuous on a full-measure, dense subset of $ X \times X $. As a result: 
\begin{enumerate}[(i)]
\item $ P $ maps $ L^2( X, \mu) $ into the space of $ \mu $-a.e.\ continuous functions on $ X$. 
\item $ P $ compactly maps $ C^0( X ) $ into itself. 
\item The norms of the operators $ P $ in (i) and (ii) are bounded above by $ \lVert p \rVert_{L^\infty(X \times X)} $.
\item For every $f\in C^0(X)$, $P_{Q} f$ is a sequence of continuous functions converging $\mu$-a.e.\ to $P f$.
\end{enumerate}
\end{thm}

\begin{rk*} \blue{The class of integral operators $P_Q$ studied in this work has previously been used for dimension reduction and mode decomposition of high-dimensional time series (e.g., \cite{GiannakisMajda11c,GiannakisMajda12a,BerryEtAl13,SlawinskaGiannakis17}). In these works, a phenomenon called in \cite{BerryEtAl13} ``timescale separation'' was observed; namely, it was observed that at increasingly large $ Q $ the eigenfunctions of $ P_{Q} $ capture increasingly distinct timescales of a multiscale input signal. Theorems~\ref{thm:A} and~\ref{thm:B} provide an interpretation of this observation from the point of view of spectral properties of Koopman operators; in particular, from the fact that $ P_{Q} $ has, in the limit $ Q \to \infty $, common eigenfunctions with $ U^t $ and the latter capture distinct timescales associated with the eigenfrequencies $ \omega $. 
Even though in this work we focus on the class of Markov operators $P_Q$, analogous results also hold for other classes of integral operators for data analysis that employ delays, including the covariance operators used in singular spectrum analysis (SSA) \cite{PackardEtAl80,BroomheadKing86,VautardGhil89} and the related Hankel matrix analysis \cite{TuEtAl14,BruntonEtAl17,ArbabiMezic16}. Collectively, these results establish a connection between two major branches of data analysis techniques for dynamical systems, namely those based on Koopman operators, and those based on kernel integral operators.} \end{rk*}

Theorems~\ref{thm:A}--\ref{thm:D} 
are proved in Section~\ref{sect:proof_main}. A result analogous to Theorem~\ref{thm:A}, but restricted to smooth manifolds, smooth observation maps, and Koopman operators with pure point spectrum and smooth eigenfunctions, was presented in \cite{Giannakis17}. Theorem~\ref{thm:A} generalizes this result to non-smooth state spaces and Koopman operators with mixed spectra. With this result, the eigenvalues and eigenfunctions of $P_{Q}$ consistently approximate those of $P$, and the latter can be used in turn to \blue{construct orthonormal bases of Koopman eigenspaces. The availability of such bases is useful in many applications, including approximation techniques for the eigenvalues and eigenfunctions of $U^t$ or its generator (defined in Section~\ref{sect:theory} ahead). One such technique will be presented in Section~\ref{sect:Galerkin}, utilizing the eigenvalues and eigenfunctions of $P$ to perform diffusion regularization of the generator, and then solving the eigenvalue problem for the generator via a Petrov-Galerkin method. Note that the Markov property of $P$ is not trivial; for instance, it does not hold for covariance kernels. The commutativity between $U^t$ and $P$, in conjunction with the Markov property, lead to well posedness of these schemes despite the presence of a continuous spectrum of $V$. }

\paragraph{Physical measures.} A point $x\in M$ is said to be in the basin of the measure $\mu$ with respect to the discrete-time map $\Phi^{\Delta t}$ if
\begin{equation}\label{eqn:Phys1}
\lim_{N\to\infty} \frac{ 1 }{ N } \sum_{n=0}^{N-1} f(\Phi^{n\, \Delta t}( x)) = \int_X f(y) \, d\mu(y), \quad \quad \forall f\in C^0(M).
\end{equation}
The basin $\basin$ of an invariant ergodic measure $\mu$ always includes $\mu$-a.e.\ point in the support of $\mu$ (in this case, $X$), and is a forward-invariant set. An important property that we need the invariant measure $\mu$ to have is that it is physical \cite{SRB_young}. Moreover, we will require that the dynamics has a suitable absorbing ball property. These assumptions can be summarized as follows:

\begin{Assumption}\label{asmptn:Phys2}
The set $\basin$ of points satisfying \eqref{eqn:Phys1} has positive Lebesgue measure, i.e., the measure $ \mu $ is physical. Moreover, there exists a subset $\mathcal{V} \subseteq \mathcal{B}_\mu$, also of positive Lebesgue measure, such that for every $ x_0 \in \mathcal{V} $ there exists a compact set $ \nbd $ (which may depend on $x_0$, and necessarily includes $X$), such that the orbit $ x_n = \Phi^{n \, \Delta t}(x_0) $ enters $\nbd$ and never leaves it.
\end{Assumption}
Examples where Assumption~\ref{asmptn:Phys2} is satisfied include: (i) ergodic flows on compact manifolds \blue{with Lebesgue absolutely continuous, fully supported, invariant measures}, in which case $ \mathcal{U } = \mathcal{V} =\overline{\mathcal{B}_\mu} = M = X $; (ii) certain classes of dissipative flows on potentially noncompact manifolds (e.g., the Lorenz 63 (L63) system on $ M = \real^3 $ \cite{Lorenz63} studied in Section~\ref{sect:examples} ahead); and (iii) certain classes of dissipative partial differential equations possessing inertial manifolds and physical measures \cite{LuEtAl13,LianEtAl16}.

The following result shows that under Assumptions~\ref{asmptn:standing}--\ref{asmptn:Phys2}, the nonzero eigenvalues of $ P_{Q} $ and the corresponding (continuous) eigenfunctions can be approximated to any degree of accuracy by data-driven operators $P_{Q,\Ndata}$, acting on the finite-dimensional Hilbert space $L^2(\nbd,\mu_N)$ associated with the sampling probability measure $\mu_N = \sum_{n=0}^{N-1} \delta_{x_n} / N$. \blue{These operators are constructed from time-ordered measurements $F(x_0),\ldots,F(x_{\Ndata-1})$ of the observable $ F$ analogously to~\eqref{eqn:KOp}--\eqref{eqn:def:kQ}, replacing throughout integrals with respect to the invariant measure $\mu$ by integrals with respect to the sampling measure $\mu_N$. Moreover, because $P_Q$ and $P_{Q,N}$ act on different Hilbert spaces, we will approach the problem of comparing their eigenvalues and eigenfunctions through integral operators $P''_Q: C^0(\mathcal{U}) \to C^0(\mathcal{U})$ and $P''_{Q,N}: C^0(\mathcal{U}) \to C^0(\mathcal{U})$, defined analogously to $P_Q$ and $P_{Q,N}$, respectively, but acting on the same Banach space of continuous functions on $\mathcal{U}$. A complete description of these constructions will be made in Section~\ref{sect:numerics}.}

\begin{thm}\label{thm:E}
Let Assumptions~1--4 hold. Then, for any initial point $ x_0 \in \mathcal{V} $, 
\begin{enumerate}[(i)]
\item \blue{Every eigenfunction of $P_Q$ ($P_{Q,N}$) at nonzero eigenvalue extends to a continuous eigenfunction of $P''_Q$ ($P''_{Q,N}$), corresponding to the same eigenvalue.} 
\item As $ N \to \infty $, $ P''_{Q,\Ndata} $ converges in spectrum to $P''_{Q}$ in the sense of Corollary~\ref{corSpectral}. 
\end{enumerate}
\end{thm}

Theorem~\ref{thm:E} will be proved in Section~\ref{sect:numerics}. Figure \ref{fig:4DTorus_additive} shows numerical eigenfunctions of $P_{Q,N}$ obtained from data generated by two mixed-spectrum dynamical systems, described in~\eqref{eqn:4D_additive} and~\eqref{eqn:l63_skew_nonlinear}, respectively. In both examples, we start with a $C^\infty$ vector field $\vec V$ on a smooth manifold $M$. In the first example, $M=X=\TorusD{4}$, so $\nbd=X=M$; in the second example, $M=\real^3\times S^1$ and $X= X_\text{Lor} \times S^1 \subset M $, where $X_\text{Lor} $ is the Lorenz~63 attractor embedded in $\real^3$. Eigenfunctions of the operator $P_{Q,N}$ are then computed using a large number of delays, $Q = 2000$. 

\begin{figure}
\centering
\includegraphics[width=.44\textwidth]{\Path 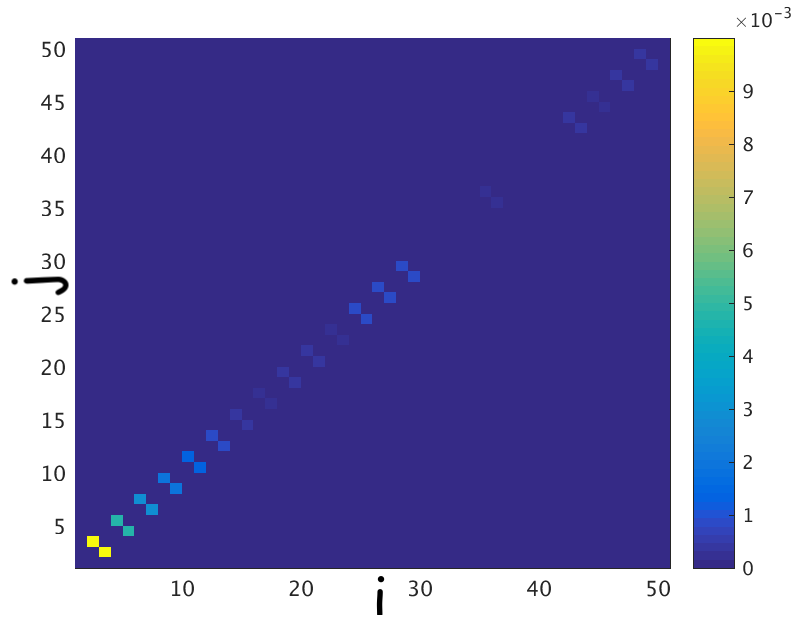}
\includegraphics[width=.54\textwidth]{\Path 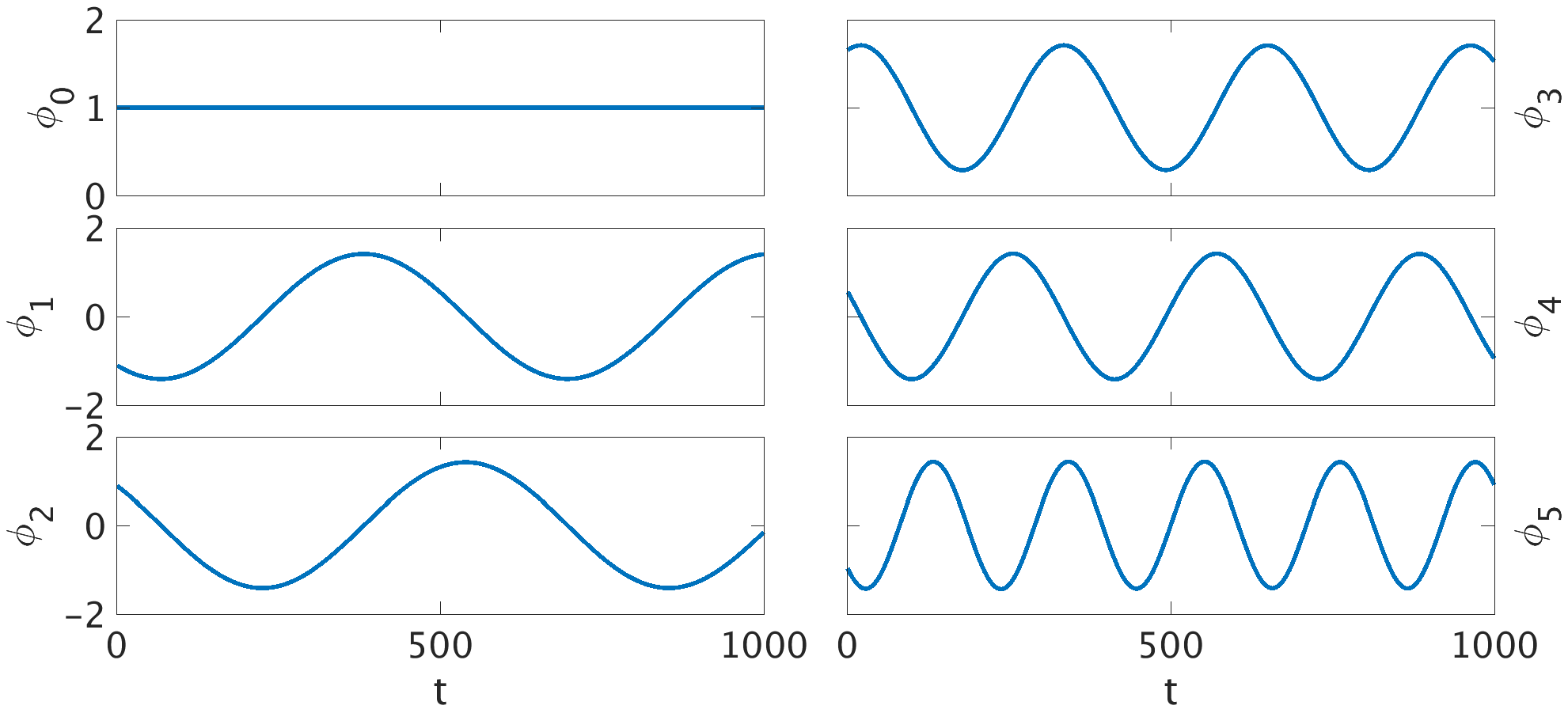}
\includegraphics[width=.44\textwidth]{\Path 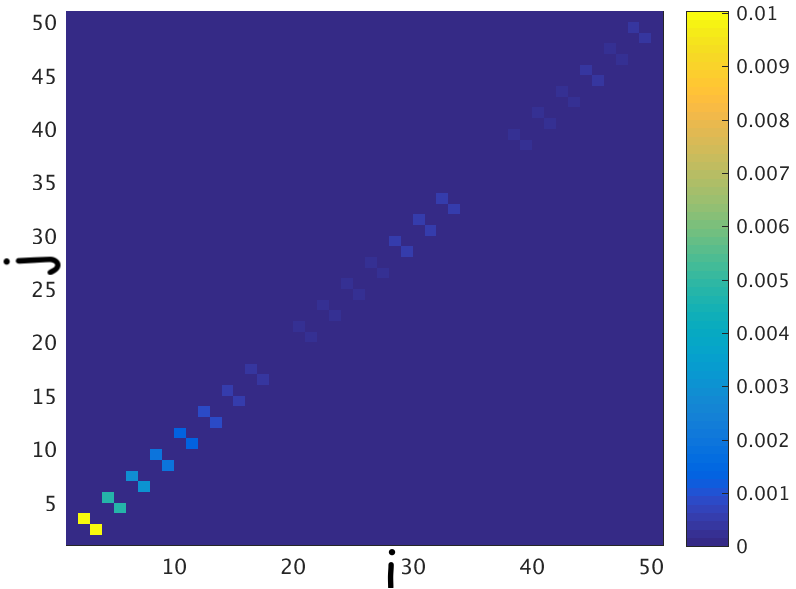}
\includegraphics[width=.54\textwidth]{\Path 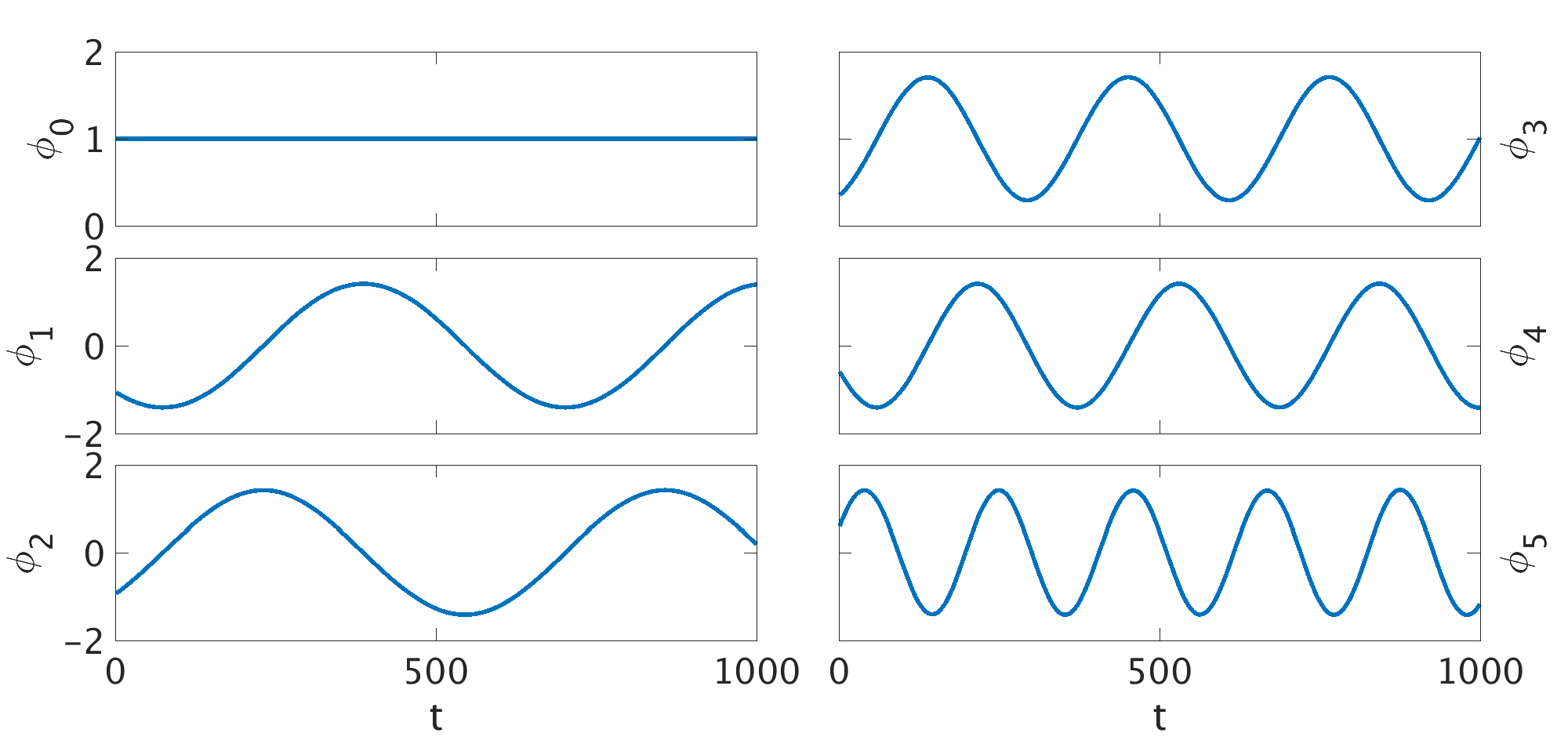}
\caption[results]{Representative eigenfunctions of $P_{Q,N}$ and the associated matrix representation of the generator $ V $ for the torus-based flow $\Phi^t_{\mathbb{T}^3} \times \Phi_\omega^t$ in \eqref{eqn:4D_additive} (top panels) and the L63-based flow $\Phi^t_\text{Lor}\times \Phi_\omega^t$ in \eqref{eqn:l63_skew_nonlinear} (bottom panels). The eigenfunctions $\phi_i$ have been computed using a large number of delays, $Q=2000$, and plotted as a time series along an orbit. These time series are near-sinusoidal, with frequencies close to integer multiples of the rotation frequency $ \omega $. \blue{Moreover, each frequency has multiplicity 2, and the corresponding time series are phase-shifted by $\pi/2$.} The left-hand panels show the absolute values $ |V_{ij}| = |\langle \phi_i, V \phi_j \rangle | $ of the matrix representation of the generator in the $ \{ \phi_i \} $ basis. Note that $ V_{ji} \approx - V_{ij} $, which is consistent with the fact that $V$ is a skew-adjoint operator. Since the first eigenfunction $\phi_0$ of $P_{Q,N}$ is the constant function and $V\phi_0=0$, the first column and row only have zero entries. \blue{Together, the $2\times 2$ block-diagonal form of the matrix representations of $V$ and the structure of the eigenfunction time series indicate that each of the pairs $(\phi_1,\phi_2), (\phi_3,\phi_4), \ldots$ spans an eigenspace of $V$, which is consistent with Theorems~\ref{thm:A}, \ref{thm:B}--\ref{thm:E} and Corollary~\ref{corSpectral}.}}
\label{fig:4DTorus_additive}
\end{figure}

Using the eigenvalues and eigenfunctions of $P_{Q,N}$, \blue{we will also construct data-driven Galerkin schemes for the eigenvalue problem of the generator, which are structurally identical to its counterparts formulated in terms of the eigenvalues and eigenfunctions of $P$. Because we do not assume a priori knowledge of the vector field of the dynamics and/or closed-form expressions for the eigenfunctions of $P_{Q,N}$, these schemes will estimate the action of the generator on eigenfunctions through finite-difference approximations at the sampling interval $\Delta t $. In effect, $\Delta t$ will play the role of an additional asymptotic approximation parameter, such that the data-driven solutions converge in a suitable joint limit of vanishing sampling interval ($\Delta t \to 0$), large data ($N\to\infty$), infinitely many delays ($Q\to\infty$), and infinite Galerkin approximation space dimension. This convergence result, \blue{along with minimal regularity requirements on the dynamical flow and the kernel, will be stated in a precise manner in Proposition~\ref{prop:Fin_diff} and Assumption~\ref{asmptn:Cr}, respectively}. Note that, intuitively, our data-driven Galerkin framework for the generator $V$ requires $\Delta t$ as an additional approximation parameter over methods that approximate the Koopman subgroup generated by $U^{\Delta t} $ at a fixed time step $ \Delta t $, since $V$ encodes the information of the entire Koopman group, parameterized by the real time parameter $t$.} 

\paragraph{Outline of the paper.} In Section~\ref{sect:theory}, we review some important concepts from the spectral theory of dynamical systems. In Section~\ref{sect:kernel}, we construct the integral operator $P_{Q}$, which is the key tool of our methods and is also the operator described in Theorems~\ref{thm:A}--\ref{thm:D}. Next, we prove these theorems and Corollary~\ref{corSpectral} in Section~\ref{sect:proof_main}. In Section~\ref{sect:Galerkin}, we present a Galerkin method for the eigenvalue problem for the Koopman generator, with a small amount of diffusion added for regularization, formulated in the eigenbasis of $ P $. In Section~\ref{sect:numerics}, we introduce the data-driven realization of $P_{Q}$, and establish the spectral convergence properties stated in Theorem~\ref{thm:E}, along with the convergence properties of the associated data-driven Galerkin scheme for the generator. In Section~\ref{sect:examples}, the methods are applied to two mixed-spectrum flows, followed by a discussion of the results. 

\section{Overview of spectral methods for dynamical systems}\label{sect:theory}

In this section, we review some concepts from the spectral theory of dynamical systems and establish some facts about Koopman eigenfunctions. Henceforth, we use the notations $ \langle f, g \rangle = \int_X f^* g \, d\mu $ and $ \lVert f \rVert = \langle f, f \rangle^{1/2} $ to represent the inner product and norm of $L^2( X,\mu) $, respectively. 

\paragraph{Generator of a flow.} By continuity of the flow $ \Phi^t $, the family of operators $U^t$ is a strongly continuous, 1-parameter group of unitary transformations of the Hilbert space $L^2(X,\mu)$. By Stone's theorem \cite{Stone1932}, any such family has a generator $V$, which is a skew-adjoint operator with a dense domain $ D( V ) \subset L^2(X,\mu)$, defined as 
\begin{equation}\label{eqn:def_gen_flow}
V f:=\lim_{t\to 0} \frac{ 1 }{ t } \left(U^t f - f\right), \quad f \in D( V ).
\end{equation}
The operators $U^t$ and $V$ share the same eigenfunctions; in particular, $z \in D(V)$ with $ U^t z = e^{i\omega t } z $ satisfies
\[Vz=i\omega z.\]
In light of \eqref{eqn:def_gen_flow} and the above relation, we can interpret the quantity $ \omega \in \real$ as a frequency intrinsic to the dynamical system (which we sometimes refer to as an ``eigenfrequency''). 

\paragraph{Vector fields as generators.} If we start with a vector field $\vec V$ on a \blue{$C^1$ manifold $M$}, then under appropriate regularity conditions (for example, $\vec V$ is locally Lipschitz continuous and satisfies suitable growth bounds at infinity), this vector field induces a $C^1 $ flow $\Phi^t : M \to M $ defined for all $ t \in \real $. Suppose that there is a compact invariant set $X\subseteq M$ with an ergodic invariant measure $\mu$. This set $X$ is not necessarily a submanifold, and may not even have any differentiability properties. Nevertheless, $(X,\Phi^t,\mu)$ is an ergodic dynamical system with an associated strongly-continuous, unitary group of Koopman operators $ U^t $. \blue{Acting on $ C^1(M) $ functions restricted to $ X $, the generator $ V $ of this group coincides with the vector field $\vec V$, the latter viewed as an operator $ \vec V : C^1(M) \to C^0(M) $.} For example, in quasiperiodic systems, $X=M=\TorusD{m}$, $ \vec V $ generates a rotation, and $\mu$ is equivalent to the Lebesgue volume measure. On the other hand, for the Lorenz attractor (see \eqref{eqn:l63}), $M=\real^3$, $ \vec V $ is smooth and dissipative, $X$ is a compact subset with non-integer fractal dimension \cite{LorentzFract}, and $\mu$ is supported on $X$.

\paragraph{Eigenfunctions as factor maps.} We state the following properties of a Koopman eigenfunction $ z $ of an ergodic dynamical system. 
\begin{enumerate}
\item If $ z $ corresponds to a nonzero eigenfrequency $ \omega $, then it has zero mean with respect to the invariant measure $\mu$. This can be concisely expressed as $\langle 1,z \rangle=0$.
\item The flow $\Phi^t$ is semi-conjugate to the irrational rotation by $\omega t$ on the unit circle, with $z$ acting as a semiconjugacy map. This follows directly from \eqref{eqn:Def_koop_eigen}. Since the eigenfunctions are $L^2$ equivalence classes, the semiconjugacy is measure-theoretic (holds $\mu$-a.e.), but would be $C^r$ if the eigenfunctions have a $C^r$ representation.
\item Normalized eigenfunctions with $ \lVert z \rVert = 1 $ have $\lvert z( x) \rvert = 1 $ for $ \mu $-a.e.\ $ x \in X $, by \eqref{eqn:Def_koop_eigen}. As a result, the map $z$ can now be viewed as a projection onto a circle in a measure-theoretic sense, $ z(x)\in S^1 $ for $ \mu $-a.e.\ $ x \in X $.
\end{enumerate}


\paragraph{Eigenfunctions form a group.} Another important property of Koopman eigenfunctions for ergodic dynamical systems is that they form a group under multiplication. That is, the product of two eigenfunctions of $U^t$ is again an eigenfunction, because of the following relation:
\begin{gather*}U^t z_i=\exp(it\omega_i) z_i, \quad \ i\in \{1,2 \}, \\
\implies U^t(z_1 z_2) = (U^t z_1)(U^t z_2) = \exp(it(\omega_1+\omega_2)) z_1 z_2.
\end{gather*}
Moreover, an analogous relation holds for the eigenfunctions and eigenvalues of $ V $. The fact that products of Koopman eigenfunctions are Koopman eigenfunctions leads to the  following result about products of elements of $\Disc$ with elements of $\Disc^\bot$. 
\begin{lem}\label{lem:cont_spect_ideal}
    Let $\Phi^t$ be an ergodic flow on a probability space $(X,\mu)$ such that $U^t$ has a mixed spectrum. Then, for every $ f\in\Disc$ and $ g\in\Disc^\perp$ for which $ fg \in L^2( X, \mu )$, $fg $ lies in $\Disc^\bot$.
\end{lem}

The eigenvalues of $V$ are closed under integer linear combinations. \blue{Moreover, if all corresponding eigenfunctions are continuous, they} are generated by a finite set of rationally independent eigenvalues $i\omega_1,\ldots,i\omega_{m}$. That is, every eigenvalue of $V$ is simple, and has the form $\omega_{\vec a} = \sum_{j=1}^m a_j \omega_j$ for some \blue{$ \vec a = (a_1,\ldots,a_m)\in\integer^m$}. Moreover, the corresponding eigenfunction is given by 
\begin{equation}\label{eqEigProd}
z_{\vec a} = \prod_{j=1}^m z_1^{a_1} \cdots z_m^{a_m}, 
\end{equation}
where $ z_j $ is the eigenfunction at eigenvalue $ i \omega_j $. By virtue of~\eqref{eqEigProd} the evolution of every observable $f \in \mathcal{D}$ under $U^t$ has the closed-form expression
    \begin{equation}
        \label{eqClosedForm}
        U^tf = \sum_{\vec a \in \mathbb{Z}^m} \hat f_{\vec a} e^{i \omega_{\vec a} t} z_{\vec a}, \quad \hat f_{\vec a} = \langle z_{\vec a}, f \rangle,
    \end{equation}
which can be evaluated given knowledge of finitely many generating eigenfunctions and eigenfrequencies. The following is a generalization of Property~2 of Koopman eigenfunctions listed above.
\begin{prop} \label{prop:semi_conj_pi}
Given an arbitrary collection $ \{ z_{{\vec a}_1}, z_{{\vec a}_2}, \ldots, z_{{\vec a}_l}\} $ of $l$ Koopman eigenfunctions, there exists a map $ \pi : X \to \mathbb{ C }^l $ with 
\begin{displaymath}
\pi(x ) = ( z_{{\vec a}_1}( x ), \ldots, z_{{\vec a}_l}( x ) ), \quad \text{for $\mu $-a.e.\ $ x \in X $}, 
\end{displaymath}
such that: 
\begin{enumerate}[(i)]
\item The image $ \pi( X ) $ is a torus of dimension $ D \leq \min\{ m, l \} $, with $ D = l $ if $ \omega_{{\vec a}_1}, \ldots, \omega_{{\vec a}_l} $ are rationally independent.
\item The flow $ (\Phi^t, \mu ) $ on $ X $ is semi-conjugate to an ergodic rotation $ (\Omega^t, \Leb ) $ on $ \mathbb{ T }^D $ (i.e., $ \pi \circ \Phi^t = \Omega^t \circ \pi $, $ \mu $-a.e.) associated with a frequency vector whose components are a subset of $ \{ \omega_{{\vec a_1}}, \ldots, \omega_{{\vec a_l}} \} $. 
\item Every Koopman eigenfunction $ z $ whose corresponding eigenfrequency is a linear combination of the $ \omega_{{\vec a}_1}, \ldots, \omega_{{\vec a}_l} $ satisfies $ z( x ) = \zeta(\pi(x ) ) $ for $ \mu $-a.e.\ $ x \in X $, where $\zeta \in C^\infty( \mathbb{ T}^D ) $ is a smooth Koopman eigenfunction of the ergodic rotation on the $ D $-torus corresponding to the same eigenfrequency.
\end{enumerate} 
\end{prop}

\begin{rk*}
If $m>1$, the set of eigenvalues $ \{ i\omega_{\vec a} \}_{\vec a \in \mathbb{ Z}^m }$ is dense on the imaginary axis. This property adversely affects the stability of numerical approximations of Koopman eigenvalues and eigenfunctions even in systems with pure point spectrum, necessitating the use of regularization \cite{Giannakis17}. We will return to this point in Section~\ref{sect:Galerkin}.
\end{rk*}

\begin{lem}[\cite{FourierDyn}, Section 2.3]\label{lem:erg_Fouri}
Let $ \Delta t > 0 $ be as in Assumption~\ref{asmptn:standing}. Then, the orthogonal projection $\pi_\omega f $ of \blue{an observable $f \in L^2(X,\mu)$} onto the eigenspace of $ U^{\Delta t}$ corresponding to the eigenvalue $ e^{i \omega\, \Delta t}$ of $U^{\Delta t}$ is given by
\begin{displaymath}
\pi_\omega f = \lim_{N\to\infty}\frac{1}{N}\sum_{n=0}^{N-1} e^{-i\omega n\, \Delta t} U^{n\, \Delta t}f. 
\end{displaymath}
Moreover, $\pi_\omega\equiv 0$ if \blue{$i\omega$ is not an eigenvalue of the generator}. Otherwise, $U^{\Delta t}\pi_\omega f=e^{i\omega \, \Delta t} \pi_\omega f$.
\end{lem}

\paragraph{Mixing and weak mixing.} An observable $f\in L^2(X,\mu)$ is said to be mixing if for all $ g\in L^2(X,\mu)$, $\lim_{t\to\infty} \langle g, U^t f \rangle=0$; it is said to be weak-mixing if $\lim_{t\to \infty} t^{-1} \int_0^T \lvert \langle g, U^s f \rangle \rvert \, ds = 0$. The latter, is equivalent to the requirement that for Lebesgue almost every $\Delta t\in\real$, $\lim_{N\to\infty}N^{-1}\sum_{n=0}^{N-1} \lvert \langle g, U^{n\, \Delta t}f \rangle \rvert = 0$. The flow $\Phi^t$ is said to be (weak-) mixing if $f$ is (weak-) mixing for all $ f\in L^2(X,\mu)$. It is known that every $f\in\Disc^\bot$ is weak mixing (see, e.g., Mixing Theorem, p.~45 in \cite{Halmos1956}), whereas no observable in $ \Disc $ is weak-mixing. Thus, the component $\Disc$, often called the quasiperiodic subspace, shows no decay of correlation, unlike its complement $\Disc^\bot$, which represents the chaotic component of the dynamics. In addition, weak-mixing observables in $ \Disc^\perp $ and observables in $ \Disc $ have a useful pointwise decorrelation property:
\begin{lem}\label{lem:Decor}
Let $ f \in \Disc^\perp $ and $ g \in \Disc $. Then, for $\mu$-a.e.\ $x,y\in X$, 
\[\lim_{N\to\infty} \frac{ 1 }{ N } \sum_{n=0}^{N-1} g^*( \Phi^{n \, \Delta t}(x ) ) f( \Phi^{n \, \Delta t}(y ) ) = 0. \]
\end{lem} 
\begin{proof} 
Without loss of generality, we may assume that $ g $ is an eigenfunction of $U^{\Delta t}$ with eigenvalue $e^{i \omega \, \Delta t}$. Then, 
\begin{multline*}
\lim_{N\to\infty} \frac{1}{N} \sum_{n=0}^{N-1} g^*( \Phi^{n \, \Delta t}(x )) f( \Phi^{n\, \Delta t}(y)) \\= g^*(x) \lim_{N\to\infty} \frac{1}{N}\sum_{n=0}^{N-1} e^{-in\omega \, \Delta t} f( \Phi^{n\, \Delta t}(y)),
\end{multline*}
which is equal to $g^*(x)\pi_{\omega} f (y)$ by Lemma~\ref{lem:erg_Fouri}. The latter is equal to zero since $f\in\Disc^\bot$. \qed
\end{proof}

\section{Kernel integral operators from delay-coordinate mapped data}\label{sect:kernel}

\subsection{\label{secKernelChoice}Choice of kernel}
Consider a kernel integral operator of the class~\eqref{eqn:KOp} associated with an $L^2$ kernel $k:M \times M\to\real$. Then, the following properties hold (e.g., \cite{EigenIntOp1,EigenIntOp2}):
\begin{enumerate}
\item $K$ is a Hilbert-Schmidt, and therefore compact, operator on $L^2( X, \mu ) $, with operator norm bounded by $ \lVert k \rVert_{L^2(X\times X)} $.
\item If $k$ is symmetric, then $K$ is self-adjoint. 
\item If $k$ is $C^0$, then $Kf$ is also $C^0$ for every $ f\in L^2(X,\mu)$.
\item If $M$ is a $C^r$ manifold and $k$ is $C^r$, then $Kf$ is also $C^r$ for every $f\in L^2(X,\mu)$.
\end{enumerate}

As stated in Section~\ref{secMainResults}, we will work with kernels of the form 
\begin{equation}\label{eqKQ2}
k_Q( x, y ) = h( d_Q( x , y ) ), 
\end{equation}
where $ h $ is a continuous shape function on $ \real $, and $ d_Q : M\times M \to \real_0 $ is the distance-like function on $ M$ from~\eqref{eqn:def:dQ}, parameterized by the number of delays $Q$. \blue{Kernels of this class are sometimes referred to as stationary kernels \cite{Genton01}, as they only depend on distances between datapoints.} For example, in \eqref{eqn:def:kQ}, we used a Gaussian shape function, which is popular in manifold learning and other related geometrical data analysis techniques. Note that $d_Q$ is symmetric, non-negative, and satisfies the triangle inequality, but depending on the properties of $ F $ and the number of delays it may vanish on distinct points. That is, $ d_Q $ is a pseudo-distance on $ M $, induced from delay-coordinate mapped data with $ Q $ delays. 

\blue{The kernels in~\eqref{eqKQ2} satisfy Assumption \ref{asmptn:ker}(i), and the associated kernel integral operators $K_{Q}$ have all four properties listed above. In addition, if $h$ is strictly positive, $k_Q$ satisfies Assumption~\ref{asmptn:ker}(ii). The behavior of integral operators associated with other classes of kernels, e.g., the covariance operators employed in SSA and Hankel matrix analysis induced by inner products in data space, can be studied via similar techniques to those presented below. However, it should be kept in mind that the Markov normalization procedure described in Section~\ref{secMarkov} (which will be important for the well-posedness of the Galerkin schemes in Sections~\ref{sect:Galerkin} and~\ref{sect:numerics}) requires that the kernel be sign-definite. Another consideration to keep in mind is that the ability to approximate Koopman eigenfunctions with our techniques depends on the ``richness'' of the range of $K_Q$. As can be readily verified, the operator $K_Q$ constructed from covariance kernels in $d$-dimensional data space (as in Assumption~\ref{asmptn:standing}) has at most a $dQ$-dimensional range, whereas the corresponding operators associated with Gaussian kernels, as well as other non-polynomial kernels, have typically infinite-dimensional range for any $Q$. Our approach should also be applicable with little modification to families of kernels of the form
\begin{displaymath}
\tilde k_Q(x,y) = \frac{1}{Q} \sum_{q=0}^{Q-1} h(d_1(\Phi^{q\,\Delta t}(x),\Phi^{q\,\Delta t}(y))),
\end{displaymath}
where averaging takes place after application of the shape function.} Lemma \ref{lem:Rho} below states some useful properties of $K_{Q}$ associated with strictly positive kernels. In what follows, $1_S$ will denote the constant function equal to 1 on a set $S$. 

\begin{lem} \label{lem:Rho} 
\blue{Under Assumptions~\ref{asmptn:standing} and~\ref{asmptn:ker}(ii)}, for any $ Q \in \num$, the functions $ \rho_{Q} = K_{Q} 1_X $, and $ \sigma_{Q} =K_{Q} \left( 1/\rho_{Q}\right) $ are continuous and positive. Moreover, restricted on $ X $, they are bounded away from zero.
\end{lem}
\begin{proof}
The claims follow directly by compactness of $X$ and the fact that  $ k_Q \rvert_{X\times X} $ is a continuous function, bounded away from zero. \qed
\end{proof}

Intuitively, $ \rho_{Q} $ can be thought of as a ``sampling density'' on $ X $. For instance, if $ X $ were a manifold embedded in $ \real^{Qd} $ by a delay-coordinate map constructed from $ F $, then up to an $ \epsilon $-dependent scaling, $ \rho_{Q} $ would approximate the density of the invariant measure $ \mu $ relative to the volume measure associated with that embedding. 

\begin{rk*}In a number of applications, such as statistical learning on manifolds \cite{BelkinNiyogi03,CoifmanLafon06,BerrySauer16b,BerryHarlim16}, one-parameter families of integral operators such as $ K_{Q} $ and $ P_{Q} $ are studied in the limit $ \epsilon \to 0 $, where under certain conditions they can be used to approximate generators of Markov semigroups; one of the primary examples being the Laplace-Beltrami operator on Riemannian manifolds. Here, the fact that the state space $ X $ may not (and in general, will not) be smooth precludes us from taking such limits unconditionally. However, according to Theorem~\ref{thm:B}(ii), passing first to the limit $ Q \to \infty $ allows one to view $ K $ and $ P$ as operators on functions on a smooth manifold, namely a $ D $-dimensional torus, and study the small-$\epsilon $ behavior of these operators in that setting. 
\end{rk*}

\subsection{Asymptotic behavior in the infinite-delay limit}

To study the behavior of $ K_{Q} $ in the limit of infinitely many delays, $ Q \to \infty $, we first consider the properties of the pseudometric $d_{Q}$ in the same limit. The latter can be studied in turn through a useful (nonlinear) map $ \Psi : C^0(X) \to L^\infty( X \times X, \mu \times \mu ) $, which maps a given observation function $F $ into a (pseudo)metric on $X$, namely, 
\begin{equation}\label{eqn:def:inf_delays}
\begin{gathered}
\InfDelay(F)(x,y):= \lim_{Q\to\infty} \InfDelay_Q(F)(x,y), \\ 
\InfDelay_Q(F)(x,y) := \frac{ 1 }{Q} \sum_{q=0}^{Q-1} \left\lVert F( \Phi^{q \, \Delta t } ( x ) )-F( \Phi^{q \, \Delta t }(y) ) \right\rVert^2.
\end{gathered}
\end{equation}
In what follows, $d_X:X\times X\to\real$ will denote the metric $X$ inherits from $M$.

\begin{thm}\label{thm:inf_delay_decomp}
Let Assumption \ref{asmptn:standing} hold, and $F=F_{\Disc} + F_{\Disc^\bot}$ be the $L^2 $ decomposition of $ F $ from~\eqref{eqn:L2_decomp}. Then, $ \InfDelay(F) $ in~\eqref{eqn:def:inf_delays} is well-defined as a function in $L^\infty(X\times X,\mu\times\mu)$, and $\Psi_Q(F)$ converges to $\Psi(F)$ in $L^p(X\times X,\mu\times \mu) $ norm for $ 1 \leq p <\infty$. Moreover:
\begin{enumerate}[(i)]
\item For every $ t \in \real $ and $\mu$-a.e. $ x, y \in X$, $\InfDelay(F)(\Phi^{t}(x), \Phi^{t}(y))=\InfDelay(F)(x,y)$. %
\item For $\mu$-a.e. $x,y\in X$, $ \InfDelay(F)(x,y) =\InfDelay(F_{\Disc^\bot})(x,y) + \InfDelay(F_{\Disc})(x,y)$. 
\item $\InfDelay(F_{\Disc^\bot})$ is a constant almost everywhere and equals $2 \|F_{\Disc^\bot}\|_{L^2}^2$. Therefore, 
\begin{equation}\label{eqn:InfDelay_chrctzn}
\InfDelay(F) = \InfDelay(F_{\Disc}) + 2 \|F_{\Disc^\bot}\|_{L^2}^2.
\end{equation}
In particular, $\InfDelay(F) \in \Disc\times\Disc $. 
\setcounter{enum_sav}{\value{enumi}}
\end{enumerate}
If, moreover, Assumption \ref{asmptn:C0_ae} holds:
\begin{enumerate}[(i)]
\setcounter{enumi}{\value{enum_sav}} 
\item $\InfDelay(F_{\Disc})\in C^0(X\times X)$ and $\InfDelay_Q(F_{\Disc})$ converges to $\InfDelay(F_{\Disc})$ uniformly on $X\times X$. 
\item $\InfDelay(F)$ is uniformly continuous on a full-measure, dense subset of $X\times X$. 
\item $ \InfDelay(F) $ has a unique continuous extension $ \bar \Psi( F ) \in C^0(X\times X) $, and $ \Psi_Q( F ) $ converges to $ \bar \Psi( F ) $ $\mu$-almost uniformly.
\end{enumerate}
\end{thm}

\begin{proof} To prove well-definition of $ \Psi $, note that $\InfDelay(F)$ exists $ \mu $-a.e.\ since it is the pointwise limit of the Birkhoff averages $\InfDelay_Q(F)$ of the continuous function $ d_1: ( x, y ) \to \lVert F( x ) - F( y ) \rVert $ with respect to the product flow $ \Phi^t \times \Phi^t $ on $X\times X$. By compactness of $X\times X$, each of the functions $\InfDelay_Q(F)$ is bounded above by $ \lVert d_1 \rVert_{C^0(X\times X)} $. Therefore, $ \Psi( F ) $ lies in $ L^\infty( X \times X, \mu \times \mu ) $, and thus in $L^p( X \times X, \mu \times \mu ) $, $ 1 \leq p < \infty$, since $ \mu \times \mu $ is a probability measure. The $\Psi_Q( F ) \to \Psi(F)$ convergence in $ L^p(X\times X, \mu\times\mu ) $, 1 $ \leq p < \infty $, then follows from the $L^p $ Von Neumann ergodic theorem.

By the invariance of the infinite Birkhoff averages, $\InfDelay(F)$ is invariant under the flow $\Phi^{\Delta t}\times\Phi^{\Delta t}$. Thus $\InfDelay(F)$ must lie in the kernel of $V\otimes V$ and thus is invariant under the flow $\Phi^{t}\times\Phi^{t}$ for all $t \in \real$, proving Claim (i).

To prove Claim (ii), let $x_q$ and $ y _q$ denote $\Phi^{q\, \Delta t}(x)$ and $\Phi^{q\,\Delta t}(y)$ respectively. Let $G_{\Disc}:X\times X\to\real^d$:= $(x,y)\to F_{\Disc}(x_q) - F_{\Disc}(y_q)$, and similarly define $G_{\Disc^\bot}:X\times X\to\real^d$. Expanding the right-hand side of~\eqref {eqn:def:inf_delays} gives,
\[\begin{split}
\InfDelay(F)(x,y) &= \lim_{Q\to\infty}\frac{ 1 }{Q} 
\sum_{q=0}^{Q-1} \left( \left\lVert G_{\Disc}(x_q,y_q)\right\rVert^2 + \left \lVert G_{\Disc^\bot}(x_q,y_q)\right\rVert^2 \right)\\
& -2\lim_{Q\to\infty}\frac{ 1 }{Q} \sum_{q=0}^{Q-1} G_{\Disc}(x_q,y_q) \cdot G_{\Disc^\bot}(x_q,y_q) ,
\end{split}\]
and the first two terms in the equation above are $\InfDelay(F_{\Disc})(x,y)$ and $\InfDelay(F_{\Disc^\bot})(x,y)$ respectively. Therefore, to prove Claim~(ii), it suffices to prove that the third term vanishes. This is equivalent to showing that for $\mu$-a.e.\ $x,y\in X$, 
\begin{displaymath}
\lim_{Q\to\infty} \frac{ 1 }{Q} \sum_{q=0}^{Q-1} \left(F_{\Disc^\bot}(x_q)-F_{\Disc^\bot}(y_q)\right)\cdot\left(F_{\Disc}(x_q)-F_{\Disc}(y_q)\right)=0,
\end{displaymath}
which follows from Lemma~\ref{lem:Decor}. This completes the proof of Claim~(ii).

To prove Claim~(iii), let $x_n$ and $y_n$ denote $\Phi^{n\, \Delta t }(x)$ and $\Phi^{n \, \Delta t}(y)$, respectively. Then, \eqref{eqn:def:inf_delays} can be rewritten for $F_{\Disc^\bot}$ as
\[\begin{split}
(\InfDelay F_{\Disc^\bot})(x,y) &= \lim_{N\to\infty}\frac{1}{N}\sum_{n=0}^{N-1}|F_{\Disc^\bot}(x_n)|^2 + \lim_{N\to\infty}\frac{1}{N}\sum_{n=0}^{N-1}|F_{\Disc^\bot}(y_n)|^2 \\
& + 2\lim_{N\to\infty}\frac{1}{N}\sum_{n=0}^{N-1}F_{\Disc^\bot}(x_n)F_{\Disc^\bot}(y_n).
\end{split}\]
The first two terms converge to the constant $\|F_{\Disc^\bot}\|_{L^2}^2$. It is therefore sufficient to show that the last term vanishes. Indeed, since the function $J:(x,y)\to F_{\Disc^\bot}(x)F_{\Disc^\bot}(y)$ lies in the continuous spectrum subspace of the product-system $(X\times X,\Phi^t\times \Phi^t,\mu\times\mu)$, we have 
\[\lim_{N\to\infty}\frac{1}{N}\sum_{n=0}^{N-1}F_{\Disc^\bot}(x_n)F_{\Disc^\bot}(y_n) =\langle J, 1_{X\times X} \rangle = 0.\]

Since $F_{\Disc}$ is continuous, $\InfDelay(F_{\Disc})$ is continuous by a classic result of Krengel (\cite{Krengel85}, Theorem 1.2.7). This proves Claim~(iv).

Turning to Claim (v), it follows directly from Claims (iii) and (iv) that there exists a full-measure subset $ S \subseteq X \times X $ on which $ k_{\infty} $ is uniformly continuous. Suppose that $S$ were not dense in $ X \times X $. Then, there would exist an open set $ B \subset X \times X $ disjoint from $ S $, and with positive measure (since $X \times X $ is the support of $ \mu \times \mu$, and every open subset of the support of a Borel measure has positive measure), which would in turn imply that $ (\mu \times \mu)(S) < 1$, leading to a contradiction. Therefore, $ S $ is a full-measure, dense subset of $ X \times X $, completing the proof of the claim.

Finally, the existence of $ \bar \Psi( F ) $ in Claim (vi) follows from the fact that $ \Psi( F ) $ is uniformly continuous on the dense subset $ S $ of the compact metric space $ X \times X $, and the almost uniform convergence of $ \Psi_Q( F ) $ to $ \bar \Psi( F ) $ is a consequence of Egorov's theorem. \qed
\end{proof}

\begin{rk*} \blue{Although the measure $\mu\times\mu$ is invariant under $\Phi^t\times\Phi^t$, it is not ergodic. In fact, it is ergodic iff $(\Phi^t,\mu)$ is mixing (equivalently, $U^t$ has purely continuous spectrum and a simple eigenvalue at 1), in which case the metric $d_\infty$ would be constant almost everywhere, in accordance with \eqref{eqn:InfDelay_chrctzn}. }
\end{rk*}

Theorem~\ref{thm:inf_delay_decomp} establishes that the function $ d_\infty : D( d_\infty ) \to \real $, such that 
\begin{displaymath}
d_{\infty}(x,y) := \lim_{Q\to\infty} d_{Q}(x,y); \quad ( x, y ) \in D( d_\infty) \subseteq X \times X
\end{displaymath}
is well-defined as a function in $ L^p( X \times X, \mu \times \mu ) $, $ 1 \leq p \leq \infty $, with $ \sup d_\infty\leq \lVert d_1 \rVert_{C^0(X\times X)} $. It can also be verified that $d_\infty$ satisfies the triangle inequality and is non-negative. However, depending on the properties of the dynamical system and observation map, it may be a degenerate metric as $ d_\infty( x,y ) $ may vanish for some $x\neq y$, even if $ d_Q(x,y)$ is non-vanishing. In fact, it is easy to check that if $y $ lies in the stable manifold of $ x $, then $d_\infty(x,y)=0$. Analogously to the finite-delay case in~\eqref{eqKQ2}, we employ $ d_\infty $ and the shape function $h$ to define a corresponding kernel $ k_{\infty} : M \times M \to \real $, where 
\begin{equation}\label{eqKInf}
k_{\infty}(x,y) = h( d_\infty(x,y)), \quad ( x, y ) \in D( d_\infty), 
\end{equation}
and $ k_{\infty}(x,y)=0 $ otherwise. We also let $ K $ be the kernel integral operator from~\eqref{eqn:KOp} associated with $ k_{\infty} $. 

Proposition~\ref{prop:observability} shows that the operator $K$ depends only on the quasiperiodic component of $F$, and is a direct consequence of Theorem~\ref{thm:inf_delay_decomp} and \eqref{eqn:InfDelay_chrctzn}.
\begin{prop}\label{prop:observability}
Let $(X,\Phi^t,\mu)$ and $F$ be as in Theorem~\ref{thm:A}. Then, the integral operator $K$ is a constant scaling operator iff its kernel $k_{\infty}$ is a constant $\mu$-a.e., which occurs iff $F_{\Disc}$ is a constant.
\end{prop}

In general, $ k_{\infty} $ may not be continuous. Nevertheless, it has a number of other useful properties, which follow directly from Theorem~\ref{thm:inf_delay_decomp} in conjunction with the boundedness and continuity of the Gaussian shape function.
\begin{lem}\label{lem:KInf}
Under Assumption~\ref{asmptn:standing}, the following hold:
\begin{enumerate}[(i)]
\item $k_{\infty} $ is the $ L^p(X,\mu) $-norm limit, $ 1 \leq p < \infty $, of the sequence of continuous kernels $ k_{1}, k_{2},\ldots$.
\item $k_{\infty}$ is invariant under $U^t \times U^t $ for all $ t \in \real $. 
\item $ k_{\infty} $ lies in $ L^\infty( X \times X, \mu \times \mu ) $, and under Assumption~\ref{asmptn:ker}(ii), $1/k_\infty$ also lies in that space.
\setcounter{enum_sav}{\value{enumi}}
\end{enumerate}
Moreover, if Assumption~\ref{asmptn:C0_ae} additionally holds: 
\begin{enumerate}[(i)]
\setcounter{enumi}{\value{enum_sav}}
\item $k_{\infty}$ is uniformly continuous on a dense, full-measure subset of $ X \times X $. 
\item $ k_{\infty} $ has a unique continuous representative $ \bar k_{\infty} \in C^0(X\times X)$, and as $ Q \to \infty $, $k_{Q} $ converges to $ k_{\infty} $ almost uniformly.
\end{enumerate}
\end{lem}

The stronger regularity properties of $ k_{\infty} $ under Assumption~\ref{asmptn:C0_ae} have the following important implications on the behavior of the corresponding integral operator.
\begin{lem} \label{lem:KL2C0} Under Assumptions~\ref{asmptn:standing} and~\ref{asmptn:C0_ae}, the kernel integral operator $K$ associated with $ k_{\infty} $ has the following properties: 
\begin{enumerate}[(i)]
\item For every $ f \in L^2( X, \mu ) $, $ K f $ has a unique continuous representative. 
\item For every $ f \in C^0( X ) $, $ K f $ is continuous. 
\item $ \lVert K \rVert \leq \lVert k_{\infty} \rVert_{L^\infty(X\times X)} $ in either $L^2 $ or $ C^0 $ operator norm.
\item As an operator on $ C^0(X) $, $ K $ is compact.
\item For every $f\in C^0(X)$, $K_{Q} f$ is a sequence of continuous functions converging $\mu$-a.e. to $K f$.
\end{enumerate}
\end{lem}
\begin{proof} (i) Since $k_{\infty} $ is uniformly continuous on a set $ S \subseteq X \times X $ of full $\mu\times \mu$ measure, there exists a full $\mu$-measure set $X'\subseteq X$, such that for every $x\in X'$, $k_{\infty}(x,\cdot)$ is continuous $\mu$-a.e.\ on $X$. Moreover, proceeding analogously to the proof of Theorem~\ref{thm:inf_delay_decomp}(v), it can be shown that $X' $ is dense in $X$. Let now $f\in L^2(X,\mu)$, $\|f\|_{L^2} =1 $. Then, for every $x_1, x_2\in X'$,
\begin{align} 
\nonumber \left| K f(x_1) - K f(x_2) \right| &= \left| \int_{X'} [k_{\infty}(x_1,y) - k_{\infty}(x_2,y) ] f(y) d\mu(y) \right| \\
\nonumber&\leq \| k_{\infty}(x_1,\cdot) - k_{\infty}(x_2,\cdot)\|_{L^2} \| f\|_{L^2} \\
&\leq \| k_{\infty}(x_1,\cdot) - k_{\infty}(x_2,\cdot)\|_{L^\infty}. 
\end{align}
Since $k_{\infty}$ is uniformly continuous on $S$, for every $\epsilon>0$, there exists $\delta>0$ such that if $d_X(x_1,x_2) < \delta$, $\| k_{\infty}(x_1,\cdot) - k_{\infty}(x_2,\cdot)\|_{L^\infty} < \epsilon$. Thus, for all such $x_1$ and $x_2$, we have $\left| K f(x_1) - K f(x_2) \right| < \epsilon$, which implies that $K f$, restricted to $ X' $, is uniformly continuous. As a result, since $X'$ is dense in the compact metric space $X$, $K f\rvert_{X'}$ has a unique continuous extension $ g \in C^0(X) $. Moreover, since $ X' $ has full measure, $ g $ lies in the same $ L^2 $ equivalence class as $ K f $, proving the claim. 

(ii) Since $ k_{\infty} $ is uniformly continuous on a dense set of full measure, for any $ f \in C^0(X) $, the function $ g : X \times X \to \mathbb{C}$ with $ g(x, y ) = k_{\infty}(x,y)f(y) $ has a unique continuous representative $ \bar g \in C^0(X\times X) $. Therefore, for every $ x \in X $, the function $ k_{\infty}(x,\cdot)f $ is $ \mu $-a.e.\ equal to $ \bar g( x, \cdot ) $ by $ \mu $-a.e.\ continuity of $ k_{\infty}(x, \cdot ) $, and
\begin{displaymath}
K f(x ) = \int_X k_{\infty}(x,y)f(y)\,d\mu(y) = \int_X \bar g_{\infty}(x,y)f(y) \, d\mu(y).
\end{displaymath}
It then follows that $K f $ is continuous by continuity of integrals of $X$-sections of continuous functions on $X \times X $.

(iii) To verify the claim on the $L^2$ and $C^0$ operator norms, observe that for every $ f \in L^2( X, \mu ) $ and $x\in X'$, where $X'$ is as in the proof of Claim~(i),
\[\begin{split}
\left| K f(x) \right| & \leq \left| \int_{X'} k_{\infty}(x,y)f(y) d\mu(y) \right| \\
& \leq \| k_{\infty}(x,\cdot) \|_{L^2} \| f\|_{L^2} \leq \| k_{\infty}(x,\cdot)\|_{L^\infty} \lVert f \rVert_{L^2} \\ 
& \leq \|k_{\infty}\|_{L^\infty(X\times X)} \lVert f \rVert_{L^2}, 
\end{split}\]
and therefore 
\begin{equation}
\label{eqKL2C0}
\lVert K f \rVert_{L^\infty} \leq \lVert k_{\infty} \rVert_{L^\infty(X\times X)} \lVert f \rVert_{L^2}. 
\end{equation}
The bound on the $L^2 $ operator norm follows by setting $ \Vert f \rVert_{L^2} = 1 $ in~\eqref{eqKL2C0}, together with the fact that $ \lVert K f \rVert_{L^2} \leq \lVert K f \rVert_{L^\infty} $. The bound on the $C^0 $ operator norm follows from~\eqref{eqKL2C0} with $ f \in C^0(X) $, in conjunction with the facts that $ \lVert f \rVert_{L^2} \leq \lVert f \rVert_{C^0} $ and $ \lVert K f \rVert_{L^\infty} = \lVert K f \rVert_{C^0} $. 

(iv) Since, by the Arzel\`a-Ascoli theorem, every equicontinuous sequence of functions on a compact metric space has a limit point, it suffices to show that for every sequence $ f_n \in C^0( X ) $ with $ \lVert f_n \rVert_{C^0} \leq 1 $, the sequence $ g_n = K f_n $ has a limit point with respect to $ C^0 $ norm. Let $ \bar k_{\infty} \in C^0( X \times X ) $ be the unique continuous representative of $ k_{\infty} $. For every $ x_1, x_2 \in X $, we have
\begin{displaymath}
\lvert g_n(x_1) - g_n(x_2) \rvert \leq \lVert \bar k_{\infty}(x_1, \cdot ) - \bar k_{\infty}( x_2, \cdot ) \rVert_{C^0},
\end{displaymath}
and by uniform continuity of $ \bar k_{\infty} $, for any $ \epsilon > 0 $, there exists $ \delta > 0 $, independent of $ n $, such that, for every $ x_1, x_2 \in X $ with $ d(x_1,x_2) < \delta $, $ \lvert g_n(x_1 ) - g_n(x_2) \rvert < \epsilon $. This establishes equicontinuity of $ g_n $, and thus compactness of $ K $ on $ C^0(X) $.

(v) The continuity of $K_Q f$ and $Kf$ follows from Claim~(ii). The $\mu$-a.e.\ convergence follows from Lemma \ref{lem:KInf}(v).
\end{proof}

We end this section with two important corollaries of Theorem~\ref{thm:inf_delay_decomp} and Lemmas~\ref{lem:KInf}, \ref{lem:KL2C0}, which are central to both Theorems~\ref{thm:A} and~\ref{thm:B}. 
\begin{cor}
\label{cor:commute}
The operators $ U^t $ and $ K $ commute.
\end{cor}
\begin{proof}
Since $ \mu $ is an invariant measure, for every $ x $ in $ X $ and $ t \in \real $ we have 
\begin{displaymath}
K f( x) = \int_X k_{\infty}(x, y) f( y ) \, d\mu(y) = \int_X k_{\infty}(x, \Phi^t( y )) f( \Phi^t( y ) ) \, d\mu(y). 
\end{displaymath}
It therefore follows from Lemma~\ref{lem:KInf}(ii) that
\begin{displaymath}
K f( x ) = \int_X k_{\infty}( \Phi^{-t}(x),y) f( \Phi^t (y) ) \, d\mu(y) = U^{t*} K U^tf(x),
\end{displaymath}
and the claim of the corollary follows. \qed
\end{proof}

\begin{cor}\label{cor:inf_delay}
Under \blue{Assumptions~\ref{asmptn:standing} and~\ref{asmptn:ker}(ii)}, the function $ \rho = K 1_X $ is $ \mu $-a.e.\ equal to a constant bounded away from zero (i.e., $ 1/\rho $ lies in $ L^\infty(X,\mu) $). Further, if Assumption~\ref{asmptn:C0_ae} holds, then $ \rho\rvert_X $ and $ 1/\rho\rvert_X$ are continuous.
\end{cor} 
\begin{proof}
Corollary~\ref{cor:commute} and the fact that $ U^t 1_X= 1_X $ imply that $ U^t \rho = \rho $, and it then follows by ergodicity that $ \rho $ is constant $ \mu $-a.e. That $\lVert 1/\rho \rVert_{L^\infty} $ is finite follows from Lemma~\ref{lem:KInf}(iii). Finally, the continuity of $ \rho $ under Assumption~\ref{asmptn:C0_ae} is a direct consequence of Lemma~\ref{lem:KL2C0}. \qed 
\end{proof}

\subsection{Markov normalization}\label{secMarkov}

Next, we construct the Markov operators $ P_{Q} $ and $ P $ appearing in Theorems~\ref{thm:A} and~\ref{thm:B} by normalization of $ K_{Q} $ and $ K $. Throughout this section, we consider that Assumptions~\ref{asmptn:standing} and~\ref{asmptn:ker} hold. Under these assumptions, we employ a normalization procedure introduced in the diffusion maps algorithm \cite{CoifmanLafon06} and further developed in \cite{BerrySauer16b}, although there are also other approaches with the same asymptotic behavior. Specifically, using the normalizing functions $\rho_{Q}$ and $\sigma_{Q}$ from Lemma~\ref{lem:Rho} and $\rho$ from Corollary~\ref{cor:inf_delay}, we introduce the kernels $p_{Q} : M \times M \to \real$ and $ p : M \times M \to \real $, given by 
\begin{equation} \label{eqn:PKernel}
p_{Q}( x, y ) = \frac{ k_{Q}(x, y ) }{ \sigma_{Q}( x ) \rho_{Q}(y )}, \quad 
p( x, y ) = 
\begin{cases}
k_{\infty}(x, y )/\rho( x ), & \rho(x) > 0,\\
0, & \text{otherwise},
\end{cases}
\end{equation}
respectively. 
By Lemma~\ref{lem:Rho}, $ p_{Q} $ satisfies the boundedness and continuity properties in Assumption~\ref{asmptn:ker}. On the other hand, $ p $ is neither guaranteed to be continuous nor bounded on arbitrary compact sets, but it nevertheless follows from Lemma~\ref{lem:KInf} and Corollary~\ref{cor:inf_delay} that both $ p $ and $1/p$ lie in $L^\infty(X\times X)$. Based on these facts, we can therefore define the kernel integral operators $P_{Q} : L^2(X,\mu ) \to L^2(X,\mu)$ and $P : L^2(X,\mu) \to L^2(X,\mu)$ from~\eqref{eqn:KOp} associated with the kernels $p_{Q}$ and $p$, respectively, and these operators are both Hilbert-Schmidt (see Section~\ref{secKernelChoice}). Note that $p$ and $P$ have analogous properties to those stated for $k_{\infty}$ and $K $ in Lemmas~\ref{lem:KInf}, \ref{lem:KL2C0} and Corollary~\ref{cor:commute}. In particular, $ p$ is invariant under $U^t \times U^t$, and $ P $ commutes with $ U^t $.

The operators $P_{Q}$ and $P$ can also be obtained directly from $K_{Q}$ and $K$, respectively, through the sequence of operations
\begin{equation} \label{eqn:def_KHP_Q}
\tilde{K}_{Q}f := K_{Q} \left( \frac{ f }{ K_{Q} 1_X } \right), \quad P_{Q}f = \frac{ \tilde{K}_{Q} f }{ \tilde{K}_{Q} 1_X }, \quad P f = \frac{ K f }{ K 1_X}.
\end{equation}
In \cite{BerrySauer16b}, the steps leading to $ \tilde{K}_{Q} $ from $ K_{Q} $ and to $ P_{Q} $ from $ \tilde{K}_{Q} $ are called right and left normalization, respectively. In the case of $P$, the effects of right normalization cancel since $ K 1_X $ is $ \mu $-a.e.\ constant by Corollary~\ref{cor:inf_delay}, so it is sufficient to construct this operator directly via left normalization of $ K$.

As is evident from~\eqref{eqn:def_KHP_Q}, $P_{Q}$ and $P$ are both Markov operators preserving constant functions. Moreover, for all $ x \in M $ we have $ \int_X p_{Q}( x, \cdot ) \, d\mu = 1 $, and for $\mu$-a.e.\ $ x \in M $, $ \int_X p( x, \cdot ) \, d\mu = 1 $, i.e., both $p_{Q}$ and $p$ are transition probability kernels. In particular, since $ X $ is compact and $ p_{Q} $ and $ p $ are essentially bounded below, $P_{Q}$ and $P$ are both ergodic Markov operators; that is, their eigenspaces at eigenvalue 1 are one-dimensional.

The Markov kernel $ p $ is $\mu$-a.e. symmetric by symmetry of $ k_{\infty} $ and the fact that $ \rho $ is $\mu$-a.e. constant. As a result, $ P $ is self-adjoint, its eigenvalues admit the ordering $ 1 = \lambda_0 > \lambda_1 \geq \lambda_2 \geq \cdots $, and there exists a real orthonormal basis $L^2(X,\mu) $ consisting of corresponding eigenfunctions, $ \phi_j $, with $ \phi_0 $ being constant. On the other hand, because $ p_{Q} $ is not symmetric, the operator $P_{Q}$ is not self-adjoint, but is nevertheless related to a self-adjoint operator via a similarity transformation by a bounded multiplication operator with a bounded inverse. To verify this, define
\[\tilde{\sigma}_{Q}=\sigma_{Q}/\rho_{Q},\quad \hat{\sigma}_{Q}= \sqrt{\sigma_{Q} \rho_{Q}},\]
where $\rho_{Q}$ and $\sigma_{Q}$ are as in Lemma~\ref{lem:Rho}. Let also $D_{Q} $ be the multiplication operator which multiplies by $\tilde{\sigma}_{Q}$, and $\hat{P}_{Q}$ the kernel integral operator with kernel $\hat p_{Q} : M \times M \to \real$
\begin{equation}
\label{eqPHat}
\hat p_{Q}(x,y) = \frac{k_{Q}(x,y)}{\hat{\sigma}_{Q}(x)\hat{\sigma}_{Q}(y)}.
\end{equation}
Observe now that $\hat{P}_{Q}$ is a symmetric operator, and $P_{Q}$ is related to $\hat{P}_{Q}$ via the similarity transformation 
\begin{equation}
\label{eqSimilarity}
\hat P_{Q} =D_{Q}^{1/2}P_{Q} D_{Q}^{-1/2}; 
\end{equation}
that is, for every $f\in L^2(X,\mu)$,
\[\begin{split}
D_{Q}^{1/2}P_{Q} D_{Q}^{-1/2} f(x) &= \int_{X} \sqrt{\frac{\sigma_{Q}(x)}{\rho_{Q}(x)}} \frac{k_{Q}(x,y)}{\sigma_{Q}(x) \rho_{Q}(y)} f(y) \sqrt{\frac{\rho_{Q}(y)}{\sigma_{Q}(y)}} \, d\mu(y)\\
& = \int_{X} \frac{k_{Q}(x,y)}{\hat{\sigma}_{Q}(x)\hat{\sigma}_{Q}(y)} f(y) \, d\mu(y) = \hat{P}_{Q}f(x).
\end{split}\]
The following are useful properties of $\hat{P}_{Q}$ that follow from its relation to $ P_{Q} $.
\begin{enumerate}
\item $\hat{P}_{Q}$ has the same discrete spectrum as $ P_{Q} $, consisting of eigenvalues $\lambda_{j,Q}$ with $1=\lambda_{0,Q} > \lambda_{1,Q} \geq \lambda_{2,Q} \geq \cdots$. 
\item Let $\phi_{j,Q}$ denote the eigenfunctions of $\hat{P}_{Q}$ corresponding to the nonzero eigenvalues $ \lambda_{j,Q} $. These form an orthonormal basis for the closed subspace $\overline{\ran\hat{P}_{Q} } = (\ker \hat{P}_{Q})^\bot$. Moreover, the $\phi_{j,Q}$ can be chosen to be real-valued.
\item The eigenfunction $ \phi_{0,Q} $ of $ \hat P_{Q} $ is equal up to proportionality constant to $\rho_{Q} \sigma^{1/2}_{Q}$. 
\end{enumerate}


\begin{rk*}
In applications, it may be the case that $ \rho_{Q} $ and $ 1/ \rho_{Q} $ take a large range of values. In such situations, it may be warranted to replace~\eqref{eqn:def:dQ} by a variable-bandwidth kernel of the form $k_{Q}(x,y) = \exp\left( - \frac{ d^2_Q( x, y ) }{ \epsilon r_{Q}( x ) r_{Q}(y) } \right)$, with a bandwidth function $ r_{Q} $ introduced so as to control the decay of the kernel away from the diagonal, $ x = y $. Various types of bandwidth functions have been proposed in the literature, including functions based on neighborhood distances \cite{ZelnikManorPerona04,BerrySauer16}, state space velocities \cite{GiannakisMajda12a,Giannakis15}, and local density estimates \cite{BerryHarlim16}. While we do not study variable bandwidth techniques in this work, our approach should be applicable in that setting too, so long as Corollary~\ref{cor:inf_delay} holds. 
\end{rk*}


\section{Proof of Theorems~\ref{thm:A}--\ref{thm:D} and Corollary~\ref{corSpectral}}\label{sect:proof_main}

\paragraph{Proof of Theorem~\ref{thm:A}.} That $ P $ and $ U^t $ commute follows from the invariance of $ p $ under $ U^t \times U^t $ and an analogous calculation to that in the proof of Corollary~\ref{cor:commute}. Next, as $Q\to\infty$, $p_{Q}$ converges to $p$ in any $L^p(X\times X,\mu\times\mu)$ norm with $ 1 \leq p < \infty $ by the analogous result to Lemma~\ref{lem:KInf}(i) that holds for these kernels (see Section~\ref{secMarkov}). In particular, that $p_{Q}$ converges to $p$ in $L^2(X\times X,\mu\times \mu) $ norm implies that $ P_{Q} $ converges to $ P$ in $L^2(X,\mu)$ operator norm, since $P_{Q} - P$ is Hilbert-Schmidt and thus bounded in operator norm by $\lVert p_{Q}- p \rVert_{L^2(X\times X)}$. \qed

\paragraph{Proof of Theorem~\ref{thm:B}.} We first establish that $ \tau $ is a.e.\ invariant under $ \Phi^t \times \Phi^t $. Since the integral operator $T$ commutes with $U^t$, for $\mu$-a.e.\ $x\in X$,
\begin{multline*}
\int_X \tau(\Phi^{t}(x),\Phi^{t}(y'))f(\Phi^{t}( y')) \, d\mu(y') = \int_X \tau(\Phi^{t}(x),y)f(y) \,d\mu(y)\\
= U^t Tf(x) = T(U^tf)(x) = \int_X \tau(x,y') f(\Phi^{t}( y')) \, d\mu(y'),
\end{multline*}
where the second equality was obtained by the change of variables $y= \Phi^t(y')$, and utilizes the invariance of the measure $\mu$ under $\Phi^t$. The only way the terms at the two ends of the equation can be equal for $\mu$-a.e.\ $x\in X$ is if $\tau(\Phi^{t}(x),\Phi^{t}(y'))$ = $\tau(x,y')$ $\mu$-a.e.

Next, observe that, by \eqref{eqn:L2_decomp}, the space $L^2(X\times X, \mu\times \mu)$ splits as the $U^t \times U^t $-invariant orthogonal sum of $\Disc\otimes\Disc$, $\Disc^\bot\otimes\Disc^\bot$, $\Disc^\bot\otimes\Disc$, and $\Disc\otimes\Disc^\bot$. Since $\tau$ is an $L^2$ kernel, it has orthogonal projections onto each of these subspaces, all of which are $ U^t \otimes U^t $-invariant by the invariance of $ \tau $ just established. By symmetry of $ \tau $, the projections onto $\Disc^\bot\otimes\Disc$ and $\Disc\otimes\Disc^\bot$ vanish. Moreover, the projection $ \tau_{\Disc^\perp \otimes \Disc^\perp} \in \Disc^\perp \otimes \Disc^\perp $ is orthogonal to constant functions, and it follows by the Birkhoff ergodic theorem that for $ \mu \times \mu $-a.e.\ $x, y \in X \times X $, 
\begin{align*}
0 &= \langle 1_{X \times X}, \tau_{\Disc^\perp \otimes \Disc^\perp} \rangle \\
&= \lim_{N\to\infty} \frac{1}{N} \sum_{n=0}^{N-1} \tau_{\Disc^\perp \otimes \Disc^\perp} ( \Phi^{n\,\Delta t} (x), \Phi^{n\,\Delta t} (y) ) \\
&= \lim_{N\to\infty} \frac{1}{N} \sum_{n=0}^{N-1} \tau_{\Disc^\perp \otimes \Disc^\perp} ( x, y ) \\
&= \tau_{\Disc^\perp \otimes \Disc^\perp} ( x, y ).
\end{align*}
This completes the proof of Claim~(i). The statements in Claim~(ii) that $\Disc^\bot \subset \ker(T)$ and that $\Disc$ and $ \Disc^\perp $ are invariant under $T$ are direct consequences of Claim~(i). 

The remaining two claims in the theorem, which requires that both $ \Disc $ and $ \ran T $ contain non-constant functions, can be proved by means of the following, slightly stronger, result.
\begin{prop} \label{prop:W_lambda}
For any nonzero eigenvalue $\lambda $ of $T$, the corresponding eigenspace $W_\lambda$ is invariant under the action of the Koopman generator $V$, and $V\rvert_{W_\lambda}$ is diagonalizable. Moreover, the constant function $1_X$ is an eigenfunction of $T$. If $ W_{\lambda} $ does not contain $ 1_X$, its dimension is an even number. 
\end{prop}
\begin{proof}
Since $T$ is compact, every nonzero eigenvalue $\lambda$ has finite multiplicity and its corresponding eigenspace $W_{\lambda}$ has finite dimension, $ l = \dim W_{\lambda} $. Since $U^t $ commutes with $T$, $ U^t $ and hence $V$ leave $W_{\lambda}$ invariant. Similarly, since the constant function is an eigenfunction of $V$, it is an eigenfunction of $T$. 

Let $ \lambda_0 $ be the eigenvalue of $ T $ corresponding to the constant eigenfunction, and $\lambda \neq \lambda_0$ be any other eigenvalue of $ T $. Then, $V\rvert_{W_{\lambda}}$ is a skew-symmetric operator on a finite-dimensional space, and thus can be diagonalized with respect to a basis of simultaneous eigenfunctions of $T$ and $ V$. Fix any element $ \zeta $ of this basis. By our choice of $\lambda$, $\zeta$ is a non-constant eigenfunction of $V$, hence $\langle\zeta,1\rangle=0$. Therefore, by ergodicity of $ ( \Phi^t, \mu ) $, $V\zeta=i\omega\zeta$ for some $\omega\neq 0$. This implies that $\zeta$ has non-zero real and imaginary parts. Hence, the conjugate $\zeta^*$ is linearly independent from $\zeta$ and corresponds to eigenvalue $-i\omega$ of $V$. However, since $T$ is a real operator, $\zeta^*$ lies in $W_\lambda$. We therefore conclude that $W_\lambda$ can be split into disjoint 2-dimensional spaces spanned by the conjugate pair of eigenfunctions $\zeta$ and $\zeta^*$. Therefore $\dim W_{\lambda} $ is an even number. \qed
\end{proof}

\begin{cor}\label{corBlock}
\blue{The representation of $V\rvert_{\overline{\ran P}} $ in the basis $ \{ \phi_0, \phi_1, \ldots \} $ has a block-diagonal structure, consisting of even-sized blocks associated with the eigenspaces $ W_{\lambda \neq 1} $, and a $1\times 1$ block with the element 0, associated with $W_1$. Moreover, the range of $P$ lies in the domain of $V$, and $V\rvert_{\overline{\ran P}} $ and $P\rvert_{\overline{\ran P}} $ are simultaneously diagonalizable.} 
\end{cor}

By Proposition \ref{prop:W_lambda}, $U^t$ and $T$ have joint eigenfunctions, each of which factors the dynamics into a rotation on the circle in accordance with~\eqref{eqn:Def_koop_eigen}. According to Proposition~\ref{prop:semi_conj_pi}, any collection of $D$ such eigenfunctions factors the dynamics into a rotation on $\TorusD{D}$. This proves Claim~(iii).

To prove Claim (iv), we use~\eqref{eqEigProd} to expand the kernel as
\begin{displaymath}
\tau = \sum_{\vec a, \vec b \in \mathbb{Z}^m } \tilde \tau_{\vec a \vec b} z_{\vec{a}} \otimes z_{\vec{b}}.
\end{displaymath}
In this expansion, there is a minimal number $D\leq m $ of generating eigenfunctions $z_j$ from~\eqref{eqEigProd}, arranged without loss of generality as $ z_1, \ldots, z_D $, such that the expansion coefficients $ \tilde \tau_{\vec a \vec b} $ corresponding to $ \vec a = ( a_1, \ldots, a_m ) $ and $ \vec b = ( b_1, \ldots, b_m ) $ with nonzero $ a_{D+1}, \ldots, a_m $ and $ b_{D+1}, \ldots, b_m $, respectively, vanish (in other words, the kernel $ \tau $ does does not project onto the subspaces generated by $ z_{D+1},\ldots, z_m $ and their powers). By Proposition~\ref{prop:semi_conj_pi}, the Koopman eigenfunctions corresponding to non-vanishing $ \tilde \tau_{\vec a \vec b} $ can be expressed as $ z_{\vec a} = \zeta_{\vec a} \circ \pi $, where the $ \zeta_{\vec a } $ are smooth Koopman eigenfunctions on $ \mathbb{T}^D $ associated with an ergodic rotation. Thus, denoting the index set for the nonzero $ \tilde \tau_{\vec a \vec b} $ by coefficients by $ I \in \mathbb{Z}^m \times \mathbb{Z}^m $, we have $ \tau( x , y ) = \hat \tau( \pi( x ), \pi( y ) ) $ for $ \mu \times \mu $-a.e., $ (x,y) \in X \times X $, where $ \hat \tau $ is the $ L^2 $ kernel on $ \mathbb{T}^D $ given by 
\begin{displaymath}
\hat \tau = \sum_{\vec a, \vec b \in I} \tilde \tau_{\vec a \vec b} \zeta_{\vec a} \otimes \zeta_{\vec b}.
\end{displaymath}
This completes the proof of Claim (v) and of Theorem~\ref{thm:B}. \qed

\paragraph{Proof of Theorem~\ref{thm:D}.} That $p$ is uniformly continuous on a full-measure, dense subset of $X\times X $ follows from the analogous result to Lemma~\ref{lem:KInf}(iv), which holds for $p$ (see Section~\ref{secMarkov}). Claims~(i)--(iv) of the theorem follow analogously to Lemma~\ref{lem:KL2C0}. \qed

\paragraph{Rates of convergence in the continuous case.} As an auxiliary result, we state a lemma that establishes rates of convergence with respect to the number of delays $ Q $ of the kernel integral operators studied in this work. 

\begin{lem}[Convergence of commutators]\label{lem:cmmt_Koopman}
    Let the assumptions of Theorem~\ref{thm:D} hold, and the shape function $h$ from \eqref{eqKQ2} be continuously differentiable. Then, the following operators converge in $C^0(X)$ operator norm to $0$ as $Q\to\infty$, with rates given below:
\begin{enumerate}[(i)]
\item $\left \|U^{\Delta t} K_{Q} - K_{Q} U^{\Delta t}\right\| = O\left(Q^{-1}\right)$,
\item $\left\|U^{\Delta t} \tilde{K}_{Q} - \tilde{K}_{Q} U^{\Delta t}\right\| =O\left(Q^{-1}\right)$,
\item $\left\|U^{\Delta t} P_{Q} - P_{Q} U^{\Delta t}\right\| =O\left(Q^{-1}\right)$.
\end{enumerate}
\end{lem}
\begin{proof} Let $ \tilde{F}_{Q,\Delta t}(x,y):= \left\| F(x) - F(y) \right\| - \left\| F(\Phi^{Q\, \Delta t }x) - F(\Phi^{Q\, \Delta t }(y)) \right\| $, and notice that by continuity of $ F $ and compactness of $X$ this quantity is bounded on $X\times X$. Note that (i) $ d_{Q}( \Phi^{\Delta t }(x), \Phi^{\Delta t }(y)) = d_{Q}(x,y)+Q^{-1} \tilde{F}_{Q,\Delta t}(x,y) $; and (ii) $h(u + \Delta u) = h(u) +\Delta u\, h'(u) + o(\Delta u)$, as $\Delta u\to 0$. Thus,
\[\begin{split}
k_{Q}(\Phi^{\Delta t}(x), \Phi^{\Delta t}(y)) &= h(d_{Q}( \Phi^{\Delta t }(x), \Phi^{\Delta t }(y))) = h(d_{Q}(x,y) + Q^{-1} \tilde{F}_{Q,\Delta t}(x,y)) \\
& = h(d_{Q}(x,y)) + O(Q^{-1}) =   k_{Q}(x,y) + O(Q^{-1}), 
\end{split}\]
where the estimate holds uniformly with respect to $x,y \in X$. Therefore, for every $ f\in L^2(X,\mu) $ and $ x \in X$ we have
\[\begin{split}
\nonumber U^{\Delta t}K_{Q} f(x) &= \int_{X} k_{Q}(\Phi^{\Delta t}(x),y)f(y)\, d\mu(y) \\
\nonumber &= \int_{X} k_{Q}(\Phi^{\Delta t}(x),\Phi^{\Delta t} (y))f(\Phi^{\Delta t} (y)) \, d\mu(y) \\
&= \int_{X} \left[k_{Q}(x,y) + O(Q^{-1}) \right] (U^{\Delta t}f)(y) \, d\mu(y). \\ 
\end{split}\]
Note that we have used the fact that $ \mu $ is an invariant measure in the second-to-last line. Since $ k_{Q} $ is continuous, it follows from the Cauchy-Schwarz inequality that $ \|K_{Q}f\|_{C^0} \leq \|K_{Q}\|_{C^0}\|f\|_{L^2} $. Substituting this result in the right-hand side and taking the supremum over $x\in X$ yields
\[\left\|(U^{\Delta t}K_{Q} - K_{Q} U^{\Delta t} ) f \right\|_{C^0} = O(Q^{-1}) \left\|f \right\|_{L^2}.\]
Claim~(i) then follows from the fact that $ \lVert \cdot \rVert_{L^2} \leq \lVert \cdot \rVert_{C^0} $. Claims~(ii) and~(iii) can be proved in a similar manner. \qed
\end{proof}

\section{Galerkin approximation of Koopman eigenvalue problems} \label{sect:Galerkin}

In this section, we formulate a Galerkin method for the eigenvalue problem of the Koopman generator $ V $ in the eigenbasis of $ P$, under the implicit assumption that the latter operator is available to us from $P_{Q}$ after having taken a large number of delays $Q$. The task of finding the eigenvalues of $V$ has two challenges, namely, (i) $V$ is an unbounded operator defined on a proper subspace $ D( V ) \subset L^2(X,\mu)$ which is not known a priori; (ii) the spectrum of $V$ could be dense in $i \real$ (even for a pure point spectrum system such an ergodic rotation on $ \mathbb{ T }^D $ with $ D \geq 2 $; e.g.,~\cite{Giannakis17}, Remark~8), in which case, solving for its eigenvalues is a numerically ill-posed problem. Following \cite{GiannakisEtAl15,Giannakis17}, we will address these issues by employing a Galerkin scheme for the eigenvalue problem of $ V $, with a small amount of judiciously constructed diffusion added for regularization. Throughout this section, we consider that Assumptions~\ref{asmptn:standing}, \ref{asmptn:ker}, and \ref{asmptn:C0_ae} hold. Further, we assume the following.
\begin{Assumption}\label{asmptn:SPD}
The kernels $k_Q$, and thus $k_\infty$, are symmetric positive-definite. That is, (i) $k_Q(x,y) $, for every $x,y \in M$; (ii) for every $x_0, x_1, \ldots, x_n \in M $ and $ c_0, c_1, \ldots, c_n \in \mathbb{C} $, $ \sum_{i,j=0}^{n-1} c_i^* k_Q(x_i,x_j) c_j \geq 0$; and (iii) the analogous conditions hold for $k_\infty$.
\end{Assumption}
Our approach has the following steps.

\paragraph{Step 1. Sobolev spaces.} We first construct subspaces of $L^2$ in which we search for eigenfunctions. These spaces will be shown to be dense in $\Range $, defined as the closed subspace of $ \overline{\ran P}$ orthogonal to constant functions (that is, $ \Range $ only consists of zero-mean functions). Note that $ \{ \phi_j \}_{j \in J} $, where $J $ is an index set for the nonzero eigenvalues $\lambda_j$ of $P$, strictly less than 1, is an orthonormal basis of $H$. For any $ p \geq 0 $, we define
\begin{equation}\label{eqHP}
H^p = \left \{\sum_{j\in J} c_j\phi_j \in\Range : \sum_{j\in J} \lvert c_j \rvert^2 \lvert \eta_j\rvert ^{p}<\infty \right\}, \quad \eta_j = ( \lambda^{-1}_j - 1) / ( \lambda_1^{-1} - 1 ).
\end{equation}

The spaces $H^p $ are analogous to the usual Sobolev spaces associated with self-adjoint, positive semidefinite, unbounded operators with compact resolvents and discrete spectra (here, $\{ \eta_j \}_{j\in J} $). In particular, when $(X,g)$ is a smooth Riemannian manifold with a metric tensor $ g $ satisfying $ \vol_g = \mu $, and $ ( \eta_j, \phi_j ) $ are the eigenvalues and orthonormal eigenfunctions of the corresponding Laplace-Beltrami operator, then $ H^p $ becomes the canonical Sobolev space $ H^p( X, g) $, restricted to be orthogonal to constant functions. $H^p$ from~\eqref{eqHP} is a Hilbert space with the inner product
\[\langle f,g \rangle_{H^p}:= \sum_{q=0}^{p} \sum_{j \in J} c_j^* d_j\lvert \eta_j \rvert^{q}, \]
where $ f=\sum_{j \in J}c_j\phi_j $ and $ g=\sum_{j \in J}d_j\phi_j $. Moreover, $\{ \phi_j^{(p)} \}_{j\in J} $ with $ \phi_j^{(p)} = \phi_j/ \lVert \phi_j \rVert_{p}$, $\lVert \phi_j \rVert_{H^p}^2 = \sum_{q=0}^p \lambda_j^q$, forms an orthonormal basis of $H^p$. 

\begin{prop}\label{prop:e}
For every $p>0$, the space $H^p$ is dense in $\Range$ and moreover, the inclusion map $H^p\to\Range$, and thus $H^p \to L^2(X,\mu) $, is compact.
\end{prop}
\begin{proof}
To see that $H^p$ is dense, note that $H^p$ includes all finite linear combinations of the $\phi_j$. Since the $\phi_j$ are an orthonormal basis of $\Range$, these finite linear combinations are dense in $\Range$. Next, the embedding of $H^p$ in $\Disc$ can be represented by a diagonal operator $G : H^p \to \Range$ such that $G_{jj} := \langle \phi_j, G \phi_j^{(p)} \rangle = \eta_j^{-p/2}$. This operator is compact iff $ G_{jj} $ converges to $0$ as $ j \to \infty$. This is true by \eqref{eqn:IntOp_eigen_decay} below. The compactness of the inclusion $ H^p \to L^2(X,\mu) $ follows immediately. \qed
\end{proof}


\paragraph{Step 2. Regularized generator. } For every $\theta>0$, we define the unbounded operators $\Delta : D(\Delta ) \to H $ and $ L_{\theta} : D( L_{\theta}) \to H $, \blue{where $ D(\Delta) = D(L_\theta) \subset D(V) $}, and
\begin{equation}\label{eqn:regularize}
    \Delta := f\mapsto \sum_{k=1}^{\infty} \eta_j \langle \phi_j, f \rangle \phi_j, \quad L_{\theta} := V\rvert_{D(\Delta)} - \theta \Delta. 
\end{equation} 
As we will see in Step~3 below, the role of the diffusion term $\theta \Delta $ is to penalize the eigenfunctions of $V$ with large eigenvalues of a Dirichlet energy functional. Theorem~\ref{thm:C} below identifies a domain in which the operators in \eqref{eqn:regularize} are continuous, and establishes that the eigensolutions of $L_{\theta}$ converge to eigensolutions of $V$ as $\theta\to 0$. 
\begin{thm}\label{thm:C}
Viewed as operators from $H^2 $ to $H $, the generator $V$, as well as the operators $L_{\theta}$ and $\Delta $ from \eqref{eqn:regularize}, are bounded. In particular, we can set $ D( \Delta ) = D( L_{\theta} ) = H^2 $. \blue{Finally, for every eigenvalue $ i \omega $ of $ V $, whose corresponding eigenspace lies in $H^2$, there exists an eigenvalue $\eta$ of $\Delta$ such that the smooth curve $\theta \mapsto \gamma_{\theta} := i \omega - \theta \eta$ consists of eigenvalues $\gamma_{\theta}$ of $L_{\theta}$, converging to $i\omega$ as $\theta \to 0^+$.}
\end{thm}
\begin{proof} 


    First, by Corollary~\ref{corBlock}, we can consider that the basis $ \{ \phi_j \}_{j\in J} $ of $H$ consists of simultaneous eigenfunctions of $V$ and $P$ (and thus $\Delta$), without loss of generality. Then, to verify that $V$ is a bounded operator on $H^2$, first observe that $ \omega_j = O(j) $, which follows from the fact that the eigenvalues of $V$ are integer linear combinations of finitely many rationally independent frequencies (by Assumption~\ref{asmptn:C0_ae}; see also Section~\ref{sect:theory}). By Theorem \ref{thm:D}, the kernel $p$ associated with $P $ is $L^2$ integrable, and thus by a result of Ferreira and Menegatto on integral operators (\cite{EigenIntOp2}, Corollary~2.5), $\lambda_j=o(j^{-1}) $. Combining these estimates, we obtain
\begin{equation}\label{eqn:IntOp_eigen_decay} 
j=o(\eta_j), \quad \omega_j=o(\eta_j),\quad \eta_j^{-1} = o(j^{-1}),
\end{equation}
and therefore deduce that there exists a constant $C>0$ such that 
\begin{equation}
\label{eqOmegaEta}
\omega_j \leq C \eta_j; \quad \forall j\in J.
\end{equation}
Hence, for $f=\sum_{j \in J}c_j \phi_j\in H^2$,
\[\left \lVert Vf\right\rVert^2 = \left\lVert \sum_{j \in J}c_j V \phi_j\right\rVert^2 = \left\lVert\sum_{j \in J} ic_j \omega_j\phi_j\right\rVert^2 \leq C^2\sum_{j \in J} \lvert c_j \rvert^2 \lvert \eta_j|^2 \leq C^2\lVert f\rVert^2_{H^2},\]
proving that $V$ is a bounded operator on $H^2$. The same reasoning applies for $L_{\theta}$ and $\Delta $.
Finally, to establish convergence of the eigenvalues of $L_\theta $ to those of $V\rvert_{H^2} $, let $i \omega_j $ be the eigenvalue of $V$ corresponding to $\phi_j$. Then, by definition of $ L_\theta $ and the basis $\{ \phi_j \}_{j\in J}$, 
\begin{displaymath}
L_{\theta} \phi_j = V \phi_j - \theta \Delta \phi_j = ( i\omega_j-\theta\eta_j) \phi_j, 
\end{displaymath}
and the claim follows immediately. This completes the proof of Theorem~\ref{thm:C}. 
\end{proof}
\begin{rk*} Theorem~\ref{thm:C} establishes that $H^2$ is a domain on which $V$ is a bounded operator, but if $X$ had a smooth manifold structure, it is possible to show that the standard $H^1 $ Sobolev space associated with a Riemannian metric on $ X $ is also a suitable domain. In this work, $X$ has no smooth structure, and we can state Theorem~\ref{thm:C} above only for $ V\rvert_{H^2}$. In separate calculations, we have observed that an analog of the weak eigenvalue problem for $ L_\theta $ formulated in $ H^1 \times H^1 $ actually performs well numerically.
\end{rk*} 

\paragraph{Step 3. Galerkin method.} By virtue of Theorem~\ref{thm:C}, the eigenvalues of $L_{\theta}$ can be considered to be approximations of the eigenvalues of $V$. We will take the Galerkin approach in finding the eigenvalues of $L_{\theta}$ by solving for $z\in H^2$ and $\gamma\in\cmplx$ in the following variational (weak) eigenvalue problem:
\begin{defn}[Regularized Koopman eigenvalue problem]\label{defGalerkin}
Find $ \gamma \in \mathbb{ C } $ and $ z \in H^2 $ such that for all $ f \in H $, 
\begin{displaymath}
A( f, z ) = \gamma \langle f, z \rangle,
\end{displaymath}
where $ A : H \times H^2 \to \mathbb{C } $ is the sesquilinear form defined by 
\begin{displaymath}
A( g, f ) = \langle g, L_{\theta} f \rangle = \langle g, Vf \rangle - \theta E(g,f), \quad E( g,f ) = \langle g, \Delta f \rangle.
\end{displaymath} 
\end{defn}

In the above, the form $E:H \times H^2 \to \mathbb{C} $ induces a Dirichlet energy functional $ E( f ) = E( f, f) $, $ f \in H^2 $, providing a measure of roughness of functions in $ H^2 $. In particular, if $ X $ were a smooth Riemannian manifold, and the \blue{$ ( \eta_j, \phi_j ) $ were set to Laplace-Beltrami eigenvalues and eigenfunctions, respectively,} we would have $ E( f ) = \int_X \lVert \grad f \rVert^2 \, d\mu $. While the lack of smoothness of $ X $ in our setting precludes us from defining $ E $ by means of a gradient operator, its definition in terms of the $\eta_j $ from \eqref{eqHP} still provides a meaningful measure of roughness of functions. For instance, it follows from results in spectral graph theory that the variance of estimates $ \eta_j^{(N)} $ of the $ \eta_j $ computed from finite data sets (e.g., as described in Section~\ref{sect:numerics} ahead) increases with $ k $ \cite{VonLuxburgEtAl08,BerrySauer16}, which is consistent with the intuitive expectation that rough (highly oscillatory) functions require larger numbers of samples for accurate approximations. 

Following \cite{GiannakisEtAl15,Giannakis17}, we will order all solutions $ ( \gamma_j, z_j ) $ of the problem in Definition~\ref{defGalerkin} in order of increasing Dirichlet energy $ E( z_j ) $. Since $ A( f, f ) = - \theta E( f, f) $ by skew-symmetry of $ V $, we can compute the Dirichlet energy of eigenfunction $z_j$ directly from the corresponding eigenvalue, viz.\ $ E( z_j ) = - \Real \gamma_j / \theta $. Similarly, we have $ \omega_j =\Imag \gamma_j $. By~\eqref{eqOmegaEta}, there exist constants $C_1, C_2>0$ such that 
\begin{equation}\label{eqn:theta_eta_ratio_bound}
C_2 \leq \frac{|i\omega_j-\theta \eta_j|}{|\eta_j|} \leq C_1,\ \forall j\in J.
\end{equation}

To justify the well-posedness of the eigenvalue problem in Definition~\ref{defGalerkin}, we will state three important properties of $ A $, namely, 
\begin{equation}\label{eqn:E_bounded_H1}
|A(u,v)|\leq C_1 \|u\|_{H}\|v\|_{H^2}, \quad \forall u\in H, \quad \forall v\in H^2,
\end{equation}
\begin{equation}\label{eqn:sup_A_H}
\sup_{\substack{f\in H\\ \|f\|_{H}=1}}|A(f,v)| \geq C_2\|v\|_{H^2}^2,\ \forall v\in H^2,
\end{equation}
\begin{equation}\label{eqn:sup_A_H2}
\sup_{\substack{g\in H^2\\ \|g\|_{H^2}=1}}|A(u,g)| \geq C_2 \|u\|_{H}^2,\ \forall u\in H.
\end{equation}
We now give brief proofs of these results. In the following, $ v=\sum_{j \in J}d_j \phi_j $ and $u=\sum_{j \in J}c_j\phi_j $ will be arbitrary functions in $ H^2 $ and $ H $, respectively. Moreover, as in the proof of Theorem~\ref{thm:C}, we will assume that the basis $\{ \phi_j \}_{j\in J}$ consists of simultaneous eigenfunctions of $V$ and $\Delta$. First, note that,
\begin{displaymath}
\left| A(u,v) \right| =\left| \sum_{j \in J}(i\omega_j-\theta\eta_j) c_j^* d_j \right| \leq \sum_{j \in J}|i\omega_j-\theta\eta_j | |c_j^* d_j|.
\end{displaymath}
By the Cauchy-Schwartz inequality on $\ell^2 $ and \eqref{eqn:theta_eta_ratio_bound},
\begin{displaymath}
\left| A(u,v) \right| \leq C_1\sum_{j \in J}\lvert\eta_j \rvert \lvert c^*_j d_j \rvert \leq C_1 \lVert u\rVert_{H}\lVert v\rVert_{H^2},
\end{displaymath}
proving~\eqref{eqn:E_bounded_H1}. To prove \eqref{eqn:sup_A_H}, let $f=\sum_{j\in J}a_j\phi_j\in H$. Then, the left-hand side of that equation becomes $\sum_{j \in J}(i\omega_j/\eta_j-\theta) \eta_j a_j^* d_j$. Let $R_j := i\omega_j/\eta_j-\theta $, where $|R_j|\geq C_2$ by~\eqref{eqn:theta_eta_ratio_bound}, By the Cauchy-Schwarz inequality, under the constraint $\sum_{j \in J}|a_j|^2=1$, the sum $\left| \sum_{j \in J} a_j^* \eta_jd_j \right|$ attains the maximum value of $\sum_{j \in J}|\eta_j^2 d_j|^2$. Therefore,
\[\sup_{\substack{f\in H\\ \|f\|_{H}=1}} \left| A(f,v) \right| = \sup_{\sum_{j \in J} |a_j|^2=1} \left| \sum_{j \in J} a_j^* d_j R_j \eta_j \right| \geq C_2 \sum_{j \in J} |\eta_j d_j|^2 = C_2 \|v\|^2_{H^2}.\]
This proves \eqref{eqn:sup_A_H}. The proof of \eqref{eqn:sup_A_H2} is similar to that of \eqref{eqn:sup_A_H}, with $f$ replaced by a trial function $g=\sum_{j \in J}b_j \phi_j\in H^2$ and the constraint $\|g\|_{H^2}^2 = \sum_{j \in J}|b_j|^2 = 1$. A direct consequence of \eqref{eqn:sup_A_H} and \eqref{eqn:sup_A_H2} is,
\begin{equation}\label{eqn:inf_sup_A}
\inf_{\substack{v\in H^2\\ \|v\|_{H^2}=1}} \sup_{\substack{u\in H\\ \|u\|_{H}=1}}|A(u,v)| \geq C_2, \quad \inf_{\substack{u\in H\\ \|u\|_{H}=1}} \sup_{\substack{v\in H^2\\ \|v\|_{H^2}=1}} |A(u,v)| \geq C_2.
\end{equation}

Equations \eqref{eqn:E_bounded_H1}, \eqref{eqn:sup_A_H}, \eqref{eqn:inf_sup_A}, and the compact embedding of $H^2$ in $\Range$ by Proposition~\ref{prop:e} together guarantee that the eigenvalues of $ A $ restricted to the finite-dimensional subspaces of $ H \times H^2 $ spanned by the leading $ m $ eigenfunctions $ \phi_1, \ldots, \phi_{m} $ converge, as $ m \to \infty $, to the weak eigenvalues of $L_{\theta}$. See \cite{BabuskaOsborn91}, Section~8, for an exposition on this classic result. The resulting finite-dimensional Galerkin approximations of the weak eigenvalue problem for $L_\theta$ can be summarized as follows: 
\begin{defn}[Koopman eigenvalue problem, Galerkin approximation] \label{defGalerkin2}Set $ \tilde H_m = \spn \{ \phi_1, \ldots, \phi_m \} $ and $\tilde H^2_m = \spn\{ \phi^{(2)}_1, \ldots, \phi^{(2)}_m \} $, $m \geq 1$. Then, find $ \gamma \in \mathbb{ C } $ and $ z \in \tilde H^2_m $ such that for all $ f \in \tilde H_m $, 
\begin{displaymath}
A( f, z ) = \gamma \langle f, z \rangle,
\end{displaymath}
where the sesquilinear form $ A : H \times H^2 \to \mathbb{C } $ is as in Definition~\ref{defGalerkin}. 
\end{defn}

This problem is equivalent to solving a matrix generalized eigenvalue problem
\begin{equation}
\label{eqGEV}
\boldsymbol{A} \vec c = \lambda \boldsymbol{ B} \vec c,
\end{equation}
where $ \boldsymbol{ A }$ and $ \boldsymbol{ B} $ are $ m \times m $ matrices with elements
\begin{equation}
\label{eqABMat}
\begin{gathered}
A_{ij} = A( \phi_i, \phi_j^{(2)} ) = \frac{ V_{ij} }{ \eta_j } - \theta \Delta_{ij}, \quad
V_{ij} = \langle \phi_i, V \phi_j \rangle, \quad \Delta _{ij} = 
\delta_{ij},\\
B_{ij} = \langle \phi_i, \phi_j^{(2)} \rangle = \eta_i^{-1} \delta_{ij}, 
\end{gathered}
\end{equation}
respectively, and $ \vec c = ( c_1, \ldots, c_m )^\top $ is a column vector in $ \mathbb{ C}^m $ containing the expansion coefficients of the solution $ z $ in the $ \{ \phi_j^{(2)} \} $ basis of $ \tilde H^2_{m} $, viz.\ $ z = \sum_{k=1}^{m} c_j \phi_j^{(2)} $. It is important to note that, unlike the proofs of Theorem~\ref{thm:C} and~\eqref{eqn:E_bounded_H1}--\eqref{eqn:sup_A_H2}, in~\eqref{eqABMat} we do not require that the $ \phi_j $ be simultaneous eigenfunctions of $V$ and $P$. This concludes the description of our Galerkin approximation of the eigenvalue problem for $L_{\theta}$ and therefore for $V$.
\section{Data-driven approximation} \label{sect:numerics}

In this section, we discuss the numeric procedures used to approximate the integral operators described in Sections~\ref{sect:kernel}, \ref{sect:proof_main}, and implement the Galerkin method of Section~\ref{sect:Galerkin} using a finite, time-ordered dataset of observations $\left(F(x_n)\right)_{n=0}^{\Ndata-1}$. In addition, we will prove Theorem~\ref{thm:E}. Throughout this section, we will assume that Assumptions~\ref{asmptn:standing}--\ref{asmptn:C0_ae} hold. In particular, by Assumption~\ref{asmptn:Phys2}, we can assume without loss of generality that the underlying trajectory $\left(x_n\right)_{n=0}^{\Ndata-1}$ starts at a point $ x_0 $ in the compact set $\nbd$ (for, if $ x_0 $ were to lie in $\mathcal{V}\setminus \nbd$, the trajectory would enter $\nbd $ after finitely many steps, and its portion lying in $\mathcal{V}\setminus \nbd$ would not affect the asymptotic behavior of our schemes as $N\to\infty$). Besides this assumption, the trajectory $\left(x_n\right)_{n=0}^{\Ndata-1}$ is assumed to be unknown, and note that it need not lie on $X$.

For the purposes of the analysis that follows, it will be important to distinguish between operators that act on $L^2 $ and $C^0$ spaces. Specifically, to every kernel $ k : M \times M \to \real $ satisfying Assumption~\ref{asmptn:ker}, we will assign a bounded operator $K' : L^2(X,\mu) \to C^0(\nbd)$, acting on $f \in L^2(X,\mu)$ via the same integral formula as in~\eqref{eqn:KOp}, but with the image $K'f $ understood as an everywhere-defined, continuous function on $\nbd$. With this definition, the operator $K : L^2(X,\mu) \to L^2(X,\mu)$ acting on $L^2$ equivalence classes can be expressed as as $ K'' = \iota \circ K' $, where $\iota: C^0(\nbd) \to L^2(X,\mu)$ is the canonical $L^2 $ inclusion map on $C^0(\nbd)$, and we can also define an analog $K'' : C^0(\nbd) \to C^0(\nbd)$ acting on continuous functions via $ K'' = K' \circ \iota$. It can be verified using the Arzel\`a-Ascoli theorem that $K''$ is compact. 

\paragraph{Data-driven Hilbert spaces. } Let $\mu_\Ndata := N^{-1} \sum_{n=0}^{\Ndata-1} \delta_{x_n}$ be the sampling probability measure associated with the finite trajectory $(x_n)_{n=0}^{N-1}$. The compact set $\nbd$ from Assumption \ref{asmptn:Phys2} always contains the support of $\mu_N$. Moreover, since $x_0$ lies in the basin of the physical measure $\mu$, as $\Ndata\to\infty$, $\mu_\Ndata$ converges weakly to $ \mu $, in the sense that 
\begin{equation}
\label{eqWeakConv}
\lim_{N\to\infty} \int_{\nbd} f \, d\mu_\Ndata = \int_X f \, d\mu, \quad \forall f \in C^0(\nbd). 
\end{equation}
Our data-driven analog of the space $L^2(X,\mu)$ will be $L^2(\nbd,\mu_N)$; the set of equivalence classes of complex-valued functions on $ M $ which are square-summable and have common values at the sampled states $ x_n $. Note that $L^2(\nbd,\mu_\Ndata)\cong \cmplx^\Ndata$, and therefore every element $ f \in L^2(\nbd,\mu_\Ndata)$ can be represented in the canonical basis of $ \mathbb{ C }^N $ as an $\Ndata$-vector $\vec f = (f(x_0),\ldots ,f(x_{\Ndata-1}))$. In fact, $L^2(\nbd,\mu_\Ndata)$ is the image of $C^0(\nbd)$ under the restriction map $\pi_\Ndata : C^0(\nbd) \to L^2(\nbd,\mu_\Ndata) $, where $\pi_N f = \left(f(x_0),\ldots,f(x_{\Ndata-1}) \right) $. Moreover, given any $ f, g \in L^2(\nbd,\mu_\Ndata) $, we have $ \langle f, g \rangle_{L^2(\nbd,\mu_\Ndata)} = \vec f \cdot \vec g/ N $, where $ \cdot $ denotes the canonical inner product on $ \mathbb{ C }^N $. 

\paragraph{Kernel integral operators. } In the data-driven setting, given a continuous kernel $k: M \times M \to \real$, we define a kernel integral operator $ K_N : L^2(\nbd, \mu_N) \to C^0(\nbd) $ by (cf.\ \eqref{eqn:KOp})
\begin{displaymath}
K'_N f( x ) = \int_{\nbd} k(x,y) f(y) \, d\mu_N( y ) = \frac{1}{N} \sum_{n=0}^{N-1} k(x,x_n) f(x_n),
\end{displaymath}
and we also set $K_N : L^2(\nbd,\mu_N) \to L^2(\nbd, \mu_N)$ and $ K''_N : C^0(\nbd) \to C^0(\nbd)$ with $K_N = \pi_N \circ K'_N $ and $K''_N = K'_N \circ \pi_N $. Note that $ K_{\Ndata} $ can be represented by an $ N \times N $ matrix $ \boldsymbol{ K } $ with elements $ K_{ij} = k(x_i, x_j ) $. In this representation, the function $ g = K_N f $, $ f \in L^2(\nbd,\mu_N) $, is represented by $ \vec g = \boldsymbol{ K } \vec f $.

When $k=k_{Q}$ from~\eqref{eqKQ2}, one can similarly define operators $K'_{Q,N} : L^2(\nbd,\mu_N) \to C^0(\nbd)$, $ K_{Q,N} : L^2(\nbd,\mu_N) \to L^2(\nbd,\mu_N)$, and $K''_{Q,N} : C^0(\nbd) \to C^0(\nbd)$. This family of operators has the analogous properties to those stated for $K_{Q}$ in Lemma~\ref{lem:Rho}; namely, the functions $\rho_{Q,N} = K''_{Q,N} 1_{\nbd}$ and $\sigma_{Q,N} = K''_{Q,N}( 1/\rho_{Q,N})$ are both continuous, positive, and bounded away from zero on $\nbd$. Therefore, one can define a kernel $p_{Q,N} : M \times M \to \real$ by
\begin{gather*}
\rho_{Q,N} = K''_{Q,N} 1_{\nbd}, \quad \sigma_{Q,N} = K''_{Q,N}( 1/\rho_{Q,N}),\\
p_{Q,N}(x,y) = \frac{k_{Q,N}(x,y)}{\sigma_{Q,N}(x)\rho_{Q,N}(y)}.
\end{gather*}
The kernel $p_{Q,N}$ has the Markov property, i.e., $ \int_{\nbd}p_{Q,N}(x,\cdot) \, d\mu_N = 1 $ for every $x \in M$. Associated to $p_{Q,N}$ are the Markov operators $P'_{Q,N} : L^2(\nbd,\mu_N) \to C^0(\nbd)$ , $P_{Q,N} : L^2(\nbd, \mu_N) \to L^2(\nbd, \mu_N)$ and $P''_{Q,N} : C^0(\nbd) \to C^0(\nbd)$. Moreover, $P_{Q,N}$ is related to the self-adjoint operator $\hat P_{Q,N} : L^2(\mathcal{U,}\mu_N) \to L^2(\nbd, \mu_N) $ with kernel $\hat p_{Q,N} : M \times M \to \real$, 
\begin{equation}\label{eqn:def:hatp}
\hat p_{Q,N}(x,y) = \frac{k_{Q}(x,y)}{\hat{\sigma}_{Q,N}(x)\hat{\sigma}_{Q,N}(y)}, \quad \tilde{\sigma}_{Q,N}=\sigma_{Q,N}/\rho_{Q,N},
\end{equation}
via a similarity transformation analogous to~\eqref{eqSimilarity}. From the kernel $\hat p_{Q,N}$ one can construct the operators $\hat{P}_{Q,N}$, $\hat{P}'_{Q,N}$, and $\hat{P}''_{Q,N}$ as above.

\paragraph{Data-driven basis.} We will use the eigenvectors $ \phi_{j,Q,N} $ of $ \hat P_{Q,N} $ as an orthonormal basis of $ L^2( \nbd,\mu_N) $, and employ the corresponding eigenvalues, $ 1 = \lambda_{0,Q,N} > \lambda_{1,Q,N} \geq \cdots \geq \lambda_{N-1,Q,N} \geq 0 $, to define data-driven analogs 
\begin{equation}
\label{eqEtaQN}
\eta_{j,Q,N} = ( \lambda_{j,Q,N}^{-1} - 1)/ ( \lambda_{1,Q,N}^{-1}- 1), \quad j \in J_N, 
\end{equation}
of the $ \eta_j $ in \eqref{eqHP}, where $J_N = \{ j : \lambda_{j,Q,N} > 0 \} $. The eigenvalue problem for $\hat P_{Q,N}$ is equivalent to a matrix eigenvalue problem for the $N \times N $ symmetric matrix $ \hat{\boldsymbol{ P }} = [ \hat p_{Q,N}( x_i, x_j ) ] $ representing $\hat P_{Q,N}$. Details on the numerical solution of this problem can be found in \cite{Giannakis17,GiannakisDas_tracers}. \blue{Note that for kernels $ k_{Q} $ with exponential decay, such as the Gaussian kernels in~\eqref{eqn:def:dQ},} $\hat{\boldsymbol{P}} $ can be well approximated by a sparse matrix, allowing scalability of our techniques to large $ N $. 

To establish convergence of our schemes in the limit of large data, $N\to \infty $, we would like to establish a correspondence between the eigenvalues and eigenvectors of $\hat P_{Q,N} $ accessible from data and those of $\hat P_{Q}$, but because these operators act on the different spaces, a direct comparison of their eigenvectors is not possible. Therefore, as stated in Section~\ref{secMainResults}, we will first establish a correspondence between the eigenvalues and eigenvectors of $\hat P_{Q,N}$ ($\hat P_{Q})$ and those of $\hat P''_{Q,N}$ ($\hat P''_{Q}$), and show that $\hat P''_{Q,N}$ spectrally converges to $\hat P''_{Q}$. The latter problem is meaningful since both $ \hat P''_{Q,N}$ and $\hat P''_{Q}$ act on $C^0(\nbd)$. 

\begin{lem}\label{lemSpec}
The following correspondence between the spectra of operators holds:
\begin{enumerate}[(i)]
\item ${\lambda}_{j,Q,N}$ is a nonzero eigenvalue of $\hat P_{Q,N}$ iff it is a nonzero eigenvalue of $\hat P''_{Q,N}$. Moreover, if ${ \phi}_{j,Q,N} \in L^2(\nbd,\mu_N)$ is an eigenfunction of $\hat P_{Q,N}$ corresponding to $\lambda_{j,Q,N}$, then $ {\varphi}_{j,Q,N} = {\lambda}_{j,Q,N}^{-1} \hat P'_{Q,N} {\phi}_{j,Q,N} \in C^0(\nbd)$ is an eigenfunction of $\hat P''_{Q,N}$ corresponding to the same eigenvalue.
\item ${\lambda}_{j,Q}$ is a nonzero eigenvalue of $\hat P_{Q}$ iff it is a nonzero eigenvalue of $\hat P''_{Q,N}$. Moreover, if ${ \phi}_{j,Q} \in L^2(\nbd,\mu)$ is an eigenfunction of $\hat P_{Q}$ corresponding to $\lambda_{j,Q}$, then $ {\varphi}_{j,Q} = {\lambda}_{j,Q}^{-1} \hat P'_{Q} {\phi}_{j,Q} \in C^0(\nbd)$ is an eigenfunction of $\hat P''_{Q}$ corresponding to the same eigenvalue.
\end{enumerate}
\end{lem}

Lemma~\ref{lemSpec} \blue{is a direct consequence of the definitions of $\hat{P}_{Q,N}$ and $P_{Q,N}''$. } Next, we establish spectral convergence of $\hat P''_{Q,N}$ to $\hat P''_{Q}$. For that, we will need the following notion of convergence of operators.

\paragraph{Compact convergence.} A sequence of operators $A_n$ on a Banach space $ B $ is said to be compactly convergent to an operator $A$ if $A_n\to A$ pointwise, and for every bounded sequence of vectors $(f_n)_{n\in\num}$, $ f_n \in B $, the sequence $((A-A_n)f_n)_{n\in\num}$ has compact closure. The following proposition states that the data-driven operators $ \hat P_{Q,N} $ converge compactly, and as result in spectrum; for a proof, see \cite{VonLuxburgEtAl08}, Proposition 11, and \cite{wellner2013weak}, Theorem 2.4.1. 

\begin{prop}\label{prop:C0norm_PNQ}
Let Assumptions \ref{asmptn:standing}--\ref{asmptn:SPD} hold. Given a trajectory $ (x_n)_{n\in\num} $ starting in $\basin$, the corresponding sequence of operators $\hat{P}''_{Q,\Ndata}$ constructed from the observations $ F( x_0 ), \ldots, F(x_{N-1})$ converges compactly as $ N \to \infty $ to $\hat{P}''_{Q}$. As a result, the sequence $\hat{P}_{Q,\Ndata}$ converges spectrally, in the sense of Corollary~\ref{corSpectral}, to $\hat{P}_{Q}$. In particular, since the nonzero spectrum of a compact operator only consists of isolated eigenvalues, the convergence holds for all nonzero eigenvalues of $\hat P''_{Q,N}$ and the corresponding eigenspaces.
\end{prop}

\blue{The spectral convergence results above follow from Proposition 13 in \cite{VonLuxburgEtAl08}.} We will now prove Theorem~\ref{thm:E}. Note that there is some similarity between our methods and papers on spectral convergence of kernel algorithms, e.g., \cite{BelkinNiyogi2007,VonLuxburgEtAl08}, but our assumptions distinguishes Theorem~\ref{thm:E} from previously studied cases. In particular, we do not assume an i.i.d.\ sequence of observed quantities, or that the sampled sequence $(x_n)_{n=0}^{N-1}$ lies on the support $X$ of the invariant measure (as assumed in \cite{BelkinNiyogi2007,VonLuxburgEtAl08}). Finally, $X$ need not have a manifold structure (as assumed in \cite{BelkinNiyogi2007} and other manifold learning algorithms).

\paragraph{Proof of Theorem~\ref{thm:E}.} The claims of the theorem follow from analogous results to Lemma~\ref{lemSpec} and Proposition~\ref{prop:C0norm_PNQ} for the operators $P_{Q,N}$, $P'_{Q,N}$, $P''_{Q,N}$ and $P_{Q}$, $P'_{Q}$, $P''_{Q}$. \qed

Together, Lemma~\ref{lemSpec} and Proposition~\ref{prop:C0norm_PNQ} imply that every eigenpair $(\lambda_{j,Q},\phi_{j,Q})$ of $\hat P_{Q}$ can be consistently approximated by a sequence of eigenpairs $({\lambda}_{j,Q,N}, {\phi}_{j,Q,N} ) $ of $ \hat P_{Q,N}$. Moreover, by Corollary~\ref{corSpectral}, as $Q\to\infty$, $(\lambda_{j,Q},\phi_{j,Q})$ approximates in turn the eigenpair $(\lambda_j,\phi_j)$ of $P$; that is,
\begin{equation}\label{eqSpecConv}
\lim_{Q\to\infty}\lim_{N\to\infty} \lambda_{j,Q,N} = \lambda_j, \quad \lim_{Q\to\infty}\lim_{N\to\infty} {\lambda}_{j,Q,N}^{-1} \iota\hat P'_{Q,N} {\phi}_{j,Q,N} = \phi_j,
\end{equation}
where the second limit is taken with respect to the $L^2(X,\mu)$ norm. Since, as can be seen in~\eqref{eqABMat}, the Galerkin scheme in Section~\ref{sect:Galerkin} can be entirely formulated using the $ \lambda_j $ and the matrix elements $\langle \phi_i, V \phi_j^{(2)} \rangle$ of the generator, \eqref{eqSpecConv} indicates in turn that we can construct a consistent data-driven Galerkin scheme if we can consistently compute approximate generator matrix elements using the data-driven eigenfunctions ${\phi}_{j,Q,N}$. To that end, we will employ finite-difference approximations, as described below.

\paragraph{Finite-difference approximation.} The action $ V f $ of the generator on a function $ f \in D( V) $ is defined via the limit in \eqref{eqn:def_gen_flow}. This suggests that for data sampled discretely at sampling interval $ \Delta t $, we can approximate $ V f $ by a finite-difference approximation \cite{Giannakis15,GiannakisEtAl15,Giannakis17}. For example, the following are first- and second-order, approximation schemes for $V$, respectively:
\begin{equation}\label{eqFD}
V_{\Delta t} f = \frac{1}{\Delta t} (U^{\Delta t} f - f), \quad V_{\Delta t} f = \frac{1}{2\Delta t}\left(U^{\Delta t}f-U^{-\Delta t}f\right).
\end{equation}
In the finite-sample case, we approximate $ V_{\Delta t} $ by a corresponding $\pow$-th order finite-difference operator $ V_{\Delta t,\Ndata} : L^2(\nbd,\mu_\Ndata) \to L^2(\nbd,\mu_\Ndata)$. For example, in the case of the first-order scheme in~\eqref{eqFD}, $V_{\Delta t, N}$ becomes
\begin{equation}\label{eqFDN}
V_{\Delta t,N} f( x_n ) = \frac{f(x_{n+1}) - f(x_{n})}{\Delta t}, \quad n \in \{ 0, \ldots, N-2 \}, 
\end{equation}
and $V_{\Delta t,N} f(x_{N-1}) = 0$. To ensure that the approximations $V_{\Delta t, N} f $ converge to the true function $Vf$ for a class of functions of sufficient regularity, the following smoothness conditions are sufficient:
\begin{Assumption} \label{asmptn:Cr}
\blue{$\nbd$ is a $C^{1+\alpha}$ compact manifold for some $\alpha>0$, and $\Phi^t\rvert_{\nbd}$ is generated by a $C^{\alpha}$ vector field $\vec V$. Moreover, $F\rvert_\nbd\in C^{1+\alpha}(\nbd;\real^\dimObs)$, and the kernel shape function $h:\real\to\real$ is $C^{1+\alpha}$. $V_{\Delta t}$ and $V_{\Delta t,N}$ are first-order finite difference schemes, as in~\eqref{eqFD} and~\eqref{eqFDN}, respectively}.
\end{Assumption}

Under Assumption~\ref{asmptn:Cr}, the flow $\Phi^t$ is generated by a $C^\alpha$ vector field $ \vec V : C^1(\mathcal{U}) \to C^0(\mathcal{U})$, and the generator $V$ of the Koopman group is an extension of the latter differential operator. Moreover, we can approximate $ \vec V $ by finite-difference schemes $ \vec V_{\Delta t} : C^0(\mathcal{U}) \to C^0(\mathcal{U})$, defined analogously to~\eqref{eqFD} with $ U^{\Delta t}$ replaced by $ \Phi^t$. We then have:

\begin{prop}\label{prop:Fin_diff}
Let Assumptions \ref{asmptn:standing}, \ref{asmptn:ker}, and~\ref{asmptn:Cr} hold. Then for every $i,j\in\num$: 
\begin{enumerate}[(i)]
\item The eigenfunctions $\varphi_{j,Q,N} $ and $\varphi_{j,Q} $ from Lemma~\ref{lemSpec} lie in $C^{1+\alpha}(\nbd)$. Moreover, as $ \Delta t \to 0 $, 
\begin{displaymath}
\vec V_{\Delta t}\varphi_{j,Q} =\vec V\varphi_{j,Q} + \|\phi_{j,Q}\|_{C^{1+\alpha}(\nbd)} O(\Delta t)^{\alpha} , 
\end{displaymath}
where the estimate holds uniformly on $\nbd$.
\item $\lim_{\Delta t\to 0}\lim_{\Ndata\to\infty}\langle\hat{\phi}_{i}, V_{\Delta t,\Ndata} \hat{\phi}_{j}\rangle_{L^2(\nbd,\mu_N)} = \langle \phi_{i,Q}, V\phi_{j,Q} \rangle$.
\end{enumerate}
\end{prop}
\begin{proof} To prove Claim (i), note that under Assumption \ref{asmptn:Cr}, for a finite number of delays $Q$, by \eqref{eqn:def:dQ}, $\hat p_{Q}$ is a $C^{1+\alpha}$-smooth kernel. Hence, according to \cite{EigenIntOp2}, the ranges of the integral operators $\hat P'_Q$ and $ \hat P'_{Q,N}$, and thus $\varphi_{j,Q}$ and $\varphi_{j,Q,N}$ lie in $ C^{1+\alpha}(\mathcal{U})$. Since the vector field $ \vec V $ is $C^\alpha$, the trajectories are $C^{1+\alpha}$-smooth, and therefore, $\vec V\varphi_{j,Q}$, which is the time derivative along the orbit, has a first-order Taylor expansion. The first-order finite-difference scheme gives the $\|\varphi_{j,Q}\|_{C^{1+\alpha}(\nbd)} O(\Delta t)^{\alpha}$ error.
Claim~(ii) is a consequence of Claim~(i), in conjunction with the weak convergence of measures in~\eqref{eqWeakConv} and Lemma~\ref{lemSpec}. \qed
\end{proof}

\begin{rk*} \blue{In many cases, such as flows induced on inertial manifolds in dissipative PDEs \cite{ConstantinEtAl1989}, the $C^{1+\alpha}$ regularity in Assumption~\ref{asmptn:Cr} cannot be strengthened. Proposition \ref{prop:Fin_diff} provides the basis for numerically approximating $V$ for these cases. If $M$, $\nbd$, $\vec V $, $F$ and $h$ have a higher degree of smoothness, say $C^r$ for some $r\geq 2$, then taking $V_{\Delta t}$ to be an $r$-th order finite-difference scheme would lead to an improved, $O(\Delta t)^{r}$, convergence.}
\end{rk*}

\paragraph{Data-driven Galerkin method.} Using the $ {\eta}_{j,Q,N} $ from~\eqref{eqEtaQN}, we define the data-driven normalized basis vectors $ \hat{\phi}_{j}^{(p)} = \phi_{j,Q,N}/ \hat{\eta}^{p/2}_{j,Q,N} $, $ j \in J_N$ (cf.\ the $ \phi_j^{(p)} $ from Step~1 in Section~\ref{sect:Galerkin}), and the associated Galerkin approximation spaces $ H^p_{N,Q,m} = \spn \{ \hat{\phi}^{(p)}_{j} \}_{j=1}^m \subseteq L^2(\mathcal{U},\mu_N) $, $m\leq J_N$, where we abbreviate $H^p_{Q,N,J_N} =: H^p_{Q,N} $ and $H^0_{Q,N} =: H_{Q,N}$. We also define the positive semidefinite, self-adjoint operator $ \Delta_{Q,N} : H_{Q,N} \to H_{Q,N} $, where
\begin{displaymath}
\Delta_{Q,N} f = \sum_{j\in J_N} {\eta}_{j,Q,N} c_j {\phi}_{j,Q,N}, \quad f = \sum_{j=0}^{N-1} c_j{\phi}_{j,Q,N}. 
\end{displaymath}
This operator is a data-driven analog of $\Delta $ in~\eqref{eqn:regularize}. With these definitions and the finite-difference approximation of $ V $ described above, we pose the following data-driven analog of the Galerkin approximation in Definition~\ref{defGalerkin2}:
\begin{defn}[Koopman eigenvalue problem, data-driven form]
\label{defDataDrivenGalerkin}
Find $ \gamma \in \mathbb{ C } $ and $ z \in H^2_{Q,N,m} $ such that for all $ f \in H_{Q,N,m} $, 
\begin{displaymath}
A_{\Delta t, Q,N}( f, z ) = \langle f, z \rangle_{L^2(\nbd,\mu_N)},
\end{displaymath}
where $ A_{\Delta t, Q, N} : H_{Q,N} \times H_{Q,N} \to \mathbb{ C } $ is the sesquilinear form defined as
\begin{displaymath}
A_{\Delta t,Q,N}( f, z ) = \langle f, V_{\Delta t,\Ndata} z \rangle_{L^2(\nbd,\mu_N)} - \theta \langle f, \Delta_{Q,N} z \rangle_{L^2(\nbd,\mu_N)}.
\end{displaymath}
\end{defn}

Numerically, this is equivalent to solving a matrix generalized eigenvalue problem analogous to that in~\eqref{eqGEV}, viz.
\begin{displaymath}
\boldsymbol{A} \vec c = \lambda \boldsymbol{ B} \vec c,
\end{displaymath}
where $ \boldsymbol{ A }$ and $ \boldsymbol{ B} $ are $ m \times m $ matrices with elements
\begin{gather*}
A_{ij} = A_{\Delta t, Q,N}( \phi_{i,Q,N}, \phi_{j,Q,N}^{(2)} ) = \frac{ V_{ij} }{ \eta_{j,Q,N} } - \theta \Delta_{ij}, \\
V_{ij} = \langle {\phi}_{i,Q,N}, V_{\Delta t,N} {\phi}_{j,Q,N} \rangle_{L^2(\nbd,\mu_N)}, \quad \Delta _{ij} = 
\delta_{ij},\\
B_{ij} = \langle {\phi}_{i,Q,N}, {\phi}_{j,Q,N}^{(2)} \rangle_{L^2(\nbd,\mu_N)} = {\eta}_{i,Q,N}^{-1} \delta_{ij}, 
\end{gather*} 
respectively, and $ \vec c = ( c_1, \ldots, c_m )^\top $ is a column vector in $ \mathbb{ C}^m $ containing the expansion coefficients of the solution $ z = \sum_{j=1}^{m} c_j \hat{\phi}_{j,Q,N}^{(2)} $ in the $ \{ \hat{\phi}_{j}^{(2)} \} $ basis of $ H^2_{Q,N,m} $. \blue{Analogously to the continuous case, we define a data-driven Dirichlet energy functional $ E_{Q,N} $ on $H^2_{Q,N}$, given by $E_{Q,N}(f) = \langle f, \Delta_{Q,N} f \rangle_{L^2(\mathcal{U},\mu_N)}$, and use that functional to order the computed eigenfunctions in order of increasing Dirichlet energy. Note that, unless an antisymmetrization is explicitly performed, in the data driven setting, $V_{ij} $ will generally not be equal to $-V_{ji} $, and thus $ \Real \gamma $ will not be equal to $ -\theta E_{Q,N}(z)$ (cf.\ Section~\ref{sect:Galerkin}). Nevertheless, in practice we observe that $ \Real \gamma \approx - \theta E_{Q,N}(z) $, at least for the leading eigenfunctions.} 

For any fixed $ m $, and up to similarity transformations, the matrices $ \boldsymbol{A} $ and $ \boldsymbol{B} $ converge in the limits $ Q \to \infty $, after $\Delta t \to 0 $, after $N\to\infty $ (in that order) to the corresponding matrices in the variational eigenvalue problem in~\eqref{eqGEV}. We therefore conclude that the data-driven Galerkin method in Definition~\ref{defDataDrivenGalerkin} is consistent (as $\Delta t \to 0 $ and $Q,N\to\infty$) with the Galerkin method in Definition~\ref{defGalerkin2}, which is in turn consistent (as $m\to\infty$) with the weak eigenvalue problem for the regularized generator $L_\theta$ in Definition~\ref{defGalerkin}.

\section{Results and discussion}\label{sect:examples}

In this section, we apply the methods described in Sections \ref{sect:kernel}--\ref{sect:numerics} to two ergodic dynamical systems with mixed spectrum, constructed as products of either a mixing flow on the 3-torus, or the L63 system, with circle rotations. Our objectives are to demonstrate that (i) the results of Theorem~\ref{thm:A} and Corollaries~\ref{corSpectral} and~\ref{thm:B} hold, that is, the eigenfunctions $\hat \phi_j$ of $P_{Q,N}$ from \eqref{eqn:def_KHP_Q} are eigenfunctions of $U^t$; (ii) the eigenvalues obtained using the Galerkin scheme in Definition~\ref{defDataDrivenGalerkin} are consistent with those expected theoretically.

\subsection{\label{secMixedSpec}Two systems with mixed spectrum} 

The first system studied below is based on a strongly mixing flow on the 3-torus introduced by Fayad \cite{Fayad02}. The flow, denoted by $\Phi_{\mathbb{T}^3}^t$, is given by the solution of the ordinary differential equation (ODE) $ d( x,y, z ) / dt = \vec V( x,y,z) $, where $ ( x, y, z ) \in \mathbb{ T}^3 $, and $ \vec V $ is the smooth vector field
\begin{equation}\label{eqn:3DTorus}
\vec V( x, y, z ) = \vec \nu / \varphi( x, y, z ), \quad \varphi(x,y,z) = 1+\sum_{k=1}^{\infty}\frac{e^{-k}}{k} \Real\left[ \sum_{|l|\leq k} e^{ik(x+y)+il z} \right],
\end{equation}
parameterized by the constant frequency vector $ \vec \nu $. Hereafter, we set $\vec{\nu}=(\sqrt{2},\sqrt{10},1)^\top$. Note that the orbits under $ \Phi_{\mathbb{T}^3}^t $ are the same as that of the ergodic, non-mixing linear flow with constant vector $\vec{\nu}$. $ \Phi_{\mathbb{T}^3}^t $ has a unique Borel, invariant, ergodic probability measure $\mu$ with density $ \varphi / \int_M \varphi \, d\Leb $ relative to Lebesgue measure. Such flows are also called reparameterized flows as $ \vec \nu $ is scaled by the function $\varphi $ at each point $(x,y,z)\in\TorusD{3}$.

This system is strongly mixing with respect to its invariant measure, i.e., its generator has purely continuous spectrum \cite{Fayad02}. To construct an associated mixed-spectrum system, we take the product $ \Phi_{\mathbb{T}^3}^t\times \Phi_\omega^t $ with a periodic flow $\Phi_\omega^t$ on $S^1$, defined as
\begin{equation}\label{eqn:1D_oscill}
d \Phi_\omega^t(\alpha)/dt = \omega, \quad \omega=1.
\end{equation}
Thus, the state space of the product system is $ M = \mathbb{ T}^3 \times S^1 = \mathbb{ T }^4 $. Note that in this example the attracting set $ X $ is smooth and coincides with the state space, $ M = X $; in particular, all states sampled experimentally lie exactly on $ X $. \blue{Moreover, the Koopman generator $V: D(V) \to L^2( X, \mu ) $ is a skew-adjoint extension of the differential operator $ \vec V \oplus \vec\omega : C^\infty( X ) \to L^2( X, \mu ) $, where $\vec\omega : C^{\infty}(S^1) \to C^\infty(S^1)$ is the differential operator $ f \mapsto \vec \omega( f ) := \omega f'$.} Since $ \Phi_\omega^t $ has a pure point spectrum consisting of integer multiples of $ i \omega $ and $ \Phi_{\mathbb{T}^3}^t $ has no eigenvalues, the discrete spectrum of the product system is $ \{ i k \omega, k \in \mathbb{ Z} \} $.

The second system that we study is based on the L63 system \cite{Lorenz63}. This system is known to have a chaotic attractor $ X_\text{Lor} \subset \real^3 $ with fractal dimension $2.0627160$ \cite{LorentzFract}, supporting a physical invariant measure \cite{Tucker99} with \blue{a compact absorbing ball \cite{LawEtAl14} and mixing dynamics \cite{LuzzattoEtAl05}. The latter, implies that the generator $V$ of the system has only constant eigenfunctions, corresponding to eigenvalue $0$.} The flow, denoted by $\Phi^t_\text{Lor} $, is generated by a smooth vector field $ \vec V \in C^\infty( \real^3; \real^3 ) $, whose components at $(x,y,z)\in \real^3$ are
\begin{equation}\label{eqn:l63}
V^{(x)} = \sigma(y-x), \quad V^{(y)} = x(\rho-z) -y, \quad V^{(z)} = xy-\beta z.
\end{equation}
Throughout, we use the standard parameter values $\beta = 8/3$, $\rho = 28$, $\sigma = 10$. As in the torus case, we form the product $ \Phi^t_\text{Lor} \times \Phi_\omega^t $ with the rotation $ \Phi_\omega^t $ in~\eqref{eqn:1D_oscill}, leading to a mixed spectrum system with the same discrete spectrum $ \{ i k \omega, k \in \mathbb{ Z} \} $. Note that unlike the torus-based system, the attracting set $ X = X_\text{Lor} \times S^1 $ is a strict subset of the state space $ M = \real^3 \times S^1 $.

For each product system, we define a continuous map $ F : M \to \mathbb{R }^3 $ coupling the degrees of freedom of the continuous-spectrum subsystem with the rotation. In the case of the torus-based system, we define $ F( x, y, z, \alpha ) = ( F_1, F_2, F_3 ) $, $ ( x, y, z ) \in \mathbb{ T}^3 $, $ \alpha \in S^1 $, via additive coupling, viz.
\begin{equation}
\label{eqn:4D_additive} F_1 = \sin\alpha+ \sin x, \quad F_2 = \cos\alpha+ \sin y, \quad F_3 = \sin(2\alpha) + \sin z.
\end{equation}
In the case of the L63-based system, the coupling is nonlinear with $ F( x, y, z, \alpha )= ( F_1, F_2, F_3 ) $, $ ( x, y, z ) \in \real^3 $, $ \alpha \in S^1 $, and 
\begin{equation}
\label{eqn:l63_skew_nonlinear}F_1 = \sin(\alpha + x), \quad F_2 = \cos(2\alpha + y), \quad F_3 = \cos(\alpha + z).
\end{equation}

\begin{figure}
\centering
\includegraphics[width=.44\textwidth]{\Path 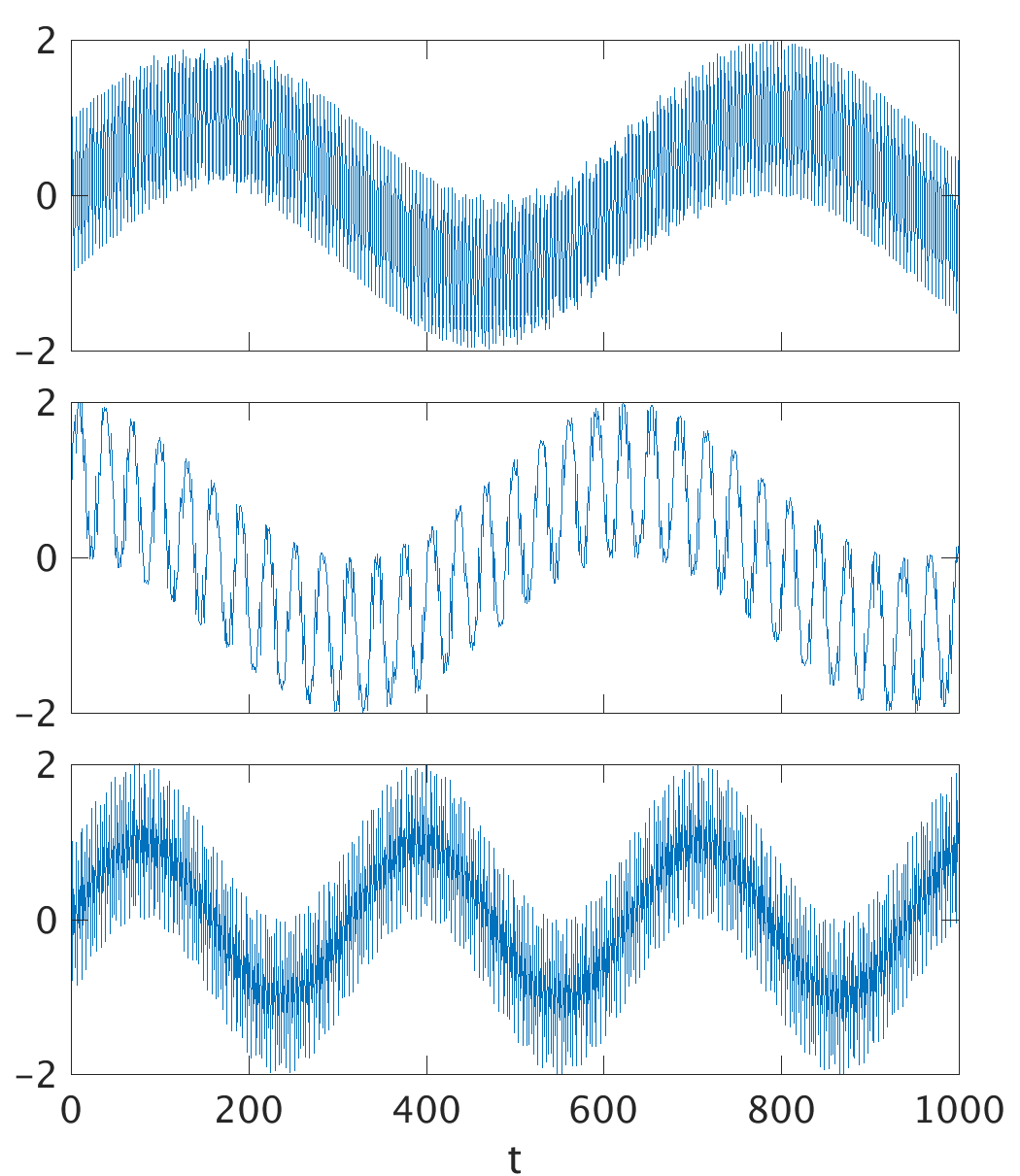}
\includegraphics[width=.52\textwidth]{\Path 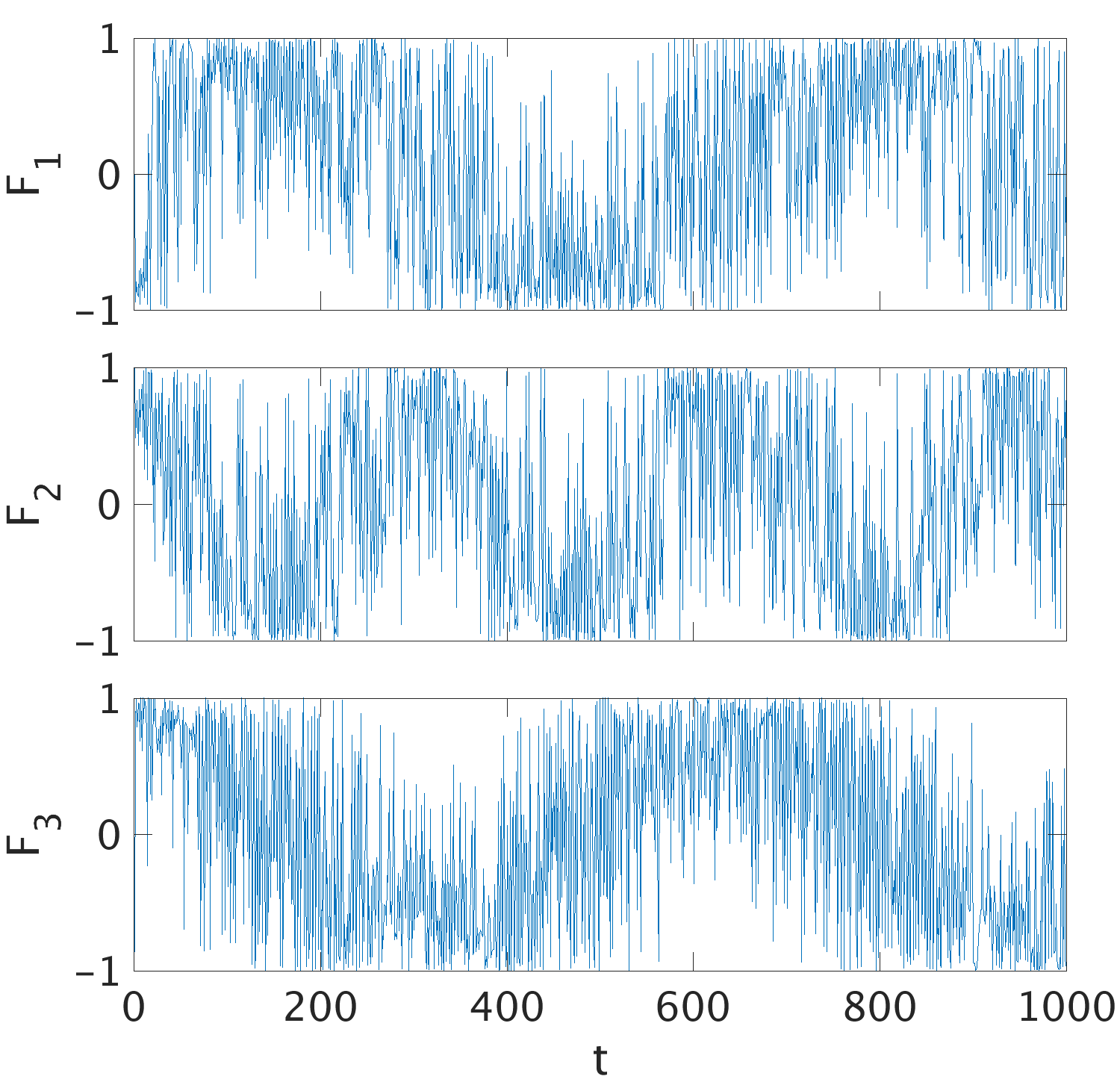}
\caption[Timeseries]{Time series of the observation maps of the torus-based system \eqref{eqn:4D_additive} (left) and the L63-based system \eqref{eqn:l63_skew_nonlinear} (right). Each time series is a nonlinear combination of data generated from two sources, one with a purely continuous spectrum, \eqref{eqn:l63} and \eqref{eqn:3DTorus}, respectively, and one with a purely discrete spectrum, \eqref{eqn:1D_oscill}. \blue{The time series clearly exhibit complex evolution, characteristic of chaotic dynamics, and recovering from them Koopman eigenvalues and eigenfunctions is a non-trivial task.}} 
\label{fig:time_series}
\end{figure}

\subsection{\label{secResults}Experimental results}

We generated numerical trajectories $ x_0, x_1, \ldots, x_{N-1} $ of the torus- and L63-based systems described above starting in each case from an arbitrary initial condition $ y \in M $. In the torus experiments, the system is always on the attractor, so the starting state $ x_0 $ in the training data was set to $ y $. In the L63 experiments, we let the system relax towards the attractor, and set $ x_0 $ to a state sampled after a long spinup time (4000 time units); that is, we formally assume that $ y $ (and therefore $ x_0 $) lie in the basin $\basin$ of the physical measure associated with $ X $. In both cases, the number of samples was $N=\text{50,000}$, the integration time-step was $0.01$, and the number of delays was $Q=2000$. Gaussian kernels $k_Q $ from~\eqref{eqn:def:kQ} were used throughout. We employed the \texttt{ode45} solver of Matlab to compute the trajectories and generated time series $ F( x_0 ), F( x_1 ), \ldots, F( x_{N-1} ) $ by applying the observation maps in~\eqref{eqn:4D_additive} and~\eqref{eqn:l63_skew_nonlinear} to the respective states $ x_n$. Portions of the observable time series from each system are displayed in Fig.~\ref{fig:time_series}. Note that the $ x_n $ were not presented to our kernel algorithm.

We computed data-driven eigenpairs $ ( \lambda_{j,Q,N},\phi_{j,Q,N}) $ by solving the eigenvalue problem for the operator \blue{$ \hat P_{Q,N} $ from \eqref{eqn:def:hatp} }, using Matlab's \texttt{eigs} iterative solver. Henceforth, for ease of notation, we abbreviate $ \lambda_{j,Q,N}$ and $\phi_{j,Q,N}$ by $\hat \lambda_j$ and $\hat \phi_j$, respectively. The bandwidth parameter $\epsilon$ of the Gaussian kernels was selected using the tuning procedure described in \cite{BerryHarlim16,GraphTomog,BerryEtAl15}, which yielded $\epsilon \approx 3.6$ and $\approx 2.053$ for the torus and L63 systems, respectively. Representative eigenfunctions $ \hat \phi_j $, plotted as time series $ n \to \hat \phi_j( x_n ) $, and the corresponding eigenvalues are displayed in Figs.~\ref{fig:4DTorus_additive} and~\ref{fig:lambda_compare_mixed_spectra}, respectively. We now describe these results in more detail.

\begin{figure}
\centering
\includegraphics[width=.48\textwidth]{\Path 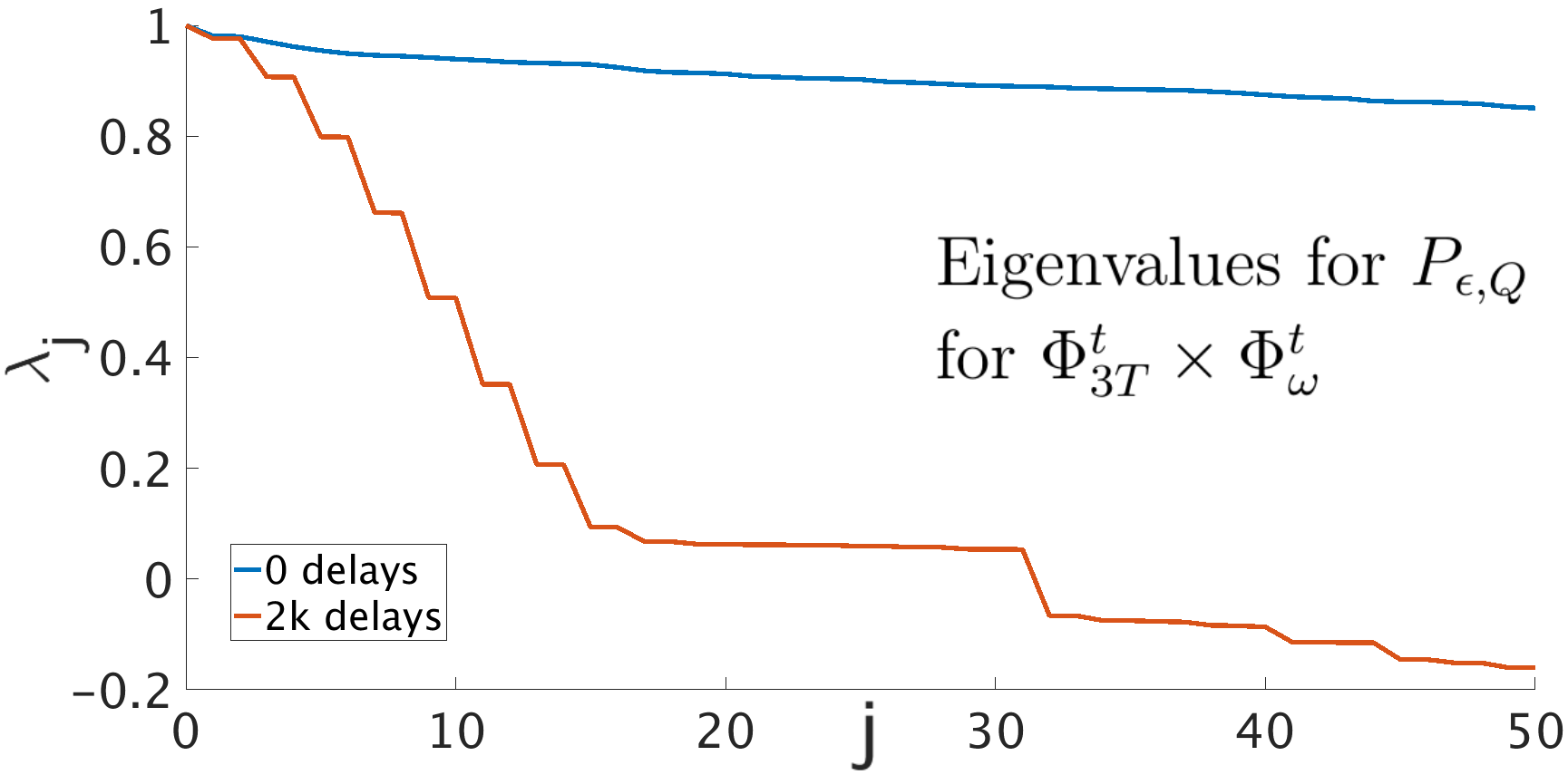}
\includegraphics[width=.48\textwidth]{\Path 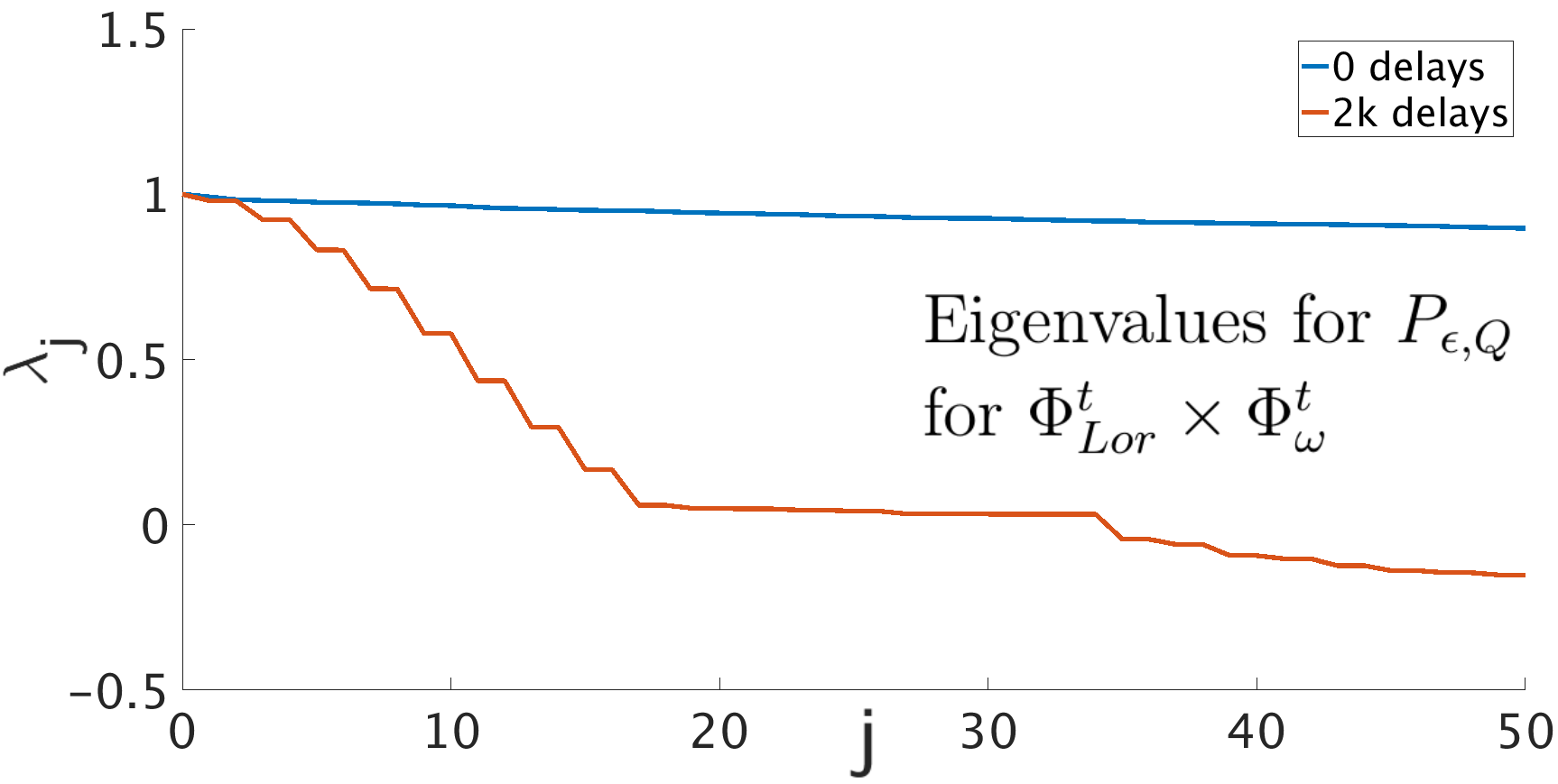}
\caption[Comparison of $lambdas$]{Eigenvalues $ \hat \lambda_j $ of the integral operator $P_{Q,N}$ for representative values of the delay parameter $ Q $ for the torus system in~\eqref{eqn:4D_additive} (left) and the L63-based system in \eqref{eqn:l63_skew_nonlinear} (right). The blue and red lines correspond to no delays ($Q=1$) and 2000 delays, respectively. When $ Q \leq 1 $, the eigenvalues are seen clustering around $1$. The eigenvalues cannot exceed $1$ as $P_{Q,N}$ is a Markov operator. At $Q = 2000 $, the eigenvalues decay more rapidly towards zero and, at least up to eigenvalue 15, have multiplicity 2 as expected from Proposition~\ref{prop:W_lambda}. \label{fig:lambda_compare_mixed_spectra}}
\end{figure}

According to Theorem~\ref{thm:A} and Proposition~\ref{prop:W_lambda}, at large numbers of delays (here, $Q=2000$), the eigenfunctions $ \hat \phi_j $ of $\hat P_{Q,N}$ should form doubly degenerate pairs, and each pair should exhibit a single frequency associated with an eigenvalue of $ V $. \blue{ More precisely, $ \hat \phi_{j} \pm i \hat \phi_{j+1} $} with $ j \in \{ 1, 3, \ldots \} $ should approximate an eigenfunction of $V $. Both systems studied here have exactly one rationally independent eigenvalue $i\omega=i$, so the eigenfunctions of $ P_{Q,N} $ are expected to evolve at frequencies $j\omega$, $j\in\num$. This is evidently the case in the time series plots in Fig.~\ref{fig:4DTorus_additive}. Also, each of the $\hat\phi_j $ has multiplicity $2$ (note that only one eigenfunction from each eigenspace is shown in Fig.~\ref{fig:4DTorus_additive}). The left-hand panels of Fig.~\ref{fig:4DTorus_additive} show a matrix representation of the generator $V$ (approximated via the finite-difference scheme in~\eqref{eqFDN}) in the 51-dimensional data-driven subspace spanned by $ \hat \phi_0, \ldots, \hat \phi_{50} $. The matrix is, to a good approximation, skew-symmetric, consistent with the fact that $V$ is a skew-symmetric operator, and exhibits prominent $ 2 \times 2 $ diagonal blocks associated with the eigenspaces of $ V $ approximated by $ ( \hat \phi_1, \hat \phi_2 ), ( \hat \phi_3,\hat \phi_4 ), \ldots $, in agreement with Corollary~\ref{corBlock}.

Figure~\ref{fig:eigenvalues} shows the approximated eigenvalues $ \gamma_j $ of the regularized generator $ L_\theta $ obtained from this basis using the Galerkin scheme in Definition~\ref{defDataDrivenGalerkin} with the diffusion regularization and spectral order parameters $\theta = 10^{-4} $ and $m = 50 $, respectively. Each plot in Fig.~\ref{fig:eigenvalues} shows the first 20 eigenvalues corresponding to eigenfunctions of increasing Dirichlet energy $ E( z_j ) $ of the corresponding eigenfunction $ z_j $ (recall that $ \Real \gamma_j \approx - \theta E_{Q,N}( z_j) $). According to Section~\ref{sect:Galerkin}, the imaginary parts of the $ \gamma_j $ should approximate the Koopman eigenfrequencies $ j(k) \omega $, where $ j $ is an integer-valued function giving the frequency of the Koopman eigenfunction with the $ k $-th smallest Dirichlet energy. In Fig.~\ref{fig:eigenvalues}, the $ \Imag \gamma_j $ are indeed equal to integer multiples of $ \omega = 1 $ to a good approximation for the first $\NEigen$ eigenvalues (ordered in order of increasing Dirichlet energy). For indices $ k$ close to $m$, the accuracy of the eigenvalues begins to deteriorate. This is due to the facts that (i) even with a ``perfect'' basis $ \{ \hat \phi_j \} $, eigenfunctions of higher Dirichlet energy (and stronger oscillatory behavior) require increasingly higher-order Galerkin approximation spaces; (ii) at finite sample numbers $ N $, the quality of the data-driven elements $ \hat\phi_j $ degrades at large $j $. 

\begin{figure}
\centering
\includegraphics[width=0.49\textwidth]{\Path 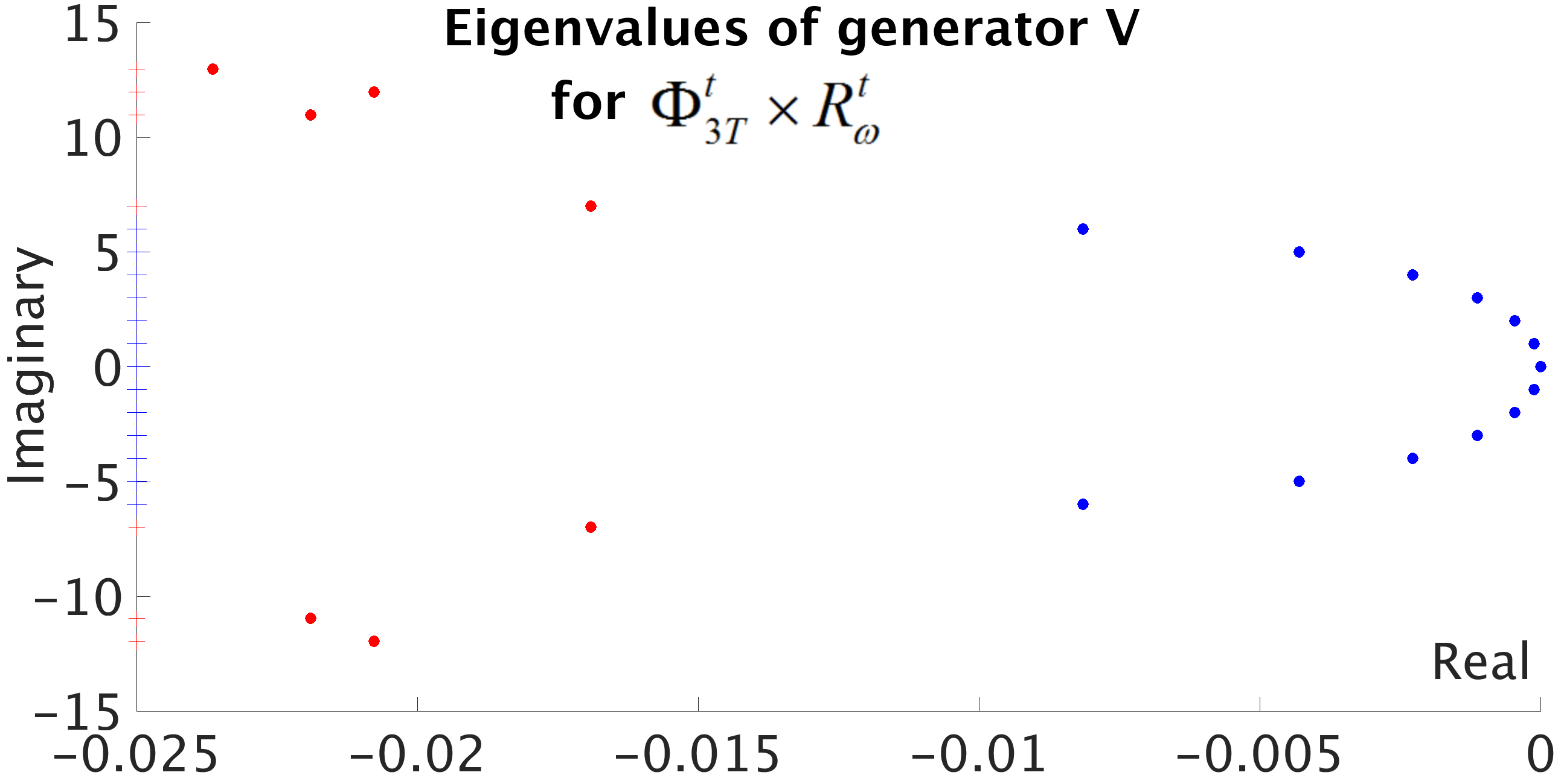}
\includegraphics[width=0.49\textwidth]{\Path 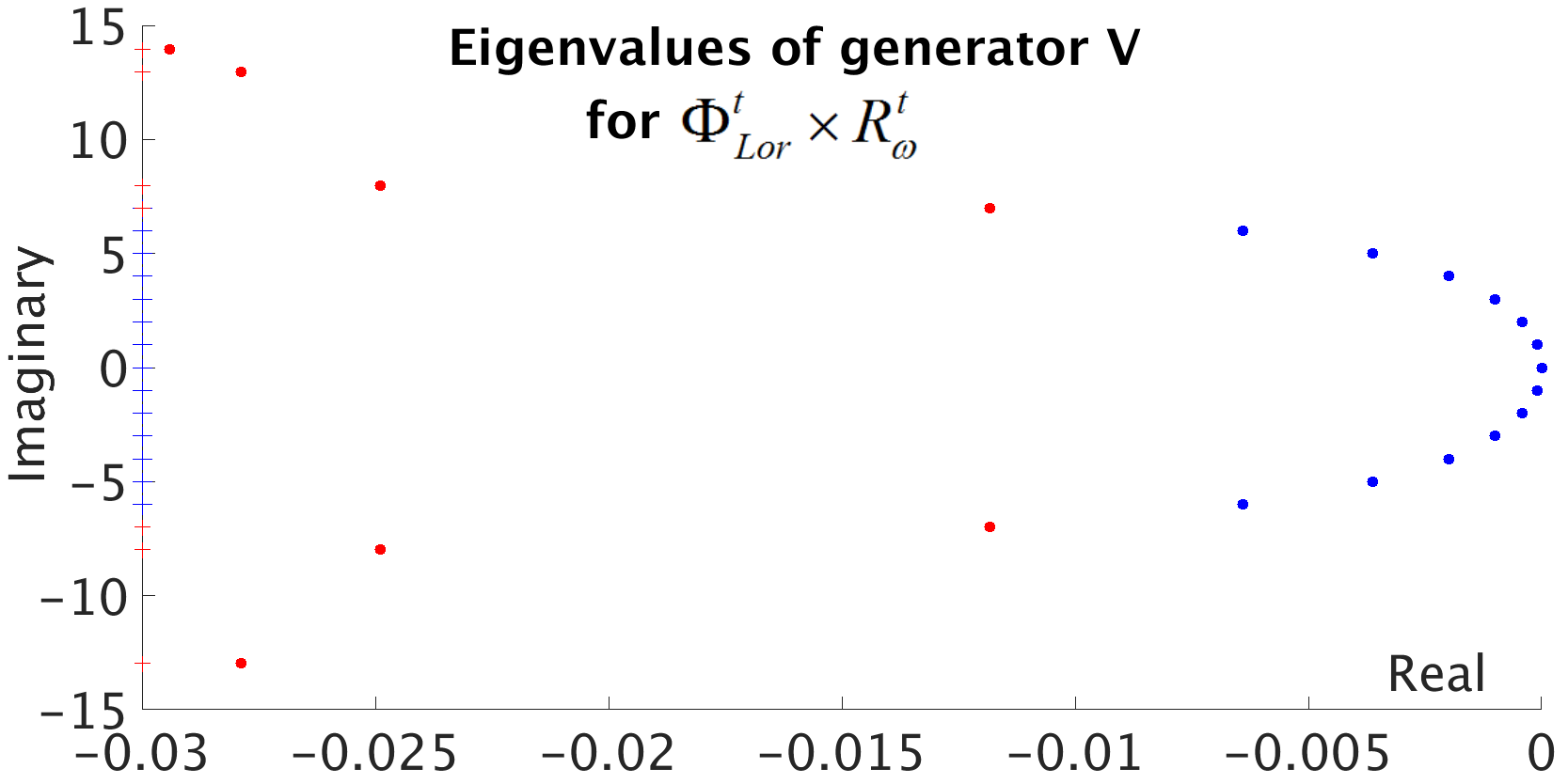}
\caption{Galerkin approximations $ \gamma_j $ of the eigenvalues of the regularized generator $ L_\theta $ for the torus-based system \eqref{eqn:4D_additive} (left) and the L63-based system \eqref{eqn:l63_skew_nonlinear} (right). The numerical eigenvalues were obtained through the data-driven variational eigenvalue problem in definition~\ref{defDataDrivenGalerkin} with a spectral order parameter $ m = 50$. Each plot shows the first 20 eigenvalues corresponding to eigenfunctions of increasing Dirichlet energy, with the first 14 plotted in blue and the remaining 5 in red. Dashes on the imaginary axes indicate the imaginary parts of the eigenvalues. The intervals between the blue-colored dashes are to a good approximation equal to 1, in agreement with the exact Koopman eigenvalues of these systems.}
\label{fig:eigenvalues}
\end{figure}

\subsection{Discussion}

The examples presented in Sections~\ref{secMixedSpec} and~\ref{secResults} are Cartesian products of weak mixing and quasiperiodic flows, with their phase variables combined through some observation map. We begin with some observations about our kernel method applied to Cartesian products. 

\paragraph{Cartesian products.} Let $(X,\Phi_X^t,\mu_X)$ and $(Y,\Phi_Y^t,\mu_Y)$ be two ergodic flows on compact metric spaces with purely continuous and pure-point spectra, respectively. We are interested in the measure-preserving mixed spectrum dynamical system $(X\times Y,\Phi_X^t\times\Phi_Y^t ,\mu_X\times\mu_Y)$. It is well known that the space $L^2(X\times Y,\mu_X\times\mu_Y)$ is densely spanned by products of the form \{$f\otimes g$ : $f\in L^2(X,\mu_X)$, $g\in L^2(Y,\mu_Y)$\}. Recall that the observation map $F$ is the basis of our construction of all our data-driven operators. 
Corollary~\ref{cor:observability} below is a direct consequence of Proposition~\ref{prop:observability}, and gives an ``observability'' condition that must be fulfilled by the observation map $ F $ in order for the methods presented here to yield non-trivial results. 
\begin{cor}\label{cor:observability}
Let $(X,\Phi_X^t,\mu_X)$ and $(Y,\Phi_Y^t,\mu_Y)$ be as described above, and $F\in L^2(X\times Y,\mu_X\times\mu_Y)$ be the sum $F=\sum_{n=1}^{\infty} f_n\otimes g_n$. Then, $F_{\Disc} = \sum_{n=1}^{\infty} E(f_n) g_n$, where $E(f_n) = \int_X f_n \, d\mu_X$. Hence, a necessary and sufficient condition that $P$ is not trivial is that $E(f_n) \neq 0$ for at least one $n\in\num$.
\end{cor}

\paragraph{Kernels with a small number of delays.} An implicit assumption in the approximation of the operator $P$ in \eqref{eqn:def_KHP_Q} by the operator $P_{Q}$ in \eqref{eqn:def_KHP_Q} with finitely many delays $Q$, is that $Q$ is large-enough for the asymptotic analysis of Lemma \ref{lem:cmmt_Koopman} to hold. When $Q$ is small, $d_{Q}$ is closer to a proper metric and therefore, the entries $K_{ij}=\exp(-d_{Q}(x_i,x_j)^2/\epsilon)$ of the kernel matrix $\boldsymbol K $ decay rapidly away from the diagonal $i=j$. Then $K_{ij}$ is close to a diagonal matrix, and $P_{i,j}$ is close to the identity matrix. On the other hand, for $Q$ large, $d_{Q}$ becomes a pseudo-metric and $P_{i,j}$ is not necessarily close to a diagonal matrix. Figure \ref{fig:lambda_compare_mixed_spectra} shows how the Koopman eigenvalues computed for the two examples from \eqref{eqn:4D_additive} and \eqref{eqn:l63_skew_nonlinear} cluster near $1$ for $Q=1$ and decay more rapidly for $Q = 2000$.

\begin{figure}
\centering
\includegraphics[width=.48\textwidth]{\Path 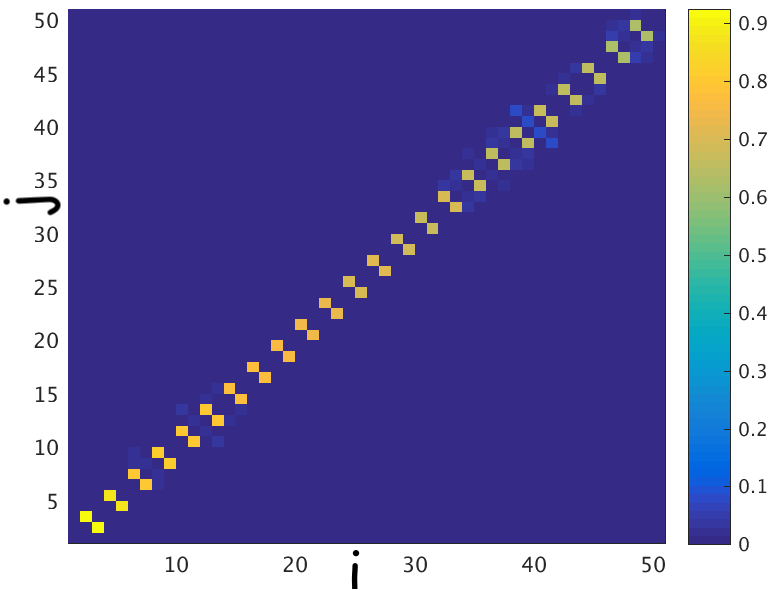}
\includegraphics[width=.48\textwidth]{\Path 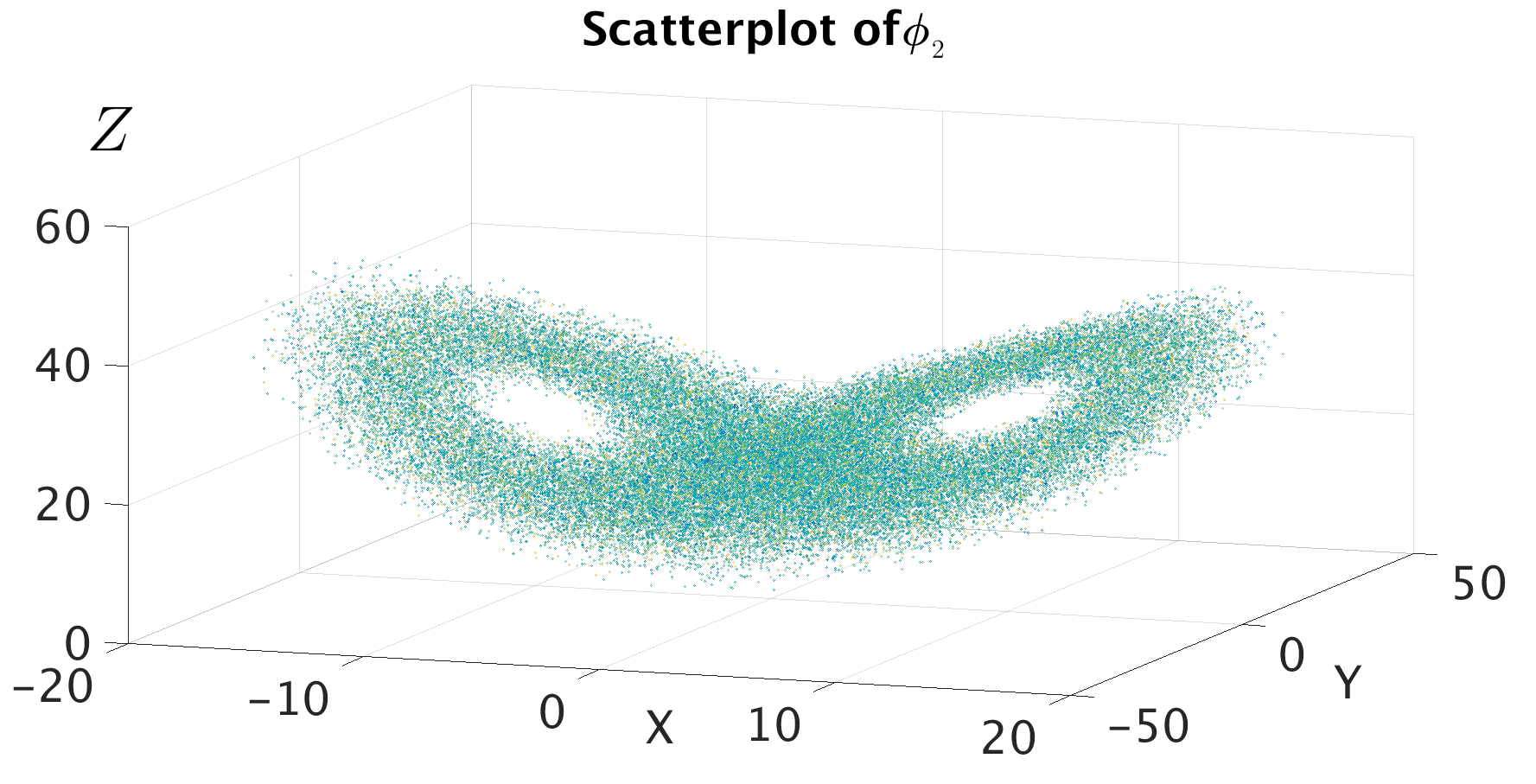}

\includegraphics[width=.48\textwidth]{\Path 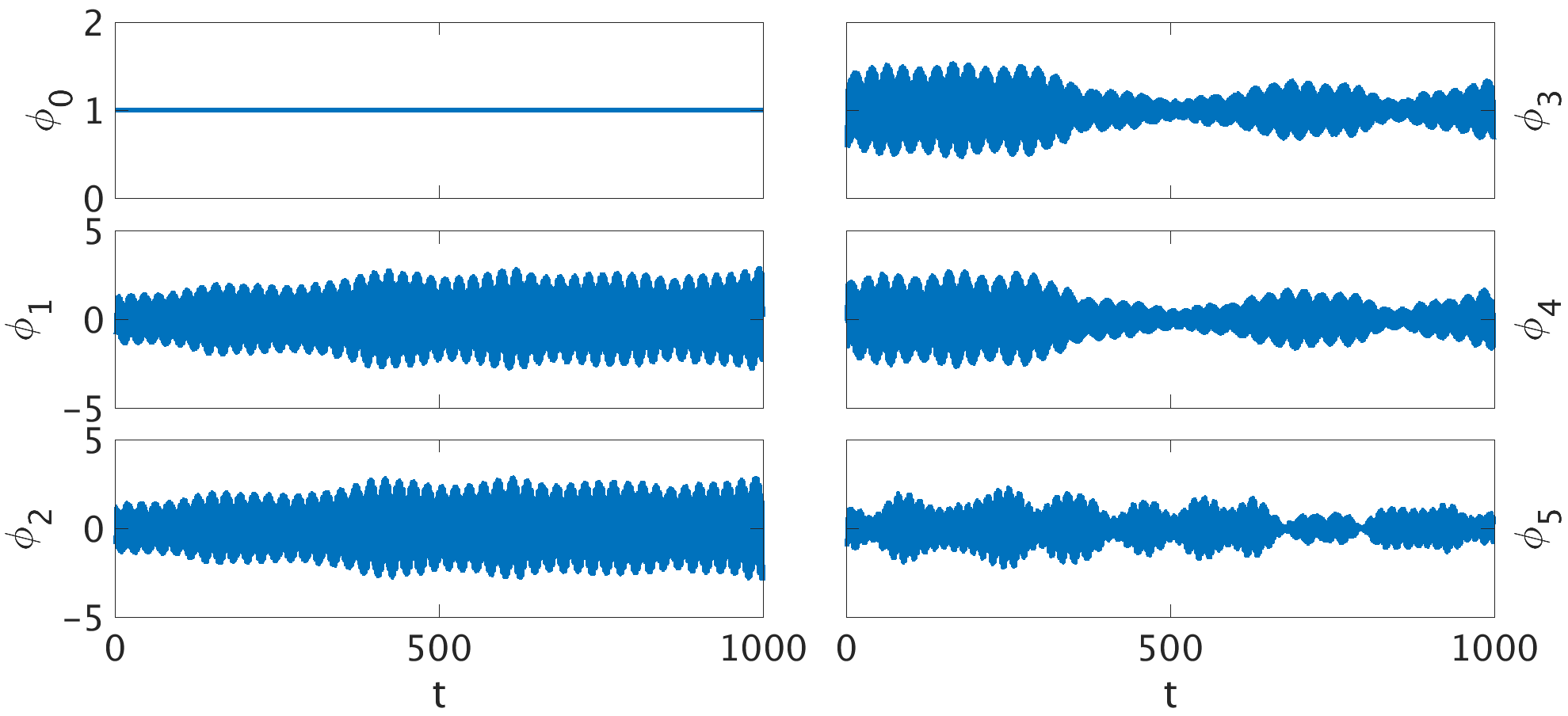}
\includegraphics[width=.48\textwidth]{\Path 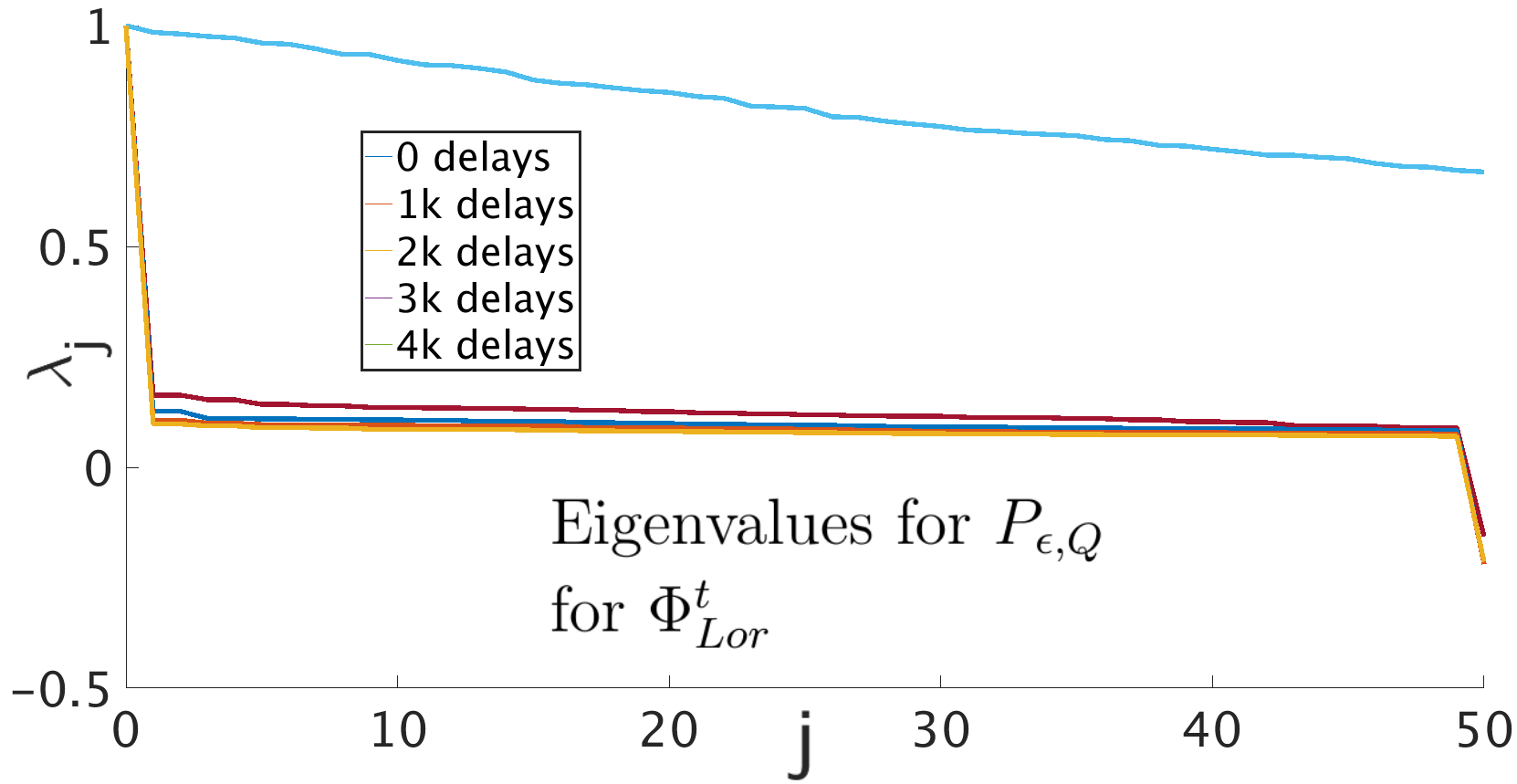}
\caption[l63]{Eigenvalues $ \hat \lambda_j $ and eigenfunctions $ \hat \phi_j $ of $P_{Q,N}$, and \blue{absolute values of a} matrix representation of the generator of the L63 system in~\eqref{eqn:l63} obtained with $Q=4000$ delays. The generator of this system has purely continuous spectrum and a trivial eigenvalue at 0. As a result, according to Theorem~\ref{thm:B}, as $Q\to\infty$ all $ \hat \lambda_j \neq 1 $ converge to $ 0$. This behavior can be seen in the bottom-right panel, where the $ \hat \lambda_j $ not equal to 1 are seen clustered around a small value $\approx 0.1$. Moreover, the time series of the $\hat\phi_j$, shown in the bottom-left panel, are manifestly non-periodic since they fail to converge to Koopman eigenfunctions. As illustrated by the phase space plot of $\hat \phi_2$ in the top-right panel, the leading eigenfunctions have a highly rough geometrical structure on the Lorentz attractor. The top left panel shows a matrix representation of the generator $V$ with respect to the $ \{ \hat \phi_j \} $ basis. Remarkably, this matrix is very nearly bi-diagonal, yet we do not have a theoretical result justifying this behavior.}
\label{fig:l63}
\end{figure}

\paragraph{Systems with purely continuous spectra.} An important assumption of our kernel-based method is that the dynamics has Koopman eigenvalues, i.e., $\Disc$ contains non-constant functions. This underlies the ability of our regularized operator $L_{\theta}$ in~\eqref{eqn:regularize} to be a suitable substitute of $V$ (Theorem~\ref{thm:C}). In fact, by Theorem~\ref{thm:B}, in the limit of infinitely many delays $Q\to\infty$, if $\Disc $ only contains constant functions, then the kernels $k_{Q}$, $p_{Q}$ converge to $ 0$ (in the $L^2$ sense). However, when using finitely many delays, $k_{Q}\neq 0$, and correspondingly $p$ obtained by normalization of $k_{\infty}$ is not close to $0$. It is not currently understood how the operator $P_{Q}$ should behave for purely continuous spectrum systems (i.e., $\Disc=\spn\{1_X\}$) and $Q<\infty$. 

\blue{One of the consequences of Theorem \ref{thm:B}(ii) is that in the limit of $Q\to\infty$, the continuous spectrum subspace $\mathcal{D}^\perp$ is annihilated by the integral operator $P_\infty$, thus rendering this operator ineffective for studying or reconstructing the mixing component of the dynamics. In particular, for weak-mixing systems, $P_\infty$ should have all but one of its eigenvalues to equal to zero.} Numerical results shown in Fig.~\ref{fig:l63} indicate that the finite-rank, data-driven operator $ P_{Q,N}$ for the L63 system still has nonzero eigenvalues strictly less than 1, but these eigenvalues are clustered around a small value ($\hat \lambda_j \approx 0.1$). This behavior is in agreement with Theorem~\ref{thm:B}, according to which all the eigenvalues of $P_{Q}$ other than $1$ should converge to zero as $ Q \to \infty $. Note that the matrix representation of $V$ (also shown in Fig.~\ref{fig:l63}) is still skew-symmetric to a good approximation, since $V$ is a skew-symmetric operator. Intriguingly, the matrix has a $2\times 2$ block-diagonal form, despite $V$ having no eigenfunctions. This form of the generator matrix has some aspects in common with the recent results of Brunton et al.\ \cite{BruntonEtAl17}, who obtained a bi-diagonal matrix representation of the L63 generator in a data-driven basis from Hankel matrix analysis. In Fig.~\ref{fig:l63}, the lack of Koopman eigenfunctions is evident from the time-series plots of the numerical eigenfunctions $\hat \phi_j$, which are clearly non-periodic. Moreover, a phase space plot of $\hat \phi_2$ illustrates that it is a highly rough function on the Lorenz attractor. 

\blue{In light of the above, the results established in this work have implications for delay-embedding techniques, as they point to a tradeoff between reconstruction of the system's state space topology in delay embedding space (favored by large numbers of delays) and the ability of operators for data analysis, such as $P_Q$, to adequately represent the mixing component of the dynamics. Nevertheless, the ability to consistently approximate the quasiperiodic dynamics through Koopman eigenfunctions is still useful, as it allows identification and efficient modeling (e.g., via~\eqref{eqClosedForm} of observables with high predictability). At the very least, the ``negative'' result described above provides a reference point that may aid the design of delay-embedding methodologies aiming to reconstruct the full structure of the dynamics.} One of the goals of our future work is to investigate the behavior of the techniques presented here away from the asymptotic limit $ Q \to \infty $ in the presence of a continuous spectrum. 

\paragraph{Acknowledgements} Dimitrios Giannakis received support from ONR YIP grant N00014-16-1-2649, NSF grant DMS-1521775, and DARPA grant HR0011-16-C-0116. Suddhasattwa Das is supported as a postdoctoral research fellow from the first grant. The authors are grateful to L S Young for her suggestions.

\section*{References}
\bibliographystyle{unsrt}
\bibliography{References}

\begin{thebibliography}{10}
\providecommand{\url}[1]{{#1}}
\providecommand{\urlprefix}{URL }
\expandafter\ifx\csname urlstyle\endcsname\relax
  \providecommand{\doi}[1]{DOI~\discretionary{}{}{}#1}\else
  \providecommand{\doi}{DOI~\discretionary{}{}{}\begingroup
  \urlstyle{rm}\Url}\fi

\bibitem{ALL2001}
Ahues, M., Largillier, A., Limaye, B.: Spectral computations for bounded
  operators.
\newblock Chapman and Hall/CRC (2001).
\newblock \urlprefix\url{https://www.taylorfrancis.com/books/9781420035827}

\bibitem{ArbabiMezic16}
Arbabi, H., Mezi\'c, I.: Ergodic theory, dynamic mode decomposition and
  computation of spectral properties of the {K}oopman operator (2016)

\bibitem{AubryEtAl91}
Aubry, N., Guyonnet, R., Lima, R.: Spatiotemporal analysis of complex signals:
  Theory and applications.
\newblock J. Stat. Phys. \textbf{64}, 683--739 (1991).
\newblock \doi{10.1007/bf01048312}

\bibitem{BabuskaOsborn91}
Babu\v{s}ka, I., Osborn, J.: Eigenvalue Problems, \emph{Handbook of Numerical
  Analysis}, vol.~2.
\newblock North Holland, Amsterdam (1991)

\bibitem{BelkinNiyogi03}
Belkin, M., Niyogi, P.: Laplacian eigenmaps for dimensionality reduction and
  data representation.
\newblock Neural Comput. \textbf{15}, 1373--1396 (2003).
\newblock \doi{10.1162/089976603321780317}

\bibitem{BelkinNiyogi2007}
Belkin, M., Niyogi, P.: Convergence of laplacian eigenmaps.
\newblock In: Advances in Neural Information Processing Systems, pp. 129--136
  (2007).
\newblock
  \urlprefix\url{http://papers.nips.cc/paper/2989-convergence-of-laplacian-eigenmaps.pdf}

\bibitem{BerryEtAl13}
Berry, T., Cressman, R., Greguri\'c-Feren\v{c}ek, Z., Sauer, T.: Time-scale
  separation from diffusion-mapped delay coordinates.
\newblock SIAM J. Appl. Dyn. Sys. \textbf{12}, 618--649 (2013).
\newblock \doi{10.1137/12088183x}

\bibitem{BerryEtAl15}
Berry, T., Giannakis, D., Harlim, J.: Nonparametric forecasting of
  low-dimensional dynamical systems.
\newblock Phys. Rev. E. \textbf{91}, 032,915 (2015).
\newblock \doi{10.1103/PhysRevE.91.032915}

\bibitem{BerryHarlim16}
Berry, T., Harlim, J.: Variable bandwidth diffusion kernels.
\newblock J. Appl. Comput. Harmon. Anal. (2015).
\newblock \doi{10.1016/j.acha.2015.01.001}

\bibitem{BerrySauer16}
Berry, T., Sauer, T.: Consistent manifold representation for topological data
  analysis (2016)

\bibitem{BerrySauer16b}
Berry, T., Sauer, T.: Local kernels and the geometric structure of data.
\newblock Appl. Comput. Harmon. Anal. \textbf{40}, 439--469 (2016).
\newblock \doi{10.1016/j.acha.2015.03.002}

\bibitem{BroomheadKing86}
Broomhead, D.S., King, G.P.: Extracting qualitative dynamics from experimental
  data.
\newblock Phys. D \textbf{20}(2--3), 217--236 (1986).
\newblock \doi{10.1016/0167-2789(86)90031-x}

\bibitem{AnosovKatok1970}
Broomhead, D.S., Lowe, D.: New examples in smooth ergodic theory. {E}rgodic
  diffeomorphisms.
\newblock Trans. Mosc. Math. Soc. \textbf{23}, 1--35 (1970)

\bibitem{BruntonEtAl17}
Brunton, S.L., Brunton, B.W., Proctor, J.L., Kaiser, E., Kutz, J.N.: Chaos as
  an intermittently forced linear system.
\newblock Nat. Commun. \textbf{8}(19) (2017).
\newblock \doi{10.1038/s41467-017-00030-8}

\bibitem{BudisicEtAl12}
Budisi\'c, M., Mohr, R., Mezi\'c, I.: Applied {K}oopmanism.
\newblock Chaos \textbf{22}, 047,510 (2012).
\newblock \doi{10.1063/1.4772195}

\bibitem{CoifmanLafon06}
Coifman, R., Lafon, S.: Diffusion maps.
\newblock Appl. Comput. Harmon. Anal. \textbf{21}, 5–30 (2006).
\newblock \doi{10.1016/j.acha.2006.04.006}

\bibitem{GraphTomog}
Coifman, R., Shkolnisky, Y., Sigworth, F., Singer, A.: Graph {L}aplacian
  tomography from unknown random projections.
\newblock IEEE Trans. Image Process. \textbf{17}(10), 1891–1899 (2008).
\newblock \doi{10.1109/tip.2008.2002305}

\bibitem{ConstantinEtAl1989}
Constantin, P., Foias, C., Nicolaenko, B., T\'emam, R.: Integral Manifolds and
  Inertial Manifolds for Dissipative Partial Differential Equations.
\newblock Springer, New York (1989).
\newblock \doi{10.1007/978-1-4612-3506-4}

\bibitem{ConstantinEtAl2008}
Constantin, P., Kiselev, A., Ryzhik, L., Zlato{\v{s}}, A.: Diffusion and mixing
  in fluid flow.
\newblock Ann. Math. pp. 643--674 (2008).
\newblock \urlprefix\url{https://www.jstor.org/stable/40345422}

\bibitem{DasGiannakis_RKHS_2018}
Das, S., Giannakis, D.: Koopman spectra in reproducing kernel {H}ilbert spaces
  (2018).
\newblock \urlprefix\url{https://arxiv.org/pdf/1801.07799.pdf}

\bibitem{DellnitzEtAl00}
Dellnitz, M., Froyland, G.: On the isolated spectrum of the
  {P}erron-{F}robenius operator.
\newblock Nonlinearity pp. 1171--1188 (2000).
\newblock \doi{10.1088/0951-7715/13/4/310}

\bibitem{DellnitzJunge99}
Dellnitz, M., Junge, O.: On the approximation of complicated dynamical
  behavior.
\newblock SIAM J. Numer. Anal. \textbf{36}, 491 (1999).
\newblock \doi{10.1137/S0036142996313002}

\bibitem{EisnerEtAl15}
Eisner, T., Farkas, B., Haase, M., Nagel, R.: Operator Theoretic Aspects of
  Ergodic Theory, \emph{Graduate Texts in Mathematics}, vol. 272.
\newblock Springer (2015)

\bibitem{Fayad02}
Fayad, B.: Analytic mixing reparametrizations of irrational flows.
\newblock Ergod. Th. \& Dynam. Sys. \textbf{22}, 437--468 (2002).
\newblock \doi{10.1017/s0143385702000214}

\bibitem{EigenIntOp2}
Ferreira, J.C., Menegatto, V.A.: Eigenvalues of integral operators defined by
  smooth positive definite kernels.
\newblock Integral Equations Operator Theory \textbf{64}(1), 61--81 (2009).
\newblock \doi{10.1007/s00020-009-1680-3}

\bibitem{EigenIntOp1}
Ferreira, J.C., Menegatto, V.A.: Eigenvalue decay rates for positive integral
  operators.
\newblock Ann. Mat. Pura Appl. \textbf{192}(6), 1--17 (2013).
\newblock \doi{10.1007/s10231-012-0256-z}

\bibitem{FroylandPadberg09}
Froland, G., Padberg, K.: Almost-invariant sets and invariant manifolds --
  {C}onnecting probabilistic and geometric descriptions of coherent structures
  in flows.
\newblock Phys. D \textbf{238}, 1507--1523 (2009).
\newblock \doi{10.1016/j.physd.2009.03.002}

\bibitem{FroylandEtAl14b}
Froyland, G., Gonz\'alez-Tokman, C., Quas, A.: Detecting isolated spectrum of
  transfer and {K}oopman operators with {F}ourier analytic tools.
\newblock J. Comput. Dyn. \textbf{1}(2), 249--278 (2014).
\newblock \doi{10.3934/jcd.2014.1.249}

\bibitem{FroylandEtAl14}
Froyland, G., Gottwald, G.A., Hammerlindl, A.: A computational method to
  extract macroscopic variables and their dynamics in multiscale systems.
\newblock SIAM J. Appl. Dyn. Sys. \textbf{13}(4), 1816--1846 (2014).
\newblock \doi{10.1137/130943637}

\bibitem{GiannakisEtAl15}
Gannakis, D., Slawinska, J., Zhao, Z.: Spatiotemporal feature extraction with
  data-driven {K}oopman operators.
\newblock J. Mach. Learn. Res. Proceedings \textbf{44}, 103--115 (2015)

\bibitem{Genton01}
Genton, M.C.: Classes of kernels for machine learning: {A} statistics
  perspective.
\newblock J. Mach. Learn. Res. \textbf{2}, 299--312 (2001)

\bibitem{Giannakis15}
Giannakis, D.: Dynamics-adapted cone kernels.
\newblock SIAM J. Appl. Dyn. Sys. \textbf{14}(2), 556--608 (2015).
\newblock \doi{10.1137/140954544}

\bibitem{Giannakis17}
Giannakis, D.: Data-driven spectral decomposition and forecasting of ergodic
  dynamical systems.
\newblock Appl. Comput. Harmon. Anal.  (2017).
\newblock \doi{10.1016/j.acha.2017.09.001}.
\newblock In press

\bibitem{GiannakisDas_tracers}
Giannakis, D., Das, S.: Extraction and prediction of coherent patterns in
  incompressible flows through space-time {K}oopman analysis (2017).
\newblock \emph{Preprint} \url{https://arxiv.org/pdf/1706.06450.pdf}

\bibitem{GiannakisMajda11c}
Giannakis, D., Majda, A.J.: Time series reconstruction via machine learning:
  Revealing decadal variability and intermittency in the {N}orth {P}acific
  sector of a coupled climate model.
\newblock In: Conference on Intelligent Data Understanding 2011. Mountain View,
  California (2011)

\bibitem{GiannakisMajda12a}
Giannakis, D., Majda, A.J.: Nonlinear {L}aplacian spectral analysis for time
  series with intermittency and low-frequency variability.
\newblock Proc. Natl. Acad. Sci. \textbf{109}(7), 2222--2227 (2012).
\newblock \doi{10.1073/pnas.1118984109}

\bibitem{Halmos1956}
Halmos, P.: Lectures on ergodic theory, vol. 142.
\newblock Amer. Math. Soc. (1956)

\bibitem{HolmesEtAl96}
Holmes, P., Lumley, J.L., Berkooz, G.: Turbulence, Coherent Structures,
  Dynamical Systems and Symmetry.
\newblock Cambridge {U}niversity {P}ress, Cambridge (1996)

\bibitem{Koopman31}
Koopman, B.O.: Hamiltonian systems and transformation in {H}ilbert space.
\newblock Proc. Natl. Acad. Sci. \textbf{17}(5), 315--318 (1931)

\bibitem{KordaMezic17}
Korda, M., Mezi\'c, I.: On convergence of extended dynamic mode decomposition
  to the koopman operator (2017)

\bibitem{KordaEtAl17}
Korda, M., Putinar, M., Mezi\'c, I.: Data-driven spectral analysis of the
  {K}oopman operator (2017)

\bibitem{Krengel85}
Krengel, U.: Ergodic theorems, vol.~6.
\newblock Walter de Gruyter (1985)

\bibitem{LawEtAl14}
Law, K., Shukla, A., Stuart, A.M.: Analysis of the {3DVAR} filter for the
  partially observed {L}orenz'63 model.
\newblock Discrete Contin. Dyn. Syst. \textbf{34}(3), 1061--10,178 (2013).
\newblock \doi{doi:10.3934/dcds.2014.34.1061}

\bibitem{LianEtAl16}
Lian, Z., Liu, P., Lu, K.: {{SRB}} measures for a class of partially hyperbolic
  attractors in {H}ilbert spaces.
\newblock J. Differential Equ. \textbf{261}, 1532--1603 (2016).
\newblock \doi{10.1016/j.jde.2016.04.006}

\bibitem{Lorenz63}
Lorenz, E.N.: Deterministic nonperiodic flow.
\newblock J. Atmos. Sci. \textbf{20}, 130--141 (1963)

\bibitem{LuEtAl13}
Lu, K., Wang, Q., Young, L.S.: Strange attractors for periodically forced
  parabolic equations.
\newblock Memoirs of the American Mathematical Society \textbf{224}(1054),
  1--85 (2013).
\newblock \doi{10.1090/S0065-9266-2012-00669-1}

\bibitem{VonLuxburgEtAl08}
von Luxburg, U., Belkin, M., Bousquet, O.: Consistency of spectral clustering.
\newblock Ann. Stat. \textbf{26}(2), 555--586 (2008).
\newblock \doi{10.1214/009053607000000640}

\bibitem{LuzzattoEtAl05}
Luzzatto, S., Melbourne, I., Paccaut, F.: The {L}orenz attractor is mixing.
\newblock Comm. Math. Phys. \textbf{260}(2), 393--401 (2005)

\bibitem{LorentzFract}
McGuinness, M.J.: The fractal dimension of the {L}orenz attractor.
\newblock Philos. Trans. R. Soc. Lond. Ser. A Math. Phys. Eng. Sci.
  \textbf{262}, 413--458 (1968).
\newblock \doi{10.1098/rsta.1968.0001}

\bibitem{Mezic05}
Mezi\'c, I.: Spectral properties of dynamical systems, model reduction and
  decompositions.
\newblock Nonlinear Dyn. \textbf{41}, 309--325 (2005).
\newblock \doi{10.1007/s11071-005-2824-x}

\bibitem{FourierDyn}
Mezi\'c, I.: Spectral properties of dynamical systems, model reduction and
  decompositions.
\newblock Nonlinear Dyn. \textbf{41}, 309--325 (2005).
\newblock \doi{10.1007/s11071-005-2824-x}

\bibitem{MezicBanaszuk04}
Mezi\'c, I., Banaszuk, A.: Comparison of systems with complex behavior.
\newblock Phys. D. \textbf{197}, 101--133 (2004).
\newblock \doi{10.1016/j.physd.2004.06.015}

\bibitem{Nadkarni}
Nadkarni, M.G.: The spectral theorem for unitary operators.
\newblock Springer Science and Business Media (1998)

\bibitem{PackardEtAl80}
Packard, N.H., et~al.: Geometry from a time series.
\newblock Phys. Rev. Lett. \textbf{45}, 712--716 (1980).
\newblock \doi{10.1103/physrevlett.45.712}

\bibitem{RowleyEtAl09}
Rowley, C.W., Mezi\'c, I., Bagheri, S., Schlatter, P., Henningson, D.S.:
  Spectral analysis of nonlinear flows.
\newblock J. Fluid Mech. \textbf{641}, 115--127 (2009).
\newblock \doi{10.1017/s0022112009992059}

\bibitem{SauerEtAl91}
Sauer, T., Yorke, J.A., Casdagli, M.: Embedology.
\newblock J. Stat. Phys. \textbf{65}(3--4), 579--616 (1991).
\newblock \doi{10.1007/bf01053745}

\bibitem{Schmid10}
Schmid, P.J.: Dynamic mode decomposition of numerical and experimental data.
\newblock J. Fluid Mech. \textbf{656}, 5--28 (2010).
\newblock \doi{10.1017/S0022112010001217}

\bibitem{SchmidSesterhenn08}
Schmid, P.J., Sesterhenn, J.L.: Dynamic mode decomposition of numerical and
  experimental data.
\newblock In: Bull. Amer. Phys. Soc., 61st APS meeting, p. 208. San Antonio
  (2008)

\bibitem{Kernel1}
Scholkopf, B., Smola, A., Mu, K.: Nonlinear component analysis as a kernel
  eigenvalue problem.
\newblock Neural Comput. \textbf{10}, 1299–1319 (1998).
\newblock \doi{10.1162/089976698300017467}

\bibitem{SlawinskaGiannakis17}
Slawinska, J., Giannakis, D.: Indo-{P}acific variability on seasonal to
  multidecadal time scales. {P}art {I}: Intrinsic {SST} modes in models and
  observations.
\newblock J. Climate \textbf{30}, 5265--5294 (2017).
\newblock \doi{10.1175/JCLI-D-16-0176.1}

\bibitem{Stone1932}
Stone, M.H.: On one-parameter unitary groups in {H}ilbert space.
\newblock Ann. Math. \textbf{33}, 643–648 (1932).
\newblock \doi{10.2307/1968538, JSTOR 1968538}

\bibitem{TuEtAl14}
Tu, J.H., Rowley, C.W., Lucthenburg, C.M., Brunton, S.L., Kutz, J.N.: On
  dynamic mode decomposition: {T}heory and applications.
\newblock J. Comput. Dyn. \textbf{1}(2), 391--421 (2014).
\newblock \doi{10.3934/jcd.2014.1.391}

\bibitem{Tucker99}
Tucker, W.: The {L}orenz attractor exists.
\newblock C. R. Acad. Sci. Paris, Ser. I \textbf{328}, 1197--1202 (1999)

\bibitem{wellner2013weak}
van~der Vaart, A., Wellner, J.: Empirical processes.
\newblock Springer-Verlag New York (1996).
\newblock \doi{10.1007/978-1-4757-2545-2}

\bibitem{VautardGhil89}
Vautard, R., Ghil, M.: Singular spectrum analysis in nonlinear dynamics, with
  applications to paleoclimatic time series.
\newblock Phys. D \textbf{35}, 395--424 (1989).
\newblock \doi{10.1016/0167-2789(89)90077-8}

\bibitem{WangEtAl17}
Wang, C., Deser, C., Yu, J.Y., DiNezio, P., Clement, A.: El {N}i{\~{n}}o and
  {S}outhern {O}scillation ({ENSO}): {A} review.
\newblock In: P.W. Glynn, D.P. Manzello, I.C. Enoch (eds.) Coral Reefs of the
  Eastern Tropical Pacific: Persistence and Loss in a Dynamic Environment,
  \emph{Coral Reefs of the World}, vol.~8, pp. 85--106. Springer Netherlands,
  Dordrecht (2017).
\newblock \doi{10.1007/978-94-017-7499-4_4}

\bibitem{WilliamsEtAl15}
Williams, M.O., Kevrekidis, I.G., Rowley, C.W.: A data-driven approximation of
  the {K}oopman operator: Extending dynamic mode decomposition.
\newblock J. Nonlinear Sci. (2015).
\newblock \doi{10.1007/s00332-015-9258-5}

\bibitem{SRB_young}
Young, L.S.: What are {{SRB}} measures, and which dynamical systems have them?
\newblock J. Stat. Phys. \textbf{108}, 733--754 (2002).
\newblock \doi{10.1023/A:1019762724717}

\bibitem{ZelnikManorPerona04}
Zelnik-Manor, L., Perona, P.: Self-tuning spectral clustering.
\newblock In: Advances in Neural Information Processing Systems, vol.~17, pp.
  1601--1608 (2004)

\end{thebibliography}
\pagebreak 
\end{document}